\documentclass[12pt,a4paper,reqno]{amsart}
\usepackage{amsmath}
\usepackage{amsthm}
\usepackage{amsfonts}
\usepackage{amssymb}
\usepackage{color}
\usepackage{mathrsfs}
\usepackage{graphicx}

\numberwithin{equation}{subsection}

\theoremstyle{plain}
\newtheorem{thm}{Theorem}[subsection]

\theoremstyle{plain}
\newtheorem{lem}{Lemma}[subsection]

\theoremstyle{plain}
\newtheorem*{propertya}{Property A}

\theoremstyle{plain}
\newtheorem*{propertyb}{Property B}

\theoremstyle{plain}
\newtheorem*{theorema}{Theorem A}

\theoremstyle{plain}
\newtheorem*{theoremb}{Theorem B}

\theoremstyle{plain}
\newtheorem*{cor641}{Corollary of Theorem~\ref{thm6.4.1}}

\theoremstyle{plain}
\newtheorem*{opone}{Open Problem 1}

\theoremstyle{plain}

\theoremstyle{plain}
\newtheorem*{opthree}{Open Problem 3}

\theoremstyle{plain}
\newtheorem*{opfour}{Open Problem 4}

\theoremstyle{plain}

\theoremstyle{plain}

\theoremstyle{plain}

\theoremstyle{remark}
\newtheorem*{remark}{Remark}

\theoremstyle{remark}

\theoremstyle{remark}

\theoremstyle{remark}

\theoremstyle{remark}

\theoremstyle{remark}

\theoremstyle{remark}

\theoremstyle{definition}
\newtheorem*{magsquare}{Rule for magnification in a polysquare}

\theoremstyle{definition}

\theoremstyle{definition}

\theoremstyle{definition}

\theoremstyle{definition}

\theoremstyle{definition}
\newtheorem*{extension}{Extension rule}

\theoremstyle{definition}
\newtheorem*{replacement}{Replacement rule}

\theoremstyle{definition}
\newtheorem*{fconfig}{The uncountable family of $\bfitf$-configurations}

\theoremstyle{definition}
\newtheorem*{lstrips}{The uncountable family of $\{\bfitL_{\bfiti}\}_{\bfiti\in\bfZ}$-strips}

\theoremstyle{definition}
\newtheorem*{case1}{Case 1}

\theoremstyle{definition}
\newtheorem*{case2}{Case 2}

\theoremstyle{definition}
\newtheorem*{case3}{Case 3}

\theoremstyle{definition}
\newtheorem*{case4}{Case 4}

\def\bfc{\mathbf{c}}

\def\bfn{\mathbf{n}}

\def\bfv{\mathbf{v}}

\def\bfZ{\mathbf{Z}}

\def\bfitf{\boldsymbol{f}}

\def\bfiti{\boldsymbol{i}}

\def\bfitL{\boldsymbol{L}}

\def\bzero{\mathbf{0}}

\def\NE{\mathrm{NE}}
\def\NW{\mathrm{NW}}

\def\eps{\varepsilon}

\def\Rr{\mathbb{R}}

\def\Zz{\mathbb{Z}}

\def\CCC{\mathcal{C}}

\def\FFF{\mathcal{F}}

\def\HHH{\mathcal{H}}

\def\PPP{\mathcal{P}}

\def\TTT{\mathcal{T}}

\def\VVV{\mathcal{V}}

\DeclareMathOperator{\length}{length}

\DeclareMathOperator{\dist}{distance}

\DeclareMathOperator{\Bil}{Bil}

\renewcommand{\le}{\leqslant}
\renewcommand{\ge}{\geqslant}

\def\nrightarrow{{+\hspace{-11pt}\rightarrow}}
\def\nuparrow{{\hspace{-0.5pt}{\scriptstyle-}\hspace{-8.5pt}\uparrow}}

\title[Non-integrable systems (III)]
{Quantitative behavior\\
of non-integrable systems (III)}

\author[Beck]{J. Beck}

\address{Department of Mathematics, Rutgers University, Hill Center for the Mathematical Sciences,
Piscataway NJ 08854, USA}

\email{jbeck@math.rutgers.edu}

\author[Chen]{W.W.L. Chen}

\address{Department of Mathematics and Statistics, Macquarie University,
Sydney NSW 2109, Australia}

\email{william.chen@mq.edu.au}

\author[Yang]{Y. Yang}

\address{Department of Mathematics, Rutgers University, Hill Center for the Mathematical Sciences,
Piscataway NJ 08854, USA}

\email{yy458@math.rutgers.edu}

\keywords{geodesics, billiards, time-quantitative density}

\subjclass[2010]{11K38, 37E35}

\begin{document}

\begin{abstract}
The main purpose of part (III) is to give explicit geodesics and billiard orbits in polysquares that exhibit time-quantitative density.
In many instances, we can even establish a best possible form of time-quantitative density called \textit{superdensity}.

We also study infinite flat dynamical systems, both periodic and aperiodic, which include billiards in infinite polysquare regions.
In particular, we can prove time-quantitative density even for aperiodic systems.
In terms of optics the billiard case is equivalent to the result that an explicit single ray of light can essentially illuminate a whole infinite polysquare region with reflecting boundary acting as \textit{mirrors}.
In fact, we show that the same initial direction can work for an uncountable family of such infinite systems.

Some of these infinite systems belong to the class of Ehrenfest wind-tree models, introduced by physicists about $100$ years ago.
Thus we obtain, for the first time, explicit density results about the time evolution of these infinite aperiodic billiard models in physics.
What makes our positive density results in the case of the periodic Ehrenfest wind-tree models particularly interesting is the recent discovery by Fraczek and Ulcigrai~\cite{FU} about these models that for almost every initial direction, the billiard orbit is \textit{not} dense.

To prove density of explicit orbits, we use a non-ergodic method, which is an \textit{eigenvalue-free} version of the shortline method.
The original \textit{eigenvalue-based} version of the shortline method, introduced and developed in \cite{BDY1,BDY2}, enables us to prove time-quantitative equidistribution of orbits.
The reader does not need to be familiar with those long papers.
Here we make a serious effort to keep this paper self-contained. 
\end{abstract}

\maketitle

\thispagestyle{empty}

%%%%%%%%%%
%
% SECTION 6
%
%%%%%%%%%%

\section{Time-quantitative density}\label{sec6}

%%%%%%%%%%
%
% SECTION 6.1
%
%%%%%%%%%%

\subsection{From density to superdensity}\label{sec6.1}

Our goal in part (III) is to prove time-quantitative density of explicit orbits, and, in many cases, even its best possible form called superdensity.

The concept of \textit{time-quantitative} density is simply a means to describing how long it takes for a particle moving with unit speed on an explicit geodesic or a point billiard to enter a given small target set.

Perhaps the reader is wondering: why should we care about density when we already know many uniformity results; for instance, the collection of uniformity results in \cite[Section~2.1]{BDY1} that are proved by ergodic methods.
Well, it is true that uniformity implies density, but uniformity does not imply any form of time-quantitative density, not to mention superdensity.
Note that in general even superuniformity (meaning extremely small poly-logarithmic error term) is not strong enough to imply superdensity.

Time-quantitative uniformity and time-quantitative density represent two (in general incomparable) complementary viewpoints to describing the \textit{evenness} of  an infinite orbit in the undelying space.

Our tool is a new \textit{eigenvalue-free} version of the surplus shortline method, which, for the convenience of the reader, will be developed here from scratch.
A great advantage of this new version is that, unlike the old eigenvalue-based version used in \cite{BDY1,BDY2}, it is flexible enough to work in higher dimensions to prove, for example, the density of $3$-dimensional billiards in cube-tiled solids, or \textit{polycube regions}, as well as to prove the density of billiard orbits in \textit{infinite aperiodic} polysquare regions.

We know very, very little about $3$-dimensional flat dynamical systems, where \textit{flat} refers to locally Euclidean $3$-space, and similarly, we know very, very little about the density of billiard orbits on infinite aperiodic polysquare regions.
So one may say that the most interesting results of part (III) are our density results for $3$-dimensional systems and for infinite aperiodic polysquare regions.
Nevertheless, we start the detailed discussion in the natural/historic order, meaning the case of lower dimension and compact underlying space.

First we study \textit{superdensity}, a best possible form of time-quantitative density.
Superdensity has already been mentioned in \cite[Section~1.1 and Theorem~3.4.1]{BDY1}.
For the convenience of the reader we repeat the formal definition.
The first place to see it is Property~A below, which is a special case.

We begin the discussion with the one-dimensional case, and very briefly recall some basic facts about the density and uniformity of the irrational rotation sequence
$\{j\alpha\}$, $j=1,2,3,\ldots,$ in the unit interval $[0,1)$.
Here $\alpha$ is irrational and $0\le\{x\}<1$ denotes the fractional part of a real number~$x$.

The density of the irrational rotation sequence has been known since the early nineteenth century, through the work of Dirichlet, Chebyshev and Kronecker, \textit{etc.}, and extended to uniform distribution in the first years of the twentieth century by Bohl, Sierpinski and Weyl.
We can clearly assume that $0<\alpha<1$, which has an infinite continued fraction expansion of the form
\begin{equation}\label{eq6.1.1}
\alpha=[a_1,a_2,a_3,\ldots]=\frac{1}{a_1+\frac{1}{a_2+\frac{1}{a_3+\cdots}}},
\end{equation}
with digits, or partial quotients, $a_i\ge1$.
The works of Hardy and Littlewood~\cite{HL1,HL2}, Ostrowski~\cite{O}, Weyl~\cite{W}, \textit{etc.} around 1920 help to clarify the key role played by the continued fraction digits $a_i$ in the quantitative aspects of the distribution of the irrational rotation sequence.
A main result of this classical work is that the sequence $\{j\alpha\}$, $j=1,2,3,\ldots,$ is \textit{most} uniformly distributed in the precise sense that it exhibits logarithmic error, which is the minimum order of magnitude, if and only if the average size of the digits is bounded, formally, if
\begin{equation}\label{eq6.1.2}
\limsup_{n\to\infty}\frac{1}{n}\sum_{i=1}^na_i<\infty.
\end{equation}

An irrational number $\alpha\in[0,1)$ is \textit{badly approximable} if and only if the continued fraction digits are bounded, \textit{i.e.}, there is a constant $C=C(\alpha)$ such that $a_i\le C$ for every digit $a_i$ in \eqref{eq6.1.1}.
For badly approximable numbers the average size of the digits is trivially bounded, \textit{i.e.}, \eqref{eq6.1.2} holds.
Note that every quadratic irrational is badly approximable, since the continued fraction is eventually periodic, a result that goes back to Euler and Lagrange.

Superdensity is closely related to this classical work about uniform distribution.
Indeed, the irrational rotation sequence exhibits superdensity if and only if $\alpha$ is badly approximable.
It means precisely that Property~A and Property~B below are equivalent.

\begin{propertya}
There is an absolute constant $C_1=C_1(\alpha)$ such that for every integer $n\ge1$ and subinterval $I\subset[0,1)$ of length~$1/n$, there exists $1\le j\le C_1n$ such that $\{j\alpha\}\in I$.
\end{propertya}

\begin{propertyb}
The number $\alpha$ is badly approximable, \textit{i.e.}, there exists a constant $C=C(\alpha)$ such that $a_i\le C$ for every digit $a_i$ in \eqref{eq6.1.1}.
\end{propertyb}

It is Property~A that we consider the definition of \textit{superdensity} in the special case of the irrational rotation sequence.

\begin{lem}\label{lem6.1.1}
Property~A and Property~B are equivalent.
\end{lem}

A proof of this can be found in, for instance, Khinchin~\cite[Theorem~26]{K2}.
For the sake of completeness we include here our shorter proof, which has the extra benefit that the reader can compare it to the more complicated proof of Lemma~\ref{lem6.1.2}.

\begin{proof}[Proof of Lemma~\ref{lem6.1.1}]
The proof is an easy exercise by using the theory of continued fractions.
Let $k\ge1$ be any integer.
The initial segment
\begin{displaymath}
[a_1,a_2,\ldots,a_k]=\frac{p_k}{q_k}
\end{displaymath}
of \eqref{eq6.1.1} is a rational number, called the $k$-th convergent of the irrational number~$\alpha$.
Here the numerators $p_k=p_k(\alpha)$ and the denominators $q_k=q_k(\alpha)$ of the convergents of $\alpha$ satisfy the recurrence relations
\begin{equation}\label{eq6.1.3}
p_k=a_kp_{k-1}+p_{k-2},
\quad
q_k=a_kq_{k-1}+q_{k-2},
\quad
p_kq_{k-1}-q_kp_{k-1}=(-1)^k
\end{equation}
for every integer $k\ge2$, together with the initial conditions $p_0=0$, $q_0=1$, $p_1=1$ and $q_1=a_1$.
The $k$-th convergent $p_k/q_k$ gives an excellent rational approximation of~$\alpha$, in the form
\begin{equation}\label{eq6.1.4}
\left\vert\alpha-\frac{p_k}{q_k}\right\vert<\frac{1}{q_kq_{k+1}}.
\end{equation}
The proof of Lemma~\ref{lem6.1.1} is based on \eqref{eq6.1.3} and \eqref{eq6.1.4}, which are well known facts in the theory of continued fractions; see any book on number theory that has a chapter on continued fractions.

First we derive Property~A from Property~B.
For any arbitrary integer $n\ge1$, let $k=k(\alpha;n)$ be the smallest integer such that
\begin{equation}\label{eq6.1.5}
q_k=q_k(\alpha)>3n.
\end{equation}
Let $I\subset[0,1)$ be of length~$1/n$.
By \eqref{eq6.1.5} there exists an integer $1\le\ell\le q_k$ such that $I$ contains both $(\ell-1)/q_k$ and $(\ell+1)/q_k$, with the convention that
$(q_k+1)/q_k$ denotes $1/q_k$.
Multiplying \eqref{eq6.1.4} by a nonzero integer $1\le j\le q_k$, we have
\begin{equation}\label{eq6.1.6}
\left\vert j\alpha-\frac{jp_k}{q_k}\right\vert<\frac{j}{q_kq_{k+1}}<\frac{1}{q_k}.
\end{equation}
From the last equation in \eqref{eq6.1.3}, we see that $p_k$ and $q_k$ are relatively prime, so there exists an integer $1\le j_0\le q_k$ such that
\begin{equation}\label{eq6.1.7}
\left\{\frac{j_0p_k}{q_k}\right\}=\frac{\ell}{q_k}.
\end{equation}
Using \eqref{eq6.1.6} with $j=j_0$, and combining it with \eqref{eq6.1.7}, we obtain that $\{j_0\alpha\}\in I$ for some $1\le j_0\le C_1n$, which proves Property~A.
Indeed, it follows from \eqref{eq6.1.5} and \eqref{eq6.1.3} that
\begin{displaymath}
q_{k-1}\le 3n
\quad\mbox{and}\quad
q_k\le(a_k+1)q_{k-1},
\end{displaymath}
which imply that $j_0\le q_k\le3(C+1)n$, so that $j_0\le C_1n$ if we take $C_1=3(C+1)$.

Next we derive Property~B from Property~A. 
For any positive integer~$k$, consider the interval
\begin{equation}\label{eq6.1.8}
I=\left[\frac{1}{3q_k},\frac{2}{3q_k}\right].
\end{equation}
Multiplying \eqref{eq6.1.4} by a nonzero integer $1\le j\le q_{k+1}/3$, we have
\begin{displaymath}
\left\vert j\alpha-\frac{jp_k}{q_k}\right\vert<\frac{j}{q_kq_{k+1}}\le\frac{1}{3q_k}.
\end{displaymath}
This implies that
\begin{displaymath}
-\frac{1}{3q_k}<j\alpha-\frac{jp_k}{q_k}<\frac{1}{3q_k},
\quad\mbox{or}\quad
\frac{3jp_k-1}{3q_k}<j\alpha<\frac{3jp_k+1}{3q_k}.
\end{displaymath}
Write $x=jp_k-[j\alpha]q_k$.
Then clearly
\begin{displaymath}
\frac{3x-1}{3q_k}<\{j\alpha\}<\frac{3x+1}{3q_k}.
\end{displaymath}
Naturally we must have $3x+1>0$, and so $x$ is a non-negative integer.
If $x=0$, then $\{j\alpha\}<1/3q_k$.
If $x\ge1$, then $\{j\alpha\}>2/3q_k$.
Thus it follows that
\begin{equation}\label{eq6.1.9}
\{j\alpha\}\not\in I
\quad\mbox{for every }
1\le j<\frac{q_{k+1}}{3}.
\end{equation}
Note from \eqref{eq6.1.8} that $\vert I\vert=1/n$ with $n=3q_k$.
If Property~A holds, then there exists $1\le j_0\le C_1n$ such that $\{j_0\alpha\}\in I$.
Combining this with \eqref{eq6.1.9}, we have
\begin{displaymath}
\frac{q_{k+1}}{3}\le j_0\le C_1n=3C_1q_k,
\end{displaymath}
and since $a_{k+1}q_k<q_{k+1}$, we obtain
\begin{displaymath}
\frac{a_{k+1}q_k}{3}<\frac{q_{k+1}}{3}\le j_0\le 3C_1q_k,
\end{displaymath}
which implies $a_{k+1}<9C_1$.
This proves Property~B with the choice $C=9C_1$.
\end{proof}

Superdensity of the \textit{discrete} irrational rotation sequence with badly approximable $\alpha$ immediately implies superdensity of the \textit{continuous} torus lines with slope $\alpha$ in the unit square.
The standard trick is \textit{discretization}.
Discretization simply means that we look at the points where the torus line hits the sides of the square.
This reduces the problem of uniformity in the $2$-dimensional case to the $1$-dimensional case.

More precisely, discrete superdensity implies via discretization that an infinite torus half-line of badly approximable slope $\alpha$ in the unit square $[0,1)^2$ has the following remarkable property.
There is an absolute constant $C_2=C_2(\alpha)$ such that for every integer $n\ge1$ and for every point $P\in[0,1)^2$ in the unit square, the initial segment of length $C_2n$ of this torus half-line gets $(1/n)$-close to~$P$.
This is what we call the \textit{superdensity of the torus line} in the unit square.

In higher dimensions we have Kronecker's classical theorem concerning the density of the torus line flow in the unit cube $[0,1)^d$, where $d\ge2$ is arbitrary.
Suppose that $\bfv=(v_1,\ldots,v_d)\in\Rr^d$ is a vector such that its coordinates are linearly independent over the rational numbers.
Then by Kronecker's theorem any infinite torus half-line of direction $\bfv$ is dense in the unit cube $[0,1)^d$.
And we also have the converse, that density implies linear independence of the coordinates of the direction vector.

It is straightforward to define superdensity of the torus line in a cube in any dimension~$d$.
An infinite torus half-line of direction vector $\bfv\in\Rr^d$ is \textit{superdense} in the unit cube $[0,1)^d$ if there is an absolute constant $C_3=C_3(\bfv)$ such that for every integer $n\ge1$ and for every point $P\in[0,1)^d$ in the unit cube, the initial segment of length $C_3n^{d-1}$ of this torus half-line gets $(1/n)$-close to~$P$.

Superdensity represents a best possible quantitative form of density in both the discrete and the continuous case.
For simplicity we just show it in the continuous case.
We prove that the polynomial order of magnitude of the length $n^{d-1}$ in the variable $n$ is necessary to get $(1/n)$-close to every point.
For simplicity we choose an integer $n\ge2$, and consider the usual decomposition of the unit cube $[0,1)^d$ into $n^d$ congruent subcubes.
Next we decompose each one of these subcubes with side length $1/n$ into $3^d$ congruent smaller cubes, and refer to the particular cube of side length $1/3n$ in the middle as a \textit{center cube}.
The distance between any two center cubes is at least~$2/3n$.
If a continuous curve $\CCC$ gets $(1/6n)$-close to every point, then it must visit every center cube.
Since there are $n^d$ center cubes, $\CCC$ must have length at least
\begin{displaymath}
(n^d-1)\frac{2}{3n}=\frac{2}{3}n^{d-1}-o(1),
\end{displaymath}
which gives the desired polynomial order of magnitude~$n^{d-1}$.

Superdensity of a torus line in a square $[0,1)^2$ is completely understood.
The necessary and sufficient condition for superdensity is that the slope is badly approximable.

Badly approximable slopes are not typical, as they form a set of zero Lebesgue measure.
But we cannot call this set \textit{totally negligible} either, since it has positive Hausdorff measure.

For almost every slope~$\alpha$, the torus line exhibits \textit{almost} superdensity.
Here the linear bound $C_3n$ above is replaced by a bound of slightly larger order of magnitude $n(\log n)^{1+\eps}$, where, as usual, $\eps>0$ can be arbitrarily small but fixed, assuming that $n$ is large enough. 
This follows from a classical result of Khinchin~\cite{K1} in diophantine approximation.

The problem of superdensity of a torus line in a cube $[0,1)^d$ with $d\ge3$ is harder.
One reason is that the theory of continued fractions does not seem to extend to higher dimensions, and one has to find an alternative approach.
What works here is the geometry of numbers, which gives rise to some \textit{transference theorems}; see, for instance, Cassels~\cite[Chapter~5]{C}.
Combining a couple of transference theorems it is not difficult to prove the following result, which is basically a weaker form of Lemma~\ref{lem6.1.1} in higher dimensions.
Lemma~\ref{lem6.1.2} below is a one-sided result.
It is a sufficient condition for superdensity in higher dimensions.
It gives infinitely many explicit superdense directions.
It is well possible that it has already been published somewhere, but we have not been able to find it.

To understand Lemma~\ref{lem6.1.2}, the reader needs to be familiar with at least the simplest basic concepts of algebraic number fields.

\begin{lem}[``possibly folklore'']\label{lem6.1.2}
Let $m\ge1$ be an integer, and let $\alpha_1,\ldots,\alpha_m$ be any $m$ numbers in a real algebraic number field of degree $m+1$ such that $1,\alpha_1,\ldots,\alpha_m$ are linearly independent over the rationals.
Write
\begin{displaymath}
\bfv=(1,\alpha_1,\ldots,\alpha_m)\in\Rr^{m+1}.
\end{displaymath}
Then any torus half-line with direction $\bfv$ is superdense in the unit cube $[0,1)^{m+1}$.
\end{lem}

\begin{remark}
Note that the $m+1$ numbers $1,\alpha_1,\ldots,\alpha_m$ must satisfy two conditions.
They must be linearly independent over the rationals, and they must all belong to the same real algebraic number field of degree $m+1$.
For the special case $m=2$, we can take $\alpha_1=2^{1/3}$ and $\alpha_2=4^{1/3}$, but not $\alpha_1=\sqrt{2}$ and $\alpha_2=\sqrt{3}$, as no real cubic number field contains the numbers $1,\sqrt{2},\sqrt{3}$, although they are linearly independent over the rationals.
\end{remark}

\begin{proof}[Proof of Lemma~\ref{lem6.1.2}]
Let $\Vert x\Vert$ denote the distance of a real number $x$ from a nearest integer.

The first step of the proof is to show that there exists a constant $C_4>0$, depending at most on $m$ and $\alpha_1,\ldots,\alpha_m$, such that
\begin{equation}\label{eq6.1.10}
\left\Vert\sum_{i=1}^mn_i\alpha_i\right\Vert\ge\frac{C_4}{(\max_{1\le i\le m}\vert n_i\vert)^m}
\end{equation}
for all nonzero integral vectors $\bfn=(n_1,\ldots,n_m)\in\Zz^m$.
The assertion \eqref{eq6.1.10} will follow from using the concept of \textit{norm} in an algebraic number field.

Since every algebraic number is the ratio of an algebraic integer and a nonzero rational integer, it is  enough to prove \eqref{eq6.1.10} when $\alpha_i$, $1\le i\le m$, are algebraic integers.
Let $n_0$ be the nearest integer to the sum $\sum_{i=1}^mn_i\alpha_i$.
The norm of the algebraic integer $n_0-\sum_{i=1}^mn_i\alpha_i$ is the product
\begin{displaymath}
\prod_{j=0}^{m}\left(n_0-\sum_{i=1}^mn_i\alpha^{(j)}_i\right),
\end{displaymath}
where $\alpha^{(0)}_i=\alpha_i$ and $\alpha^{(j)}_i$, $1\le j\le m$, are the $m$ other algebraic conjugates of~$\alpha_i$.
Since the norm of an algebraic integer is a nonzero rational integer, and so has absolute value at least~$1$, we deduce that
\begin{displaymath}
\left\Vert\sum_{i=1}^mn_i\alpha_i\right\Vert
=\left\vert n_0-\sum_{i=1}^mn_i\alpha_i\right\vert
\ge\frac{1}{\prod_{j=1}^m\vert n_0-\sum_{i=1}^mn_i\alpha^{(j)}_i\vert}
\ge\frac{C_5}{(\max_{1\le i\le m}\vert n_i\vert)^m}
\end{displaymath}
where the constant $C_5>0$ depends at most on $m$ and $\alpha_1,\ldots,\alpha_m$.
The assertion \eqref{eq6.1.10} follows.

The second step of the proof is to use Mahler's transference theorem in the relevant special case; see Mahler~\cite{Ma} or Cassels~\cite[Chapter~5, Theorem~2]{C}.

\begin{theorema}
A necessary and sufficient condition that there is a constant $C'>0$ such that
\begin{displaymath}
\left\Vert\sum_{i=1}^mn_i\alpha_i\right\Vert\left(\max_{1\le i\le m}\vert n_i\vert\right)^m\ge C'
\end{displaymath}
for every $\bfn=(n_1,\ldots,n_m)\in\Zz^m$ with $\bfn\ne\bzero$, is that there is another constant $C''>0$ such that
\begin{displaymath}
\left(\max_{1\le j\le m}\Vert n\alpha_j\Vert\right)^m\vert n\vert\ge C''
\end{displaymath}
for every $n\in\Zz$ with $n\ne0$.
\end{theorema}

The third step of the proof is to apply the following transference result, which is a special case of a theorem of Hlawka~\cite{Hl} about general linear forms; see also Cassels~\cite[Chapter~5, Theorem~6]{C}.

\begin{theoremb}
Let $\alpha_1,\ldots,\alpha_m$ be $m\ge1$ real numbers such that for all $N\ge1$,
\begin{displaymath}
\max_{1\le j\le m}\Vert n\alpha_j\Vert\ge C_0N^{-1/m}
\end{displaymath}
for every integer $1\le n\le N$, where $C_0=C_0(\alpha_1,\ldots,\alpha_m)>0$ is a constant independent of~$N$.
Then for any set of $m$ real numbers $0<\beta_i<1$, $1\le i\le m$, there is an integer $1\le\ell_0\le C^*N$ such that
\begin{displaymath}
\Vert\ell_0\alpha_i-\beta_i\Vert\le C^*N^{-1/m}
\end{displaymath}
for every $1\le i\le m$, where the constant $C^*=C^*(C_0)$ depends only on the value of~$C_0$.
\end{theoremb}

Lemma~\ref{lem6.1.2} now follows from a combination of the inequality \eqref{eq6.1.10} together with Theorem~A and Theorem~B.
\end{proof}

By Lemma~\ref{lem6.1.1}, a torus line in a square is superdense if and only if the slope is a badly approximable number.
The torus line flow in a square, the quintessential integrable flat system, is very \textit{user-friendly} in the sense that it exhibits remarkable stability and predictability.
Indeed, two particles moving on two close parallel torus lines with the same speed remain close forever, and they preserve their distance.
This raises a natural question:
When can we guarantee superdensity in the much harder case of non-integrable flat systems, where parallel orbits split and the long-term behavior becomes unpredictable?

The good news is that it is possible to guarantee infinitely many slopes with superdense geodesics for \textit{every} polysquare translation surface.
It is based on a new version of the shortline method.
We illustrate the basic idea of the proof on the simplest flat polysquare translation surface, the so-called \textit{L-surface}.

The L-surface is a compact closed flat polysquare translation surface with $3$ unit square faces forming the letter L (\textit{L-shape}).
It is obtained by identifying the two horizontal edges~$h_1$, the two horizontal edges~$h_2$, the two vertical edges~$v_1$, and the two vertical edges~$v_2$; see Figure~6.1.1.

\begin{displaymath}
\begin{array}{c}
\includegraphics[scale=0.75]{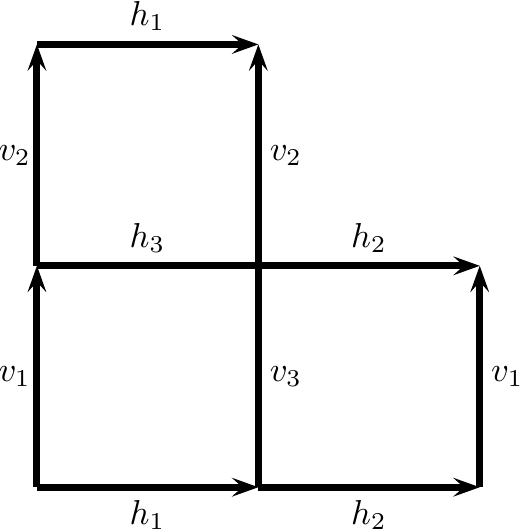}
\vspace{3pt}\\
\mbox{Figure 6.1.1: net of the L-surface with edge identification}
\end{array}
\end{displaymath}

Unlike the cube surface, the L-surface is not the surface of a $3$-dimensional solid, so one may call it \textit{exotic}.
Nevertheless, it is a perfectly legitimate surface with genus~$2$.

A geodesic on the L-surface is basically a generalized torus line on the L-shape, as illustrated in Figure 6.1.2.

\begin{displaymath}
\begin{array}{c}
\includegraphics[scale=0.75]{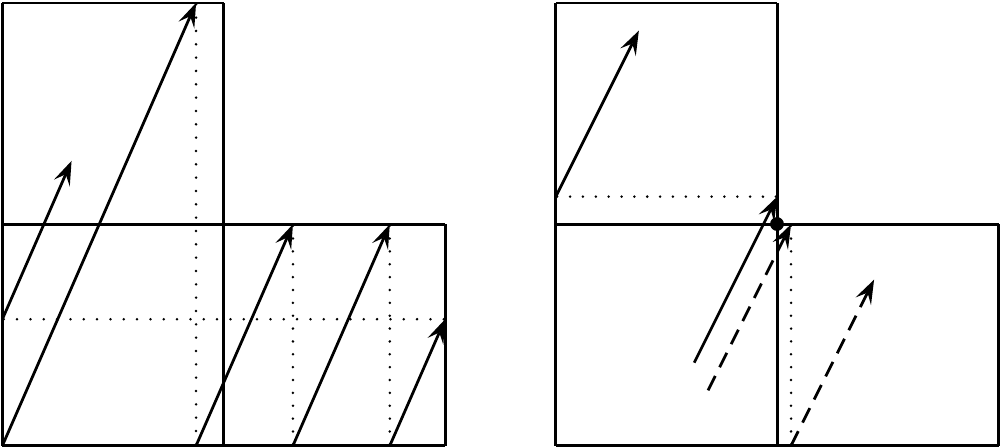}
\vspace{3pt}\\
\mbox{Figure 6.1.2: geodesics on the L-surface and a split singularity}
\end{array}
\end{displaymath}

The L-surface is non-integrable, since it has a split singularity, as demonstrated in the picture on the right in Figure 6.1.2.
Two geodesics close together behave rather differently after getting close to this singularity.

In Section~\ref{sec6.2} we shall prove the following result.

\begin{thm}\label{thm6.1.1}
Let $\alpha>1$ be a badly approximable number with continued fraction
\begin{displaymath}
\alpha=a_0+\frac{1}{a_1+\frac{1}{a_2+\frac{1}{a_3+\cdots}}},
\end{displaymath}
where the digits $a_0,a_1,a_2,a_3,\ldots$ are all positive and \textbf{even}.
Then any half-infinite $1$-directional geodesic with slope $\alpha$ exhibits superdensity on the L-surface, unless the geodesic hits a vertex and becomes undefined.
\end{thm}

\begin{remark}
We should explain at this point part of the reasoning for restricting the continued fraction digits $a_i$ to even integers.
Here we consider a geodesic of slope $\alpha=[a_0;a_1,a_2,a_3,\ldots]$.
Since $a_0\ge2$, we clearly have $\alpha>1$.
We shall call this geodesic \textit{almost vertical}.
The idea is to replace part of this geodesic by part of another geodesic which is \textit{almost horizontal}, meaning that its slope has absolute value less than~$1$.
We elaborate on this below.

Consider the digit $a_0$.
Clearly $\alpha=a_0+\{\alpha\}$, where $\{\alpha\}=[a_1,a_2,a_3,\ldots]$ is the fractional part of~$\alpha$.

Suppose that $a_0=2$, or in general, $a_0$ is even.
It is clear from the picture on the left in Figure 6.1.3 that a geodesic of slope $\alpha$ that starts from the origin, represented by the solid line in the picture, cuts the edge $v_3$ at the point $(1,\{\alpha\})$.
For the part of the geodesic of slope $\alpha$ from the origin to this point, a \textit{shortcut} can be obtained by the geodesic of slope $\{\alpha\}<1$ from the origin to this point, represented by the dashed line in the picture, and the slope of this shortcut is positive.

\begin{displaymath}
\begin{array}{c}
\includegraphics[scale=0.75]{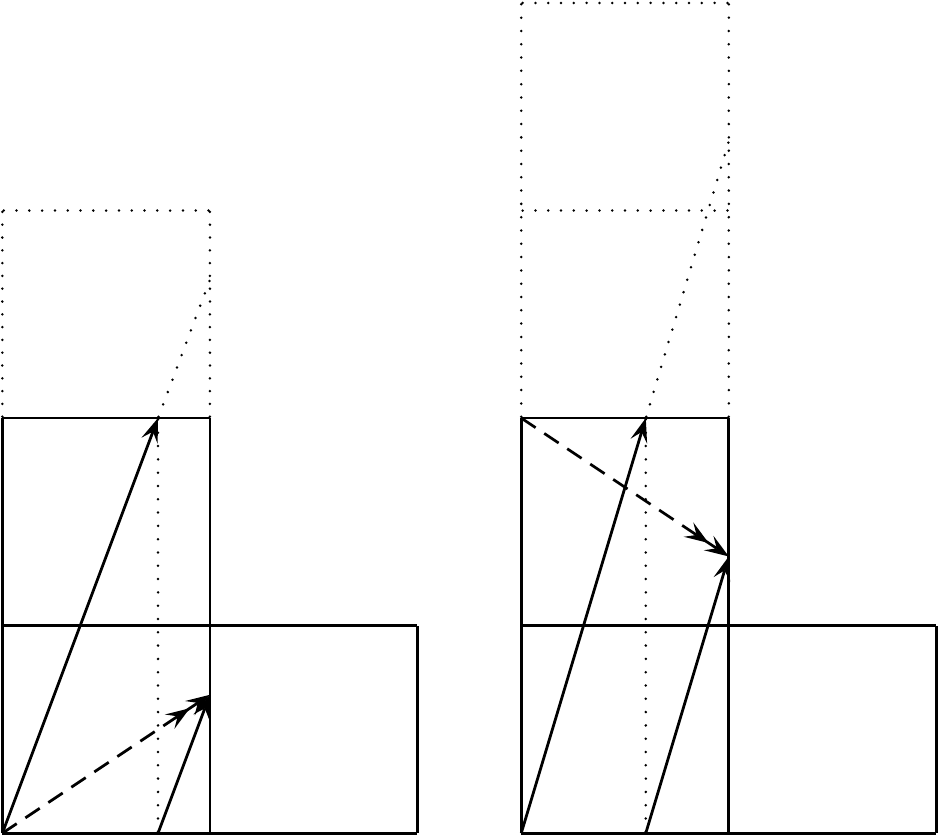}
\vspace{3pt}\\
\mbox{Figure 6.1.3: the cases $a_0=2$ and $a_0=3$}
\end{array}
\end{displaymath}

Suppose that $a_0=3$, or in general, $a_0$ is odd.
It is clear from the picture on the right in Figure 6.1.3 that a geodesic of slope $\alpha$ that starts from the origin, represented by the solid line in the picture, cuts the edge $v_2$ at the point $(1,1+\{\alpha\})$.
For the part of the geodesic of slope $\alpha$ from the origin to this point, a shortcut can be obtained by the geodesic of slope with absolute value less than $1$ from the top left vertex, which is identified with the origin, to this point, represented by the dashed line in the picture, and the slope of this shortcut is negative.

In our initial discussion, we want to avoid geodesics with negative slopes.
As will be clear later, this can be achieved by ensuring that all the continued fraction digits $a_i$ are even.
\end{remark}

The reader who has read \cite{BDY2} is probably wondering why we consider the L-surface here, when Theorem~5.3.1 there already establishes the superdensity of L-lines for all quadratic irrational slopes.
Well, the proof of Theorem~5.3.1 comes to more than 100 pages, and the proof of Theorem~\ref{thm6.1.1} in Section~\ref{sec6.2} is considerably shorter.
This difference clearly shows that the new eigenvalue-free version of the shortline method here is simpler than the eigenvalue-based shortline method developed in~\cite{BDY1,BDY2}.
We shall see later that this new approach is also much more flexible.

Theorem~\ref{thm6.1.1} provides an infinite set of \textit{good} slopes such that the corresponding geodesics are superdense on the L-surface.
From the viewpoint of set theory this set is \textit{large}, since it is uncountable.
From the viewpoint of topology this set is \textit{small}, since it is \textit{nowhere dense} on the unit circle.
In Section~\ref{sec6.4} we shall show how we can extend this set to a larger set of good slopes, which is dense on the unit circle. 
There we shall also generalize Theorem~\ref{thm6.1.1} to all polysquare translation surfaces.

Next we move to the $3$-space; in particular to the class of cube-tiled solids, both finite and infinite.
This is quite interesting because, as far as we know, there is no known density result for non-integrable systems of dimension greater than~$2$.
We elaborate on this.

The first non-trivial result for $2$-dimensional non-integrable flat systems is a result of Katok and Zemlyakov~\cite{KZ} in 1975; see \cite[Theorem~2.1.1]{BDY1}. 
It concerns the density of any infinite geodesic on a \textit{rational surface}, \textit{i.e.}, a surface where every angle on every polygonal face is a rational multiple of~$\pi$.
The proof is a clever application of Poincare's recurrence theorem, but it does not say anything definite about how long it takes for a geodesic to first enter a given test set such as a small circle on a face with radius~$1/n$.

For comparison note that Theorem~\ref{thm6.1.1} is a superdensity result, a strongest form of time-quantitative density, and it \textit{does} tell us how long it takes for a geodesic to enter first a given test set such as a small circle on a face with
radius~$1/n$.

To illustrate how little is known about the density of flat dynamical systems in general, we mention the following humiliatingly long-standing open problem.

\begin{opone}
Let $\TTT$ be an arbitrary right triangle, and consider billiards in~$\TTT$.

\emph{(a)} Does there exist a half-infinite billiard orbit that is dense in~$\TTT$?

\emph{(b)} Does there exist an explicit half-infinite billiard orbit that is dense in~$\TTT$?
Here \textit{explicit} means that we can express the starting point and the initial slope of the orbit in terms of the given data of the triangle~$\TTT$.

\emph{(c)} Does there exist a slope such that every half-infinite billiard orbit with this initial slope is dense in~$\TTT$?

\emph{(d)} Is it true that for almost every real number~$\alpha$, every half-infinite billiard orbit with initial slope $\alpha$ is dense in~$\TTT$?
\end{opone}

What Open Problem~1 really illustrates is the lack of results or methods for handling geodesic flow on \textit{infinite} flat surfaces.
Indeed, if the acute angle of the right triangle $\TTT$ is an irrational multiple of~$\pi$, then iterated \textit{unfolding} reduces billiard flow in $\TTT$ to geodesic flow on an infinite flat surface;  see, \textit{e.g.}, \cite[Section~1.3]{BDY1}.
Unfortunately we know much, much less about the case of infinite flat surfaces than about finite flat surfaces.

Starting in Section~\ref{sec6.5} we are going to prove several results for infinite polysquare translation surfaces but, unfortunately, we cannot make any progress with Open Problem~1.

We also recall \cite[Theorem~2.1.3]{BDY1}, which extends density to uniformity for every rational surface and for almost every slope.
Unfortunately, this result, like \cite[Theorem~2.1.1]{BDY1}, is a \textit{time-qualitative} result that does not say anything definite about the \textit{necessary time range}, due to the fact that Birkhoff's ergodic theorem and Poincare's recurrence theorem, the underlying thrust of the proof, do not have an explicit error term.

%%%%%%%%%%%%%%%%%%%%

In Section~\ref{sec6.2}, we introduce the size-magnification version of the shortline method to establish Theorem~\ref{thm6.1.1} and identify an infinite class of slopes with superdense geodesics on the L-surface.
In Section~\ref{sec6.4}, we discuss a generalization of Theorem~\ref{thm6.1.1} to a large class of polysquare translation surfaces.
We then extend these ideas to square-maze surfaces in Section~\ref{sec6.5}.

In Sections~\ref{sec6.7}--\ref{sec6.9}, we discuss density of aperiodic surfaces with infinite streets, a problem motivated by the Ehrenfest wind-tree models.

Sections 6.3 and~6.6, as well as the latter part Section~\ref{sec6.4}, have been deleted in this version of the manuscript, as the arguments there contain a serious error.

%%%%%%%%%%
%
% SECTION 6.2
%
%%%%%%%%%%

\subsection{Size-magnification version of the shortline method}\label{sec6.2}

\begin{proof}[Proof of Theorem~\ref{thm6.1.1}]
We refer to geodesics on the L-surface as L-lines.
We start with the concept of \textit{surplus shortline} of an L-line, and then discuss the concept of \textit{exponentially fast zigzagging to a street corner}, which has a crucial role in the new version of the shortline process.
In the rest we refer to this new eigenvalue-free version of the process as the size-magnification version.
The term \textit{size-magnification} will be justified by the arguments below.

In Figure~6.2.1 we consider an almost vertical L-line~$V$.
Assume that $V$ is infinite in both directions, and let $\alpha>1$ denote the slope of~$V$.
Note that we use the term \textit{almost vertical} (or \textit{almost horizontal}) in the very broad sense that the slope is greater than~$1$ (or it is between $0$ and~$1$).
Formally, let $\pi/4<\theta<\pi/2$ be the angle between $V$ and the horizontal side of the L-surface.
Then $\alpha=\tan\theta$; see the picture on the left in Figure~6.2.1.
Here $AB$ and $B'C$ are consecutive line segments of~$V$, and together they exhibit a left to right detour crossing of the vertical street with corners $(0,0)$, $(1,0)$, $(1,2)$ and $(0,2)$, whereas the line segment $AC$ represents a shortcut street crossing of the same street.
(This street is in fact a cylinder.
Owing to the intuitive meaning of street-crossing, we prefer to use the term \textit{street}.)
Now $AC$ is a line segment of the almost horizontal L-line~$H_1$.
We call $H_1$ the \textit{shortline} of~$V$.
Let $\pi/4<\theta_1<\pi/2$ be the angle between $H_1$ and the vertical side of the L-shape.
Then the slope of $H_1$ is $\alpha_1^{-1}$, where $\alpha_1=\tan\theta_1$.

\begin{displaymath}
\begin{array}{c}
\includegraphics[scale=0.75]{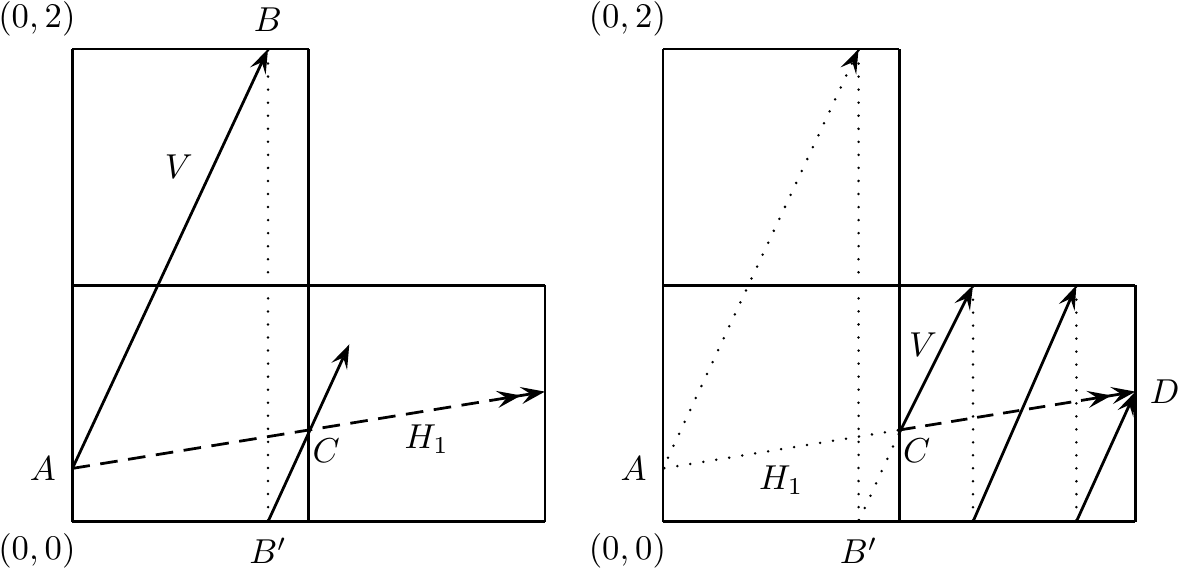}
\vspace{3pt}\\
\mbox{Figure 6.2.1: (left) $H_1$ is the shortline of $V$; (right) second detour}
\\
\mbox{street crossing of $V$, the part of $V$ from $C$ to $D$,}
\\
\mbox{with the part of the shortline $H_1$ from $C$ to $D$}
\end{array}
\end{displaymath}

By hypothesis, $V$ has badly approximable slope satisfying
\begin{equation}\label{eq6.2.1}
\mbox{slope of }V
=\alpha
=a_0+\frac{1}{a_1+\frac{1}{a_2+\frac{1}{a_3+\cdots }}}
=[a_0;a_1,a_2,a_3,\ldots],
\end{equation}
where the continued fraction has the \textit{even digit} property, \textit{i.e.}, $a_i$ is positive and even for every $i\ge0$.
Note that since $\alpha$ is irrational, \eqref{eq6.2.1} has infinitely many digits.
The geometric fact that $AC$ is a shortcut of $AB+B'C$ gives rise to an algebraic relation between the slopes of $V$ and~$H_1$.
In view of the even digit condition, the slope of $H_1$ is $\alpha-a_0$.
More precisely, we have
\begin{equation}\label{eq6.2.2}
\mbox{slope of }H_1
=\alpha_1^{-1}
=\left(a_1+\frac{1}{a_2+\frac{1}{a_3+\frac{1}{a_4+\cdots}}}\right)^{-1}
=[a_1;a_2,a_3,a_4,\ldots]^{-1}.
\end{equation}
In other words, the continued fraction of the slope of~$H_1$, the shortline of~$V$, is obtained from the continued fraction of the slope of $V$ by a shift followed by taking inverse.

The three consecutive line segments of $V$ between $C$ and $D$ in the picture on the right in Figure~6.2.1 together exhibit a left to right detour crossing of the vertical street with corners $(1,0)$, $(2,0)$, $(2,1)$ and $(1,1)$, whereas the line segment $CD$ represents a shortcut street crossing of the same street.
Now $CD$ is a line segment of the almost horizontal L-line~$H_1$, the shortline of~$V$.

Figure~6.2.1 also illustrates the crucial geometric property that any almost vertical L-line $V$ and its shortline $H_1$ have precisely the same
edge-cutting points on the vertical sides of vertical streets.
We refer to this as the \textit{vertical same edge cutting property} of the shortline process.

We can iterate this shortline process.
To find the shortline of~$H_1$, we simply repeat the argument above by switching the roles of horizontal and vertical.
Let $V_2$ denote the shortline of~$H_1$, and let $\alpha_2>1$ denote the slope of the almost vertical L-line~$V_2$.
Again, in view of the even digit condition, we have an analog of \eqref{eq6.2.1} and \eqref{eq6.2.2}, in the form
\begin{displaymath}
\mbox{slope of }V_2
=\alpha_2
=a_2+\frac{1}{a_3+\frac{1}{a_4+\frac{1}{a_5+\cdots}}}
=[a_2;a_3,a_4,a_5,\ldots].
\end{displaymath}
In other words, the continued fraction of the slope of~$V_2$, the shortline of~$H_1$, is obtained from the continued fraction of the slope of $H_1$ by taking inverse followed by a shift.

We also have the analogous \textit{horizontal same edge cutting property} of the shortline process, that an almost horizontal L-line $H_1$ and its shortline $V_2$ have precisely the same edge-cutting points on the horizontal sides of horizontal streets. 

Of course we can define in a similar way the shortline of~$V_2$, and so on.
Thus we obtain an infinite sequence
\begin{equation}\label{eq6.2.3}
V\to H_1\to V_2\to H_3\to V_4\to H_5\to\cdots,
\end{equation}
where, for every $i\ge0$,
\begin{equation}\label{eq6.2.4}
H_{2i+1}\mbox{ is the shortline of }V_{2i}
\quad\mbox{and}\quad
V_{2i+2}\mbox{ is the shortline of }H_{2i+1},
\end{equation}
\begin{equation}\label{eq6.2.5}
V_{2i}\mbox{ and }H_{2i+1}
\mbox{ satisfy the vertical same edge cutting property},
\end{equation}
and
\begin{equation}\label{eq6.2.6}
H_{2i+1}\mbox{ and }V_{2i+2}
\mbox{ satisfy the horizontal same edge cutting property}.
\end{equation}

Combining \eqref{eq6.2.3}--\eqref{eq6.2.6} we obtain an exponentially fast zigzagging to a street corner; see Figure~6.2.2.

\begin{displaymath}
\begin{array}{c}
\includegraphics[scale=0.75]{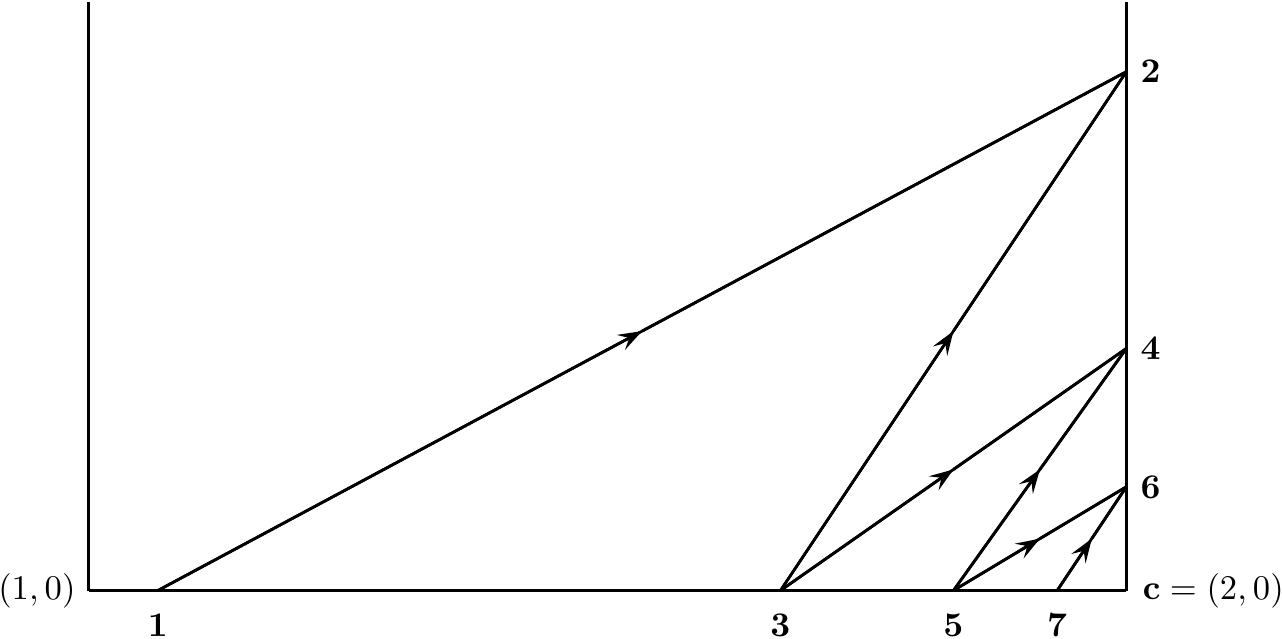}
\vspace{3pt}\\
\mbox{Figure 6.2.2: exponentially fast zigzagging to a street corner}
\end{array}
\end{displaymath}

Assume, for example, that the line segment $\mathbf{12}$ in Figure~6.2.2 belongs to an almost horizontal L-line~$H_{2i+5}$, say, in \eqref{eq6.2.3}.
The term \textit{exponentially fast zigzagging} means the following.
By \eqref{eq6.2.5} the line segment $\mathbf{23}$ in Figure~6.2.2 belongs to the almost vertical L-line $V_{2i+4}$ in \eqref{eq6.2.3}.
By \eqref{eq6.2.6} the line segment $\mathbf{34}$ belongs to the almost horizontal L-line~$H_{2i+3}$.
By \eqref{eq6.2.5} the line segment $\mathbf{45}$ belongs to the almost vertical L-line~$V_{2i+2}$.
By \eqref{eq6.2.6} the line segment $\mathbf{56}$ belongs to the almost horizontal L-line~$H_{2i+1}$.
Finally, by \eqref{eq6.2.5} the line segment $\mathbf{67}$ belongs to the almost vertical L-line~$V_{2i}$.
Thus these line segments of the successive ancestor L-lines zigzag towards a singularity $\bfc$ of the L-surface.

The zigzagging in Figure~6.2.2 represents an  exponentially fast convergence to the street corner $\bfc=(2,0)$.
More precisely, since the slope of $H_{2i+5}$ is $\alpha_{2i+5}^{-1}$, we have
\begin{equation}\label{eq6.2.7}
\frac{\length(\mathbf{1c})}{\length(\mathbf{2c})}=\alpha_{2i+5}.
\end{equation}
Similarly,
\begin{equation}\label{eq6.2.8}
\begin{array}{c}
{\displaystyle\frac{\length(\mathbf{2c})}{\length(\mathbf{3c})}=\alpha_{2i+4}},
\quad
{\displaystyle\frac{\length(\mathbf{3c})}{\length(\mathbf{4c})}=\alpha_{2i+3}},
\quad
{\displaystyle\frac{\length(\mathbf{4c})}{\length(\mathbf{5c})}=\alpha_{2i+2}},
\vspace{5pt}\\
{\displaystyle\frac{\length(\mathbf{5c})}{\length(\mathbf{6c})}=\alpha_{2i+1}},
\quad
{\displaystyle\frac{\length(\mathbf{6c})}{\length(\mathbf{7c})}=\alpha_{2i}}.
\end{array}
\end{equation}
We can also write \eqref{eq6.2.7}--\eqref{eq6.2.8} in the equivalent form
\begin{displaymath}
\begin{array}{c}
{\displaystyle\length(\mathbf{2c})=\frac{\length(\mathbf{1c})}{\alpha_{2i+5}}},
\quad
{\displaystyle\length(\mathbf{3c})=\frac{\length(\mathbf{1c})}{\alpha_{2i+5}\alpha_{2i+4}}},
\vspace{5pt}\\
{\displaystyle\length(\mathbf{4c})=\frac{\length(\mathbf{1c})}{\alpha_{2i+5}\alpha_{2i+4}\alpha_{2i+3}}},
\quad
{\displaystyle\length(\mathbf{5c})=\frac{\length(\mathbf{1c})}{\alpha_{2i+5}\alpha_{2i+4}\alpha_{2i+3}\alpha_{2i+2}}},
\vspace{5pt}\\
{\displaystyle\length(\mathbf{6c})=\frac{\length(\mathbf{1c})}{\alpha_{2i+5}\alpha_{2i+4}\alpha_{2i+3}\alpha_{2i+2}\alpha_{2i+1}}},
\vspace{5pt}\\
{\displaystyle\length(\mathbf{7c})
=\frac{\length(\mathbf{1c})}{\alpha_{2i+5}\alpha_{2i+4}\alpha_{2i+3}\alpha_{2i+2}\alpha_{2i+1}\alpha_{2i}}}.
\end{array}
\end{displaymath}

We return to the chain \eqref{eq6.2.3}.
Let $V^*$ be a finite initial segment of this almost vertical L-line $V$ with slope $\alpha>1$, and assume that $V^*$ is \textit{long}.
It is clear that $V^*$ consists of a number of \textit{whole} detour crossings and a \textit{fractional} detour crossing at the end.
Clearly the length of $V^*$ is some multiple of $(1+\alpha^2)^{1/2}$, the common length of detour crossings of slope $\alpha$ of vertical streets.
In other words,
\begin{equation}\label{eq6.2.9}
\length(V^*)=m_0(1+\alpha^2)^{1/2}
\quad\mbox{for some large positive real number $m_0$},
\end{equation}
where the integer part of $m_0$ is the number of whole detour crossings in~$V^*$.
Each whole detour crossing in $V^*$ has a shortcut, which is part of the almost horizontal shortline $H_1$ of~$V$.
The fractional detour crossing at the end in~$V^*$, if extended to a full detour crossing, also has a shortcut, which is also part of the almost horizontal shortline $H_1$ of~$V$.
For this fractional detour crossing, we shorten its shortcut by the same fraction and at the appropriate end to obtain a fractional shortcut.
We then take the union of these shortcuts and this fractional shortcut.
This union is a segment of $H_1$ that we denote by~$H_1^*$.
Clearly the length of $H_1^*$ is some multiple of $(1+\alpha_1^2)^{1/2}$, the common length of detour crossings of slope $\alpha_1^{-1}$ of horizontal streets.
In other words,
\begin{displaymath}
\length(H_1^*)=m_1(1+\alpha_1^2)^{1/2}
\quad\mbox{for some positive real number $m_1$}.
\end{displaymath}
We keep iterating this.
Let
\begin{align}
\length(V_2^*)
&
=m_2(1+\alpha_2^2)^{1/2}
\quad\mbox{for some positive real number $m_2$},
\nonumber
\\
\length(H_3^*)
&
=m_3(1+\alpha_3^2)^{1/2}
\quad\mbox{for some positive real number $m_3$},
\nonumber
\end{align}
and so on. 
Consider the decreasing sequence
\begin{equation}\label{eq6.2.10}
m_0\ge m_1\ge m_2\ge m_3\ge\cdots.
\end{equation}

At this point, we make the \textit{assumption} that the continued fraction digits~$a_i$, $i=0,1,2,3,\ldots,$ have a common upper bound $a_i<U$, where $U$ is an integer.
In other words, the number $\alpha$ is badly approximable.
This implies, in particular, that for every $i=0,1,2,3,\ldots,$ we have the bound
\begin{equation}\label{eq6.2.11}
\alpha_i<U.
\end{equation}

It is almost trivial to note that
\begin{equation}\label{eq6.2.12}
\alpha_1m_1=m_0.
\end{equation}
Indeed, we have the general form that for every $j\ge0$,
\begin{equation}\label{eq6.2.13}
\alpha_{j+1}m_{j+1}=m_j.
\end{equation}
Iterating \eqref{eq6.2.13} we deduce that
\begin{equation}\label{eq6.2.14}
m_k=m_0\prod_{j=1}^k\frac{1}{\alpha_j}.
\end{equation}

Let us return to the sequence \eqref{eq6.2.10}.
We need the following lemma.

\begin{lem}\label{lem6.2.1}
There is a member $m_\ell$ of the sequence \eqref{eq6.2.10} that satisfies the inequalities $2U+1\le m_\ell\le4U^5$ and the corresponding $V^*_\ell$ or~$H^*_\ell$, depending on the parity of~$\ell$, exhibits all six types of corner cuts illustrated in Figure~6.2.3.
\end{lem}

\begin{displaymath}
\begin{array}{c}
\includegraphics[scale=0.75]{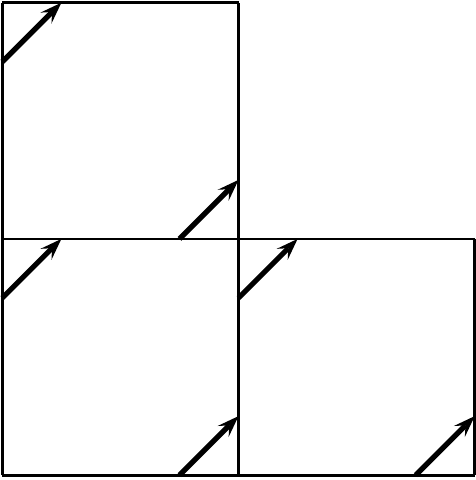}
\vspace{3pt}\\
\mbox{Figure 6.2.3: six types of corner cuts of the L-surface}
\end{array}
\end{displaymath}

\begin{remark}
The requirement $m_\ell\ge2U+1$ is motivated by a later application of \eqref{eq6.2.14}, while the other requirement $m_\ell\le4U^5$ will be clear from the proof of the lemma.
\end{remark}

Before we can prove Lemma~\ref{lem6.2.1}, we need to introduce the concept of \textit{almost vertical units} of an almost vertical L-line and \textit{almost horizontal units} of an almost horizontal L-line.

Suppose that $V$ is an almost vertical L-line of slope $\gamma>1$.
An almost vertical unit of this L-line is a finite segment of~$V$, of length $(1+\gamma^{-2})^{1/2}$, that goes from one horizontal edge of the L-surface to another horizontal edge.
There are six different types of almost vertical units, illustrated in Figure~6.2.4.

\begin{displaymath}
\begin{array}{c}
\includegraphics[scale=0.75]{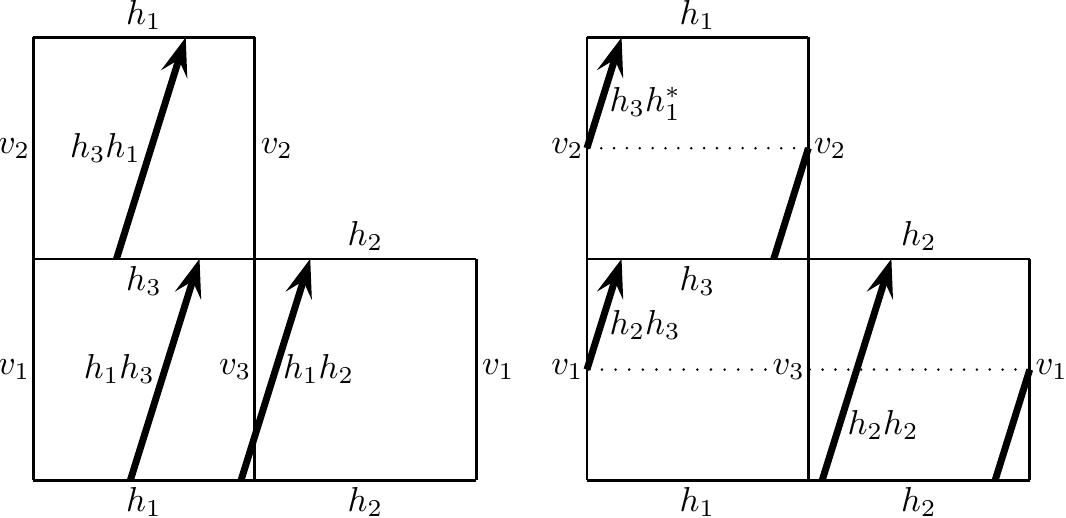}
\vspace{3pt}\\
\mbox{Figure 6.2.4: six types of almost vertical units}
\end{array}
\end{displaymath}

In Figure~6.2.4, the almost vertical unit $h_1h_3$ in the picture on the left starts from the edge $h_1$ and ends on the edge~$h_3$, and is clearly of length $(1+\gamma^{-2})^{1/2}$ since its slope is~$\gamma$.
In the picture on the left, the almost vertical units $h_1h_2$ and $h_3h_1$ are also illustrated.
Likewise, the almost vertical unit $h_2h_2$ is illustrated in the picture on the right.
As shown in the picture on the right, the two almost vertical units $h_2h_3$ and $h_3h_1^*$ are each broken into two pieces.
We can write $h_2h_3=h_2v_1h_3$ to emphasize the fact that this almost vertical unit is broken at the edge~$v_1$.
Likewise, we can write $h_3h_1^*=h_3v_2h_1$ to emphasize the fact that this almost vertical unit is broken at the edge~$v_2$.
Note that $h_3h_1$ and $h_3h_1^*$ both start from the edge $h_3$ and end on the edge~$h_1$.
While the former is in one piece, the latter is broken at the edge~$v_2$.

Suppose that $H$ is an almost horizontal L-line of slope $\gamma^{-1}$, where $\gamma>1$.
An almost horizontal unit of this L-line is a finite segment of~$H$, of length $(1+\gamma^{-2})^{1/2}$, that goes from one vertical edge of the L-surface to another vertical edge.
There are six different types of almost horizontal units, illustrated in Figure~6.2.5.

\begin{displaymath}
\begin{array}{c}
\includegraphics[scale=0.75]{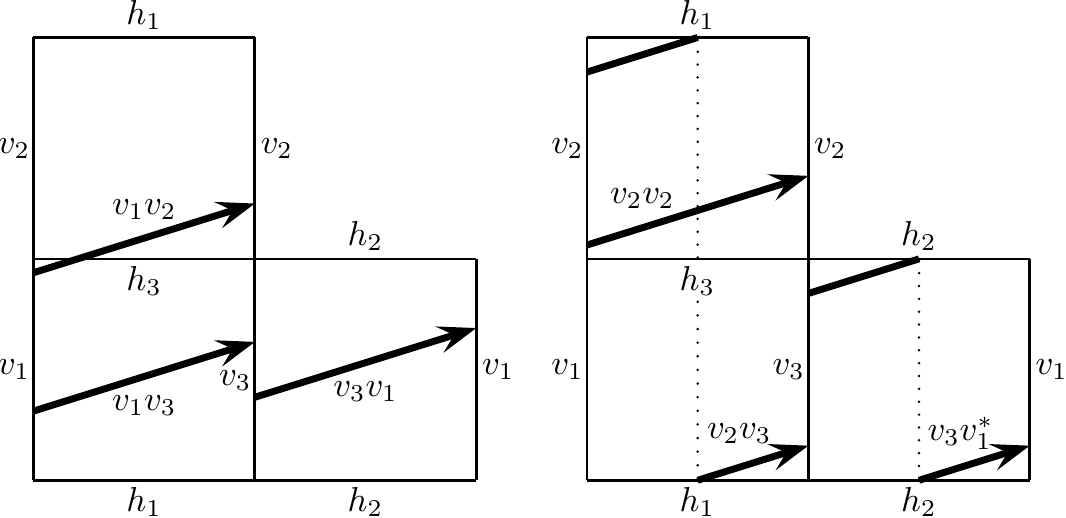}
\vspace{3pt}\\
\mbox{Figure 6.2.5: six types of almost horizontal units}
\end{array}
\end{displaymath}

Next we need to introduce the concept of \textit{ancestor units}.

Consider first an almost horizontal L-line~$H_{2i+1}$, with shortline~$V_{2i+2}$.
Then any almost vertical unit of $V_{2i+2}$ is the shortcut of an almost horizontal detour crossing of a horizontal street, made up of some almost horizontal units of~$H_{2i+1}$, together with some fractional units at the two ends.
To illustrate this, the reader is referred to Figure~6.2.6.

\begin{displaymath}
\begin{array}{c}
\includegraphics[scale=0.75]{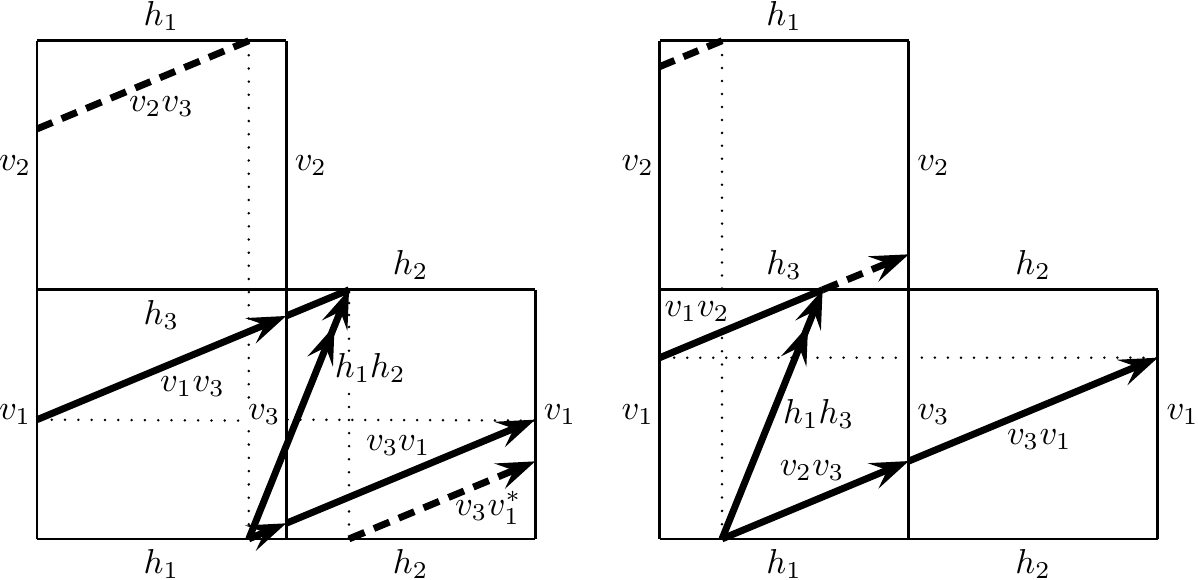}
\vspace{3pt}\\
\mbox{Figure 6.2.6: ancestor units of the almost vertical units $h_1h_2$ and $h_1h_3$}
\end{array}
\end{displaymath}

In the picture on the left, it is shown that the almost vertical unit $h_1h_2$ is the shortcut of an almost horizontal detour crossing of a horizontal street, made up of a fractional unit~$v_2v_3$, followed by two complete units $v_3v_1$ and~$v_1v_3$, and ending with a fractional  unit~$v_3v_1^*$.

In general, there may be extra copies of the whole units $v_3v_1$ and~$v_1v_3$, if the slope of the almost horizontal detour crossing of the horizontal street is very small.

\begin{extension}
Extend the fractional units at either end of the detour crossing to whole units.
\end{extension}

Applying the extension rule, the ancestor units of $h_1h_2$ must contain types $v_2v_3$, $v_3v_1$, $v_1v_3$ and $v_3v_1^*$, and we denote this fact by writing
\begin{equation}\label{eq6.2.15}
h_1h_2
\hookrightarrow
v_2v_3,v_3v_1,v_1v_3,v_3v_1^*,
\end{equation}
with the convention that ancestor units present with multiplicity are listed only once.
In the picture on the right, we start with the almost vertical unit $h_1h_3$.
Using a similar analysis and applying the extension rule, the ancestor units of $h_1h_3$ must contain types $v_2v_3$, $v_3v_1$ and $v_1v_2$, and we denote this fact by writing
\begin{equation}\label{eq6.2.16}
h_1h_3
\hookrightarrow
v_2v_3,v_3v_1,v_1v_2.
\end{equation}
Analagous considerations give
\begin{align}
h_2h_2
&\hookrightarrow
v_3v_1^*,v_1v_3,
\label{eq6.2.17}
\\
h_2h_3
&\hookrightarrow
v_3v_1^*,v_1v_3,v_3v_1,v_1v_2,
\label{eq6.2.18}
\\
h_3h_1
&\hookrightarrow
v_1v_2,v_2v_2,v_2v_3,
\label{eq6.2.19}
\\
h_3h_1^*
&\hookrightarrow
v_1v_2,v_2v_2,v_2v_3.
\label{eq6.2.20}
\end{align}

Consider next an almost vertical L-line~$V_{2i+2}$, with shortline~$H_{2i+3}$.
Then any almost horizontal unit of $H_{2i+3}$ is the shortcut of an almost vertical detour crossing of a vertical street, made up of some almost vertical units of~$V_{2i+2}$, together with some fractional units at the two ends.
Analogous to \eqref{eq6.2.15}--\eqref{eq6.2.20}, we have
\begin{align}
v_1v_2
&\hookrightarrow
h_2h_3,h_3h_1,h_1h_3,h_3h_1^*,
\label{eq6.2.21}
\\
v_1v_3
&\hookrightarrow
h_2h_3,h_3h_1,h_1h_2,
\label{eq6.2.22}
\\
v_2v_2
&\hookrightarrow
h_3h_1^*,h_1h_3,
\label{eq6.2.23}
\\
v_2v_3
&\hookrightarrow
h_3h_1^*,h_1h_3,h_3h_1,h_1h_2,
\label{eq6.2.24}
\\
v_3v_1
&\hookrightarrow
h_1h_2,h_2h_2,h_2h_3,
\label{eq6.2.25}
\\
v_3v_1^*
&\hookrightarrow
h_1h_2,h_2h_2,h_2h_3.
\label{eq6.2.26}
\end{align}

Consider now the chain
\begin{equation}\label{eq6.2.27}
H_{2i+1}\to V_{2i+2}\to H_{2i+3}\to V_{2i+4}.
\end{equation}
Starting with almost vertical units in $V_{2i+4}$ and identifying their ancestors three times iteratively, using \eqref{eq6.2.15}--\eqref{eq6.2.26}, we obtain
\begin{align}
h_1h_2
&
\hookrightarrow
v_2v_3,v_3v_1,v_1v_3,v_3v_1^*
%\nonumber
%\\
%&
\hookrightarrow
h_3h_1^*,h_1h_3,h_3h_1,h_1h_2,h_2h_2,h_2h_3
\nonumber
\\
&
\hookrightarrow
v_1v_2,v_2v_2,v_2v_3,v_3v_1,v_1v_3,v_3v_1^*,
\label{eq6.2.28}
\\
h_1h_3
&
\hookrightarrow
v_2v_3,v_3v_1,v_1v_2
%\nonumber
%\\
%&
\hookrightarrow
h_3h_1^*,h_1h_3,h_3h_1,h_1h_2,h_2h_2,h_2h_3
\nonumber
\\
&
\hookrightarrow
v_1v_2,v_2v_2,v_2v_3,v_3v_1,v_1v_3,v_3v_1^*,
\label{eq6.2.29}
\\
h_2h_2
&
\hookrightarrow
v_3v_1^*,v_1v_3
%\nonumber
%\\
%&
\hookrightarrow
h_1h_2,h_2h_2,h_2h_3,h_3h_1
\nonumber
\\
&
\hookrightarrow
v_2v_3,v_3v_1,v_1v_3,v_3v_1^*,v_1v_2,v_2v_2,
\label{eq6.2.30}
\\
h_2h_3
&
\hookrightarrow
v_3v_1^*,v_1v_3,v_3v_1,v_1v_2
%\nonumber
%\\
%&
\hookrightarrow
h_1h_2,h_2h_2,h_2h_3,h_3h_1,h_1h_3
\nonumber
\\
&
\hookrightarrow
v_2v_3,v_3v_1,v_1v_3,v_3v_1^*,v_1v_2,v_2v_2,
\label{eq6.2.31}
\\
h_3h_1
&
\hookrightarrow
v_1v_2,v_2v_2,v_2v_3
%\nonumber
%\\
%&
\hookrightarrow
h_2h_3,h_3h_1,h_1h_3,h_3h_1^*,h_1h_2
\nonumber
\\
&
\hookrightarrow
v_3v_1^*,v_1v_3,v_3v_1,v_1v_2,v_2v_2,v_2v_3,
\label{eq6.2.32}
\\
h_3h_1^*
&
\hookrightarrow
v_1v_2,v_2v_2,v_2v_3
%\nonumber
%\\
%&
\hookrightarrow
h_2h_3,h_3h_1,h_1h_3,h_3h_1^*,h_1h_2
\nonumber
\\
&
\hookrightarrow
v_3v_1^*,v_1v_3,v_3v_1,v_1v_2,v_2v_2,v_2v_3.
\label{eq6.2.33}
\end{align}

\begin{proof}[Proof of Lemma~\ref{lem6.2.1}]
Since $V^*$ is long, we can clearly assume that $m_0>4U^5$.
Let $\ell$ be the unique positive integer satisfying the inequalities
\begin{equation}\label{eq6.2.34}
m_\ell\le4U^5<m_{\ell-1}.
\end{equation}
Then it follows from \eqref{eq6.2.11}, \eqref{eq6.2.13} and \eqref{eq6.2.34} that
\begin{equation}\label{eq6.2.35}
m_\ell\ge\frac{m_{\ell-1}}{U}>4U^4\ge2U+1
\quad\mbox{and}\quad
m_{\ell+3}\ge\frac{m_{\ell-1}}{U^4}>4U.
\end{equation}
Without loss of generality, suppose that $\ell$ is odd.
To show that $H^*_\ell$ exhibits all six types of corner cuts illustrated in Figure~6.2.3, it suffices to show that it contains at least one copy of each of the three horizontal units $v_1v_2$, $v_2v_3$ and~$v_3v_1^*$.
We shall in fact show that $H^*_\ell$ contains at least one copy of each of the six types of horizontal units.

Note that it follows from the second set of inequalities in \eqref{eq6.2.35} that $m_{\ell+3}\ge4$.
This means that $V^*_{\ell+3}$ contains at least $3$ whole almost vertical detour crossings of vertical streets.
Take one such whole almost vertical detour crossing in the middle.
This must contain one of the six almost vertical units.
In finding its ancestors, the use of the extension rule is justified.
The ancestors of this almost vertical unit contains various types of almost horizontal units, all of which must be in $H^*_{\ell+2}$.
Their ancestors contain various types of almost vertical units, all of which must be in $V^*_{\ell+1}$.
In turn, their ancestors contain various types of almost horizontal units, all of which must be in~$H^*_\ell$.

Taking the chain \eqref{eq6.2.27} to be the chain $H_\ell\to V_{\ell+1}\to H_{\ell+2}\to V_{\ell+3}$, and noting the ancestor relations
\eqref{eq6.2.28}--\eqref{eq6.2.33}, it is clear that $H^*_\ell$ contains at least one copy of each of the six types of horizontal units.
This completes the proof of Lemma~\ref{lem6.2.1}.
\end{proof}

\begin{remark}
Note that our proof of Lemma~\ref{lem6.2.1} here is a brute force exercise which does not give us any insight into what really is going on.
Furthermore, this approach is only possible for the L-surface where we have very precise information on the almost horizontal and almost vertical units.
If we consider surfaces other than this particular L-surface, the corresponding brute force exercise may well turn out to be an extremely unpleasant exercise.
Indeed, in Section~\ref{sec6.4}, we shall consider generalizations of the L-surface, leading to infinitely many analogs.
Lemma~\ref{lem6.4.1} is the generalized version of Lemma~\ref{lem6.2.1}, and we shall develop there a substantially simpler proof which also gives us more insight into the problem.
\end{remark}

We next describe an iterative process for obtaining vertical and horizontal open intervals on the edges of the L-surface.
The L-surface is the simplest interesting example of a polysquare translation surface, to be defined in Section~\ref{sec6.4}.
We adopt the following simple rule concerning where flow images on such surfaces should lie.

\begin{magsquare}
Suppose, without loss of generality, that the edges of a polysquare translation surface $\PPP$ lie on lines of the form $x=x_0$ and $y=y_0$, where $x_0$ and $y_0$ are integers.
Let $I$ be an open interval lying entirely within an edge of~$\PPP$, and let $\gamma>1$.

(i) Suppose that $I$ lies on a vertical edge of a square face of~$\PPP$.
Then there exist integers $y_1$ and $y_2$ satisfying $y_2-y_1=1$ and such that the bottom edge of the square face lies on the line $y=y_1$ and the top edge of the square face lies on the line $y=y_2$.
If we project $I$ by the forward almost horizontal $\gamma^{-1}$-flow, then the image of $I$ lies on edges of $\PPP$ that form part of the line $y=y_2$.
If we project $I$ by the reverse almost horizontal $\gamma^{-1}$-flow, then the image of $I$ lies on edges of $\PPP$ that form part of the line $y=y_1$.

(ii) Suppose that $I$ lies on a horizontal edge of a square face of~$\PPP$.
Then there exist integers $x_1$ and $x_2$ satisfying $x_2-x_1=1$ and such that the left edge of the square face lies on the line $x=x_1$ and the right edge of the square face lies on the line $x=x_2$.
If we project $I$ by the forward almost vertical $\gamma$-flow, then the image of $I$ lies on edges of $\PPP$ that form part of the line $x=x_2$.
If we project $I$ by the reverse almost vertical $\gamma$-flow, then the image of $I$ lies on edges of $\PPP$ that form part of the line $x=x_1$.
\end{magsquare}

For the L-surface, the only possible choices for $(x_1,x_2)$ and $(y_1,y_2)$ are $(0,1)$ and $(1,2)$.

Suppose that $I_0$ is an open interval on a vertical edge of the L-surface.
Figure~6.2.7 shows examples of this where $I_0$ lies on the vertical edge $v_1$ of the bottom left square face.
This edge $v_1$ lies between the horizontal lines $y=0$ and $y=1$.
We wish to project this interval to edges of the L-surface that lie on these two lines using the forward or reverse almost horizontal $\alpha_1^{-1}$-flow.

\begin{displaymath}
\begin{array}{c}
\includegraphics[scale=0.75]{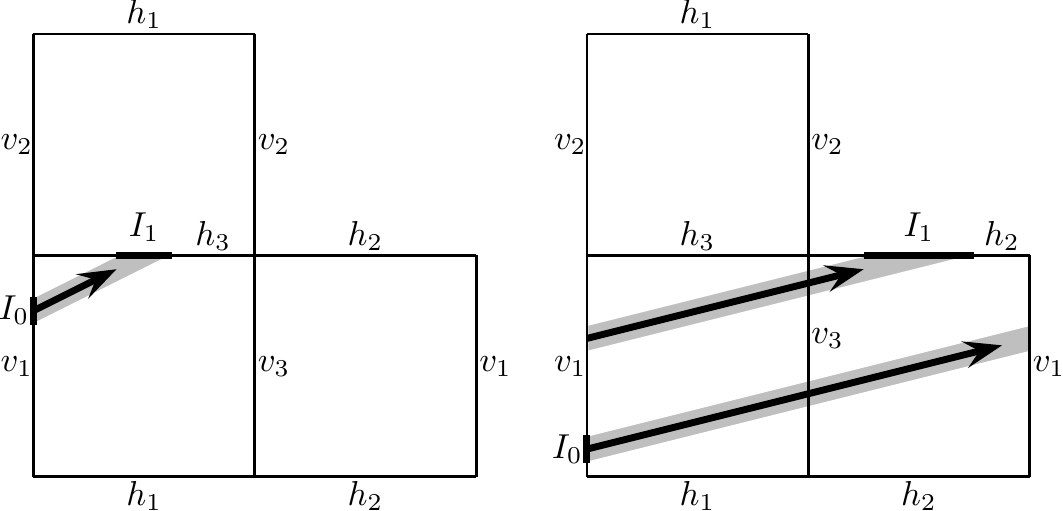}
\vspace{3pt}\\
\mbox{Figure 6.2.7: examples of good almost-horizontal flows for $I_0$}
\end{array}
\end{displaymath}

The forward almost horizontal $\alpha_1^{-1}$-flow from left to right gives rise to an image of $I_0$ on horizontal edges that lie on the line $y=1$.
Clearly the image falls on either the edge $h_3$ of the bottom left square face or the top edge $h_2$ of the right square face, as shown in Figure~6.2.7, or it is split between these two edges.

Likewise, in view of the identification of the two vertical edges $v_1$ of the L-surface, the reverse almost horizontal $\alpha_1^{-1}$-flow from right to left gives rise to an image of $I_0$ on horizontal edges that lie on the line $y=0$.
The image falls on either the bottom edge $h_1$ of the bottom left square face or the bottom edge $h_2$ of the right square face, or it is split between these two edges.

We say that this forward or reverse flow is \textit{good} for the interval~$I_0$ if the image does not hit a vertex of the L-surface and is an open interval on a single horizontal edge.
The examples given in Figure~6.2.7 are good flows.

In this case, we define $I_1$ to be the corresponding open horizontal interval on the appropriate horizontal edge.

There are clearly instances when the forward almost horizontal $\alpha_1^{-1}$-flow from left to right acting on an open vertical interval $I_0$ on some vertical edge of the
L-surface fails to deliver an image that does not hit a vertex of the L-surface, resulting in one or more \textit{splits}.

This happens precisely when the flow encounters the top right vertex of one of the three constituent square faces of the L-surface, causing the flow to split.
In this case, we say that this flow is \textit{bad} for the interval~$I_0$.
The image $I_1$ will contain a split singularity and may be in multiple pieces.

However, apart from these, its precise structure is not of any serious concern in our discussion.

The more pertinent question is the effect of the split singularity.

There are three separate cases.

Case~1.
The flow hits the split singularity which is the top right vertex of the bottom left square face, as shown in the picture on the left in Figure~6.2.8.

\begin{displaymath}
\begin{array}{c}
\includegraphics[scale=0.75]{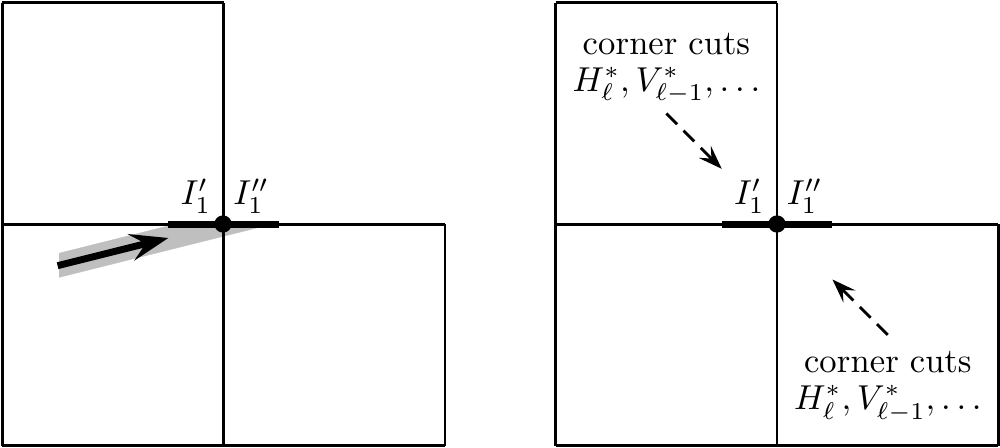}
\vspace{3pt}\\
\mbox{Figure 6.2.8: bad almost-horizontal $\alpha_1^{-1}$-flow,}
\\
\mbox{hitting the top right vertex of the bottom left square face}
\end{array}
\end{displaymath}

Case~2.
The flow hits the split singularity which is the top right vertex of the top square face, as shown in the picture on the left in Figure~6.2.9, where we have placed the interval $I'_1$ on the bottom edge, in view of edge identification.

\begin{displaymath}
\begin{array}{c}
\includegraphics[scale=0.75]{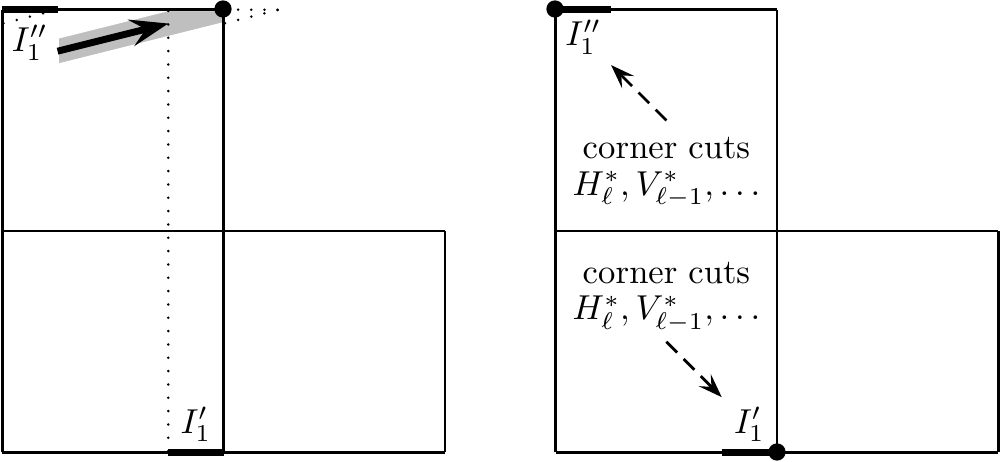}
\vspace{3pt}\\
\mbox{Figure 6.2.9: bad almost-horizontal $\alpha_1^{-1}$-flow,}
\\
\mbox{hitting the top right vertex of the top square face}
\end{array}
\end{displaymath}

Case~3.
The flow hits the split singularity which is the top right vertex of the right square face, as shown in the picture on the left in Figure~6.2.10, where we have placed the interval $I'_1$ on the bottom edge, in view of edge identification.

\begin{displaymath}
\begin{array}{c}
\includegraphics[scale=0.75]{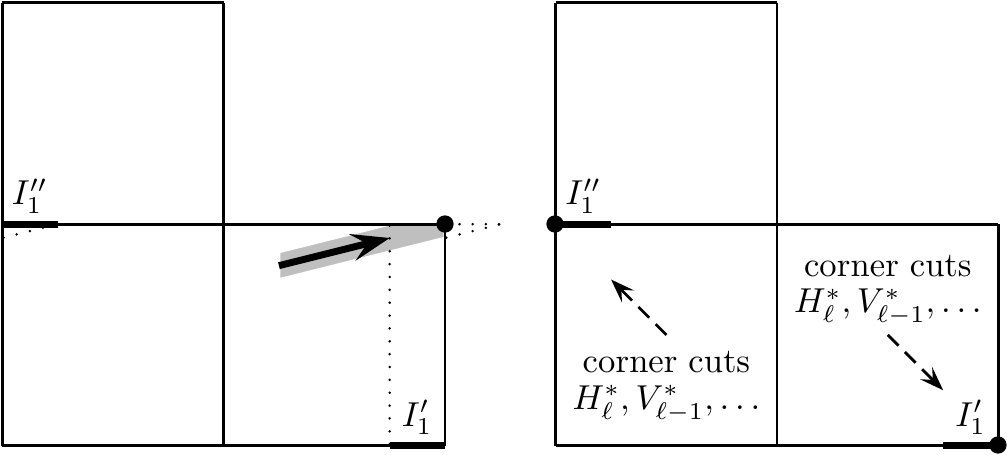}
\vspace{3pt}\\
\mbox{Figure 6.2.10: bad almost-horizontal $\alpha_1^{-1}$-flow,}
\\
\mbox{hitting the top right vertex of the right square face}
\end{array}
\end{displaymath}

We remark that in any of the three cases, the interval $I''_1$ can possibly take up a whole edge of the L-surface or even more.

There are also instances when the reverse almost horizontal $\alpha_1^{-1}$-flow from right to left acting on an open vertical interval $I_0$ on some vertical edge of the
L-surface fails to deliver an image that does not hit a vertex of the L-surface.
This happens precisely when the flow encounters the bottom left vertex of one of the three constituent square faces of the L-surface, causing the flow to split.

Assume for the time being that the forward almost horizontal $\alpha_1^{-1}$-flow from left to right or the reverse almost horizontal $\alpha_1^{-1}$-flow from right to left takes the open interval $I_0$ to a single open interval on an appropriate horizontal edge without capturing any vertex of the L-surface along the way, and that this leads to an open interval~$I_1$, as shown in the picture on the left in Figure~6.2.11.

We may apply the forward almost vertical $\alpha_2$-flow from left to right on the interval~$I_1$.
This may result in a single open interval $I_2$ on an appropriate vertical edge of the L-surface, as shown in the picture on the right in Figure~6.2.11, or the flow encounters the top right vertex of one of the three constituent square faces of the L-surface.

\begin{displaymath}
\begin{array}{c}
\includegraphics[scale=0.75]{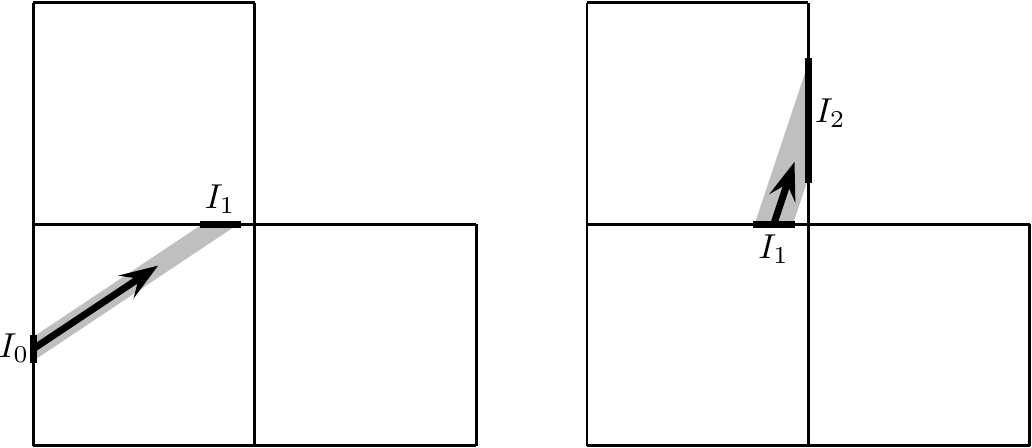}
\vspace{3pt}\\
\mbox{Figure 6.2.11: chain $I_0\to I_1\to I_2$ given by a good $\alpha_1^{-1}$ flow and a good $\alpha_2$-flow}
\end{array}
\end{displaymath}

We may also apply the reverse almost vertical $\alpha_2$-flow from right to left on the interval~$I_1$.
This may result in a single open interval $I_2$ on an appropriate vertical edge of the L-surface, or the flow encounters the bottom left vertex of one of the three constituent square faces of the L-surface.

And so on.

We thus have an iterative process, bearing in mind that the process may cease when a flow encounters a vertex of one of the three constituent square faces of the L-surface.

Suppose for the moment that we have some chain
\begin{equation}\label{eq6.2.36}
I_0\to I_1\to I_2\to I_3\to I_4\to\cdots
\end{equation}
of alternate open vertical and horizontal intervals on appropriate edges of the L-surface.
Owing to the slopes, these tilted projections magnify the intervals, as can clearly been seen in Figure~6.2.11.
It is easy to give a quantitative description of the \textit{magnification process} \eqref{eq6.2.36}.
It is clear from Figure~6.2.11 that
\begin{displaymath}
\frac{\length(I_1)}{\length(I_0)}=\alpha_1
\quad\mbox{and}\quad
\frac{\length(I_2)}{\length(I_1)}=\alpha_2,
\end{displaymath}
and so on.

Suppose further that $I_0$ is a $V^*$-free open interval on some vertical edge of the L-surface, meaning that it does not contain any edge-cutting points of~$V^*$.
The forward or reverse almost horizontal $\alpha_1^{-1}$-flow projects $I_0$ to an open interval~$I_1$ on some appropriate horizontal edge of the
L-surface.
By the same edge cutting property, $I_1$ does not contain any edge-cutting points of~$H_1^*$.
The forward or reverse almost vertical $\alpha_2$-flow projects $I_1$ to an open vertical interval~$I_2$ on some appropriate vertical edge of the
L-surface.
By the same edge cutting property, $I_2$ does not contain any edge-cutting points of~$V_2^*$.
And so on.

\begin{lem}\label{lem6.2.2}
Every $V^*$-free open interval $I_0$ on a vertical edge of the L-surface satisfies
\begin{equation}\label{eq6.2.37}
\length(I_0)\le\frac{2}{\alpha_1\alpha_2\alpha_3\cdots\alpha_{\ell-1}\alpha_\ell},
\end{equation}
where $\ell$ is the index of $m_\ell$ in Lemma~\ref{lem6.2.1}.
If there are several such indices, we choose the smallest one.
\end{lem}

The main idea of the proof of Lemma~\ref{lem6.2.2}, that of zigzagging towards a split singularity, can be summarized in the proof of the following crucial step.

\begin{lem}\label{lem6.2.3}
Suppose that under the hypotheses of Lemma~\ref{lem6.2.2}, the inequality
\begin{equation}\label{eq6.2.38}
\length(I_0)>\frac{2}{\alpha_1\alpha_2\alpha_3\cdots\alpha_{\ell-1}\alpha_\ell}
\end{equation}
holds.
Then the forward or reverse $\alpha_1^{-1}$-flow projects $I_0$ to an $H^*_1$-free interval $I_1$ on some horizontal edge of the L-surface, so that $I_1$ does not contain a split singularity.
Furthermore, we have
\begin{equation}\label{eq6.2.39}
\length(I_1)
=\alpha_1\length(I_0)
>\frac{2}{\alpha_2\alpha_3\alpha_4\cdots\alpha_{\ell-1} \alpha_\ell}.
\end{equation}
\end{lem}

\begin{proof}
Suppose on the contrary that $I_1$ contains a split singularity.
For convenience, we assume that we are using the forward $\alpha_1^{-1}$-flow, and the flow hits the top right vertex of the bottom left square face, as described in Case~1 earlier.
The other two cases, as well as the case of reverse $\alpha_1^{-1}$-flow, can be treated in almost the same way, with only very minor modifications.
The similarities and differences of the three cases concerning the forward $\alpha_1^{-1}$-flow are illustrated in Figures 6.2.8--6.2.10.

We define the \textit{temporary} intervals $I'_1$ and $I''_1$ as indicated in the picture on the left in Figure~6.2.8.
Here the interval $I''_1$ cannot be the whole top edge of the right square face, for otherwise $I''_1$, and hence also $I_1$, would contain an
edge-cutting point of~$H^*_1$, a contradiction.

Without loss of generality, suppose that the index $\ell$ of $m_\ell$ in Lemma~\ref{lem6.2.1} is odd, so that the corresponding L-line segment is the almost horizontal~$H^*_\ell$.
By Lemma~\ref{lem6.2.1}, $H^*_\ell$ exhibits all six types of corner cuts.
Figure~6.2.12 shows part of the top square face of the L-surface as well as a corner cut of~$H^*_\ell$.

\begin{displaymath}
\begin{array}{c}
\includegraphics[scale=0.75]{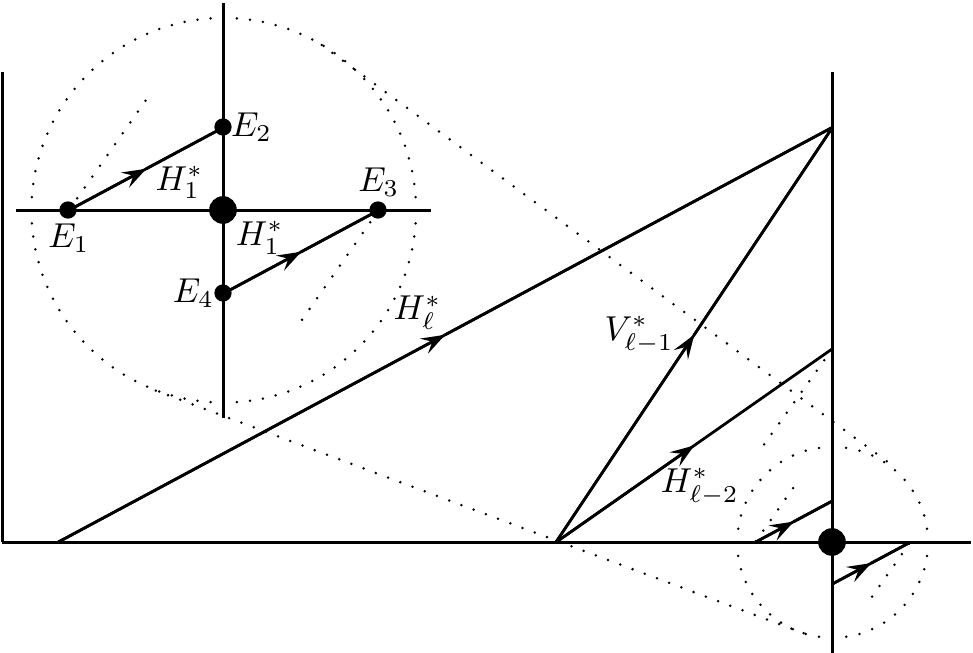}
\vspace{3pt}\\
\mbox{Figure 6.2.12: zigzagging towards the split singularity}
\end{array}
\end{displaymath}

By the vertical same edge cutting property, this corner cut of $H^*_\ell$ intersects a corner cut of $V^*_{\ell-1}$ at a point on a vertical edge of the L-surface.
As shown in Figure~6.2.12, the corner cut of $V^*_{\ell-1}$ is closer to the split singularity, indicated by the big dot.
From the corner cut of~$V^*_{\ell-1}$, we can find a corner cut of $H^*_{\ell-2}$ that is again closer to the split singularity, and so on.
Eventually we arrive at a corner cut of~$H^*_1$.
This is very close to the split singularity; in Figure~6.2.12, the parts within the dotted circles are magnified.
We denote by $E_1$ and $E_2$ the endpoints of the corner cut of $H^*_1$ in the top square face.

An analogous argument can be carried out in the right square face, with corner cuts getting closer to the same split singularity.
We denote by $E_3$ and $E_4$ the endpoints of a corresponding corner cut of $H^*_1$ in this square face.

Our next step is to estimate the distances of these endpoints from the split singularity.
We shall show that they are \textit{exponentially close} to the split singularity.

It is clear from Figure~6.2.12 that the distance of the furthest point on the corner cut of $H^*_\ell$ from the split singularity is less than~$1$.
Since $H^*_\ell$ has slope $\alpha_\ell^{-1}$, it follows that the distance of the furthest point on the corner cut of $V^*_{\ell-1}$ from the split singularity is less
than~$\alpha_\ell^{-1}$.
Since $V^*_{\ell-1}$ has slope $\alpha_{\ell-1}$, it follows that the distance of the furthest point on the corner cut of $H^*_{\ell-2}$ from the split singularity is less
than~$\alpha_\ell^{-1}\alpha_{\ell-1}^{-1}$.
Clearly $E_1$ is the point on the corner cut of $H^*_1$ which is furthest from the split singularity.
Iterating, we conclude that the distance of $E_1$ from the split singularity is less than
\begin{displaymath}
\frac{1}{\alpha_\ell\alpha_{\ell-1}\alpha_{\ell-2}\cdots \alpha_3\alpha_2}.
\end{displaymath}
Similarly, the distance of $E_3$ from the split singularity is less than this same quantity.
Hence the distance $\dist(E_1,E_3)$ between $E_1$ and $E_3$ satisfies the inequality
\begin{equation}\label{eq6.2.40}
\dist(E_1,E_3)<\frac{2}{\alpha_\ell\alpha_{\ell-1}\alpha_{\ell-2}\cdots \alpha_3\alpha_2}.
\end{equation}

The forward $\alpha_1^{-1}$-flow projects the $V^*$-free interval $I_0$ to an interval $I_1$ on some horizontal edges of the L-surface, temporarily represented by $I'_1$ and $I''_1$ as shown in Figure~6.2.8, and containing the split singularity.

Under the assumption \eqref{eq6.2.38}, it follows that
\begin{equation}\label{eq6.2.41}
\length(I'_1)+\length(I''_1)
=\alpha_1\length(I_0)
>\frac{2}{\alpha_2\alpha_3\alpha_4\cdots\alpha_{\ell-1} \alpha_\ell}.
\end{equation}
Combining \eqref{eq6.2.40} and \eqref{eq6.2.41}, we conclude that
\begin{equation}\label{eq6.2.42}
\dist(E_1,E_3)<\length(I'_1)+\length(I''_1).
\end{equation}
Since the points $E_1$ and $E_3$ fall on different sides of the split singularity, the inequality \eqref{eq6.2.42} implies that
$I'_1\cup I''_1$ must contain at least one of these two points, and so an edge-cutting point of~$H^*_1$, clearly contradicting that $I_1$ is $H^*_1$-free. 
Thus $I_1$ does not contain a split singularity, so does not split and therefore is a single interval, and \eqref{eq6.2.39} holds.
\end{proof}

\begin{proof}[Proof of Lemma~\ref{lem6.2.2}]
Lemma~\ref{lem6.2.3} sets up an iterative process.
Suppose that \eqref{eq6.2.38} holds.
Then the forward or reverse $\alpha_1^{-1}$-flow projects $I_0$ to an $H^*_1$-free interval $I_1$ on some horizontal edge of the L-surface, so that $I_1$ does not contain a split singularity, and \eqref{eq6.2.39} holds.

Then the forward or reverse $\alpha_2$-flow projects $I_1$ to~$I_2$, which is a $V^*_2$-free interval on some vertical edge of the L-surface.
With the roles of \eqref{eq6.2.38}, $I_0$ and $I_1$ replaced respectively by \eqref{eq6.2.41}, $I_1$ and $I_2$, an analogous argument shows that $I_2$ does not contain a split singularity, so does not split and is therefore a single interval, and
\begin{equation}\label{eq6.2.43}
\length(I_2)
=\alpha_2\length(I_1)
>\frac{2}{\alpha_3\alpha_4\alpha_5\cdots\alpha_{\ell-1} \alpha_\ell}.
\end{equation}

We now keep repeating this argument.

After a finite number of steps, we obtain an analog of \eqref{eq6.2.41} and \eqref{eq6.2.43}, that $I_{\ell-1}$ is a $V^*_{\ell-1}$-free interval on some vertical edge of the
L-surface, and
\begin{equation}\label{eq6.2.44}
\length(I_{\ell-1})=\alpha_{\ell-1}\cdots\alpha_1\length(I_0)>\frac{2}{\alpha_\ell}.
\end{equation}
We now consider the effect of the forward or reverse $\alpha_\ell^{-1}$-flow on~$I_{\ell-1}$.

Figure~6.2.13 illustrates the interval $I_{\ell-1}$ and its projection by the forward $\alpha_\ell^{-1}$-flow.
Consider the horizontal line that contains the top edge of a square face of the L-surface that contains $I_{\ell-1}$ as part of its left edge.
Suppose that a line of slope $\alpha_\ell^{-1}$ that passes through the top end point of $I_{\ell-1}$ intersects this horizontal line at a point~$A$.
Then a line of the same slope $\alpha_\ell^{-1}$ that passes through the bottom end point of $I_{\ell-1}$ intersects the same horizontal line at a point $B$ which is a distance of precisely $\alpha_\ell\length(I_{\ell-1})>2$ to the right of~$A$.
The interval $AB$ must therefore contain a subinterval $J$ of length $1$ that can be identified with a horizontal edge of the L-surface.
This means that $I_\ell$ must contain a whole edge of the L-surface, and therefore cannot be $H^*_\ell$-free.

\begin{displaymath}
\begin{array}{c}
\includegraphics[scale=0.75]{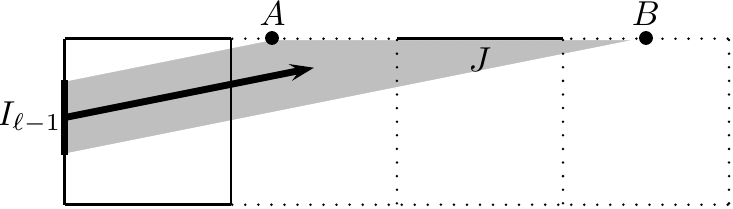}
\vspace{3pt}\\
\mbox{Figure 6.2.13: $I_{\ell-1}$ and its projection by the $\alpha_\ell^{-1}$-flow}
\end{array}
\end{displaymath}

A similar conclusion can be drawn if we use the reverse $\alpha_\ell^{-1}$-flow on~$I_{\ell-1}$.

We therefore arrive at a contradiction.
This contradiction proves that \eqref{eq6.2.38} is false, and completes the proof of Lemma~\ref{lem6.2.2}.
\end{proof}

\begin{remark}
The expression \eqref{eq6.2.44} is very simple.
However, the crucial argument in the proof is that there is no splitting throughout the process $I_0\to I_1\to\ldots\to I_{\ell-1}$ caused by the singularities of the L-surface.
This guarantees that the interval $J$ at the last step can be identified with a horizontal edge of the L-surface and is not split into parts that are identified with different edges of the
L-surface.
\end{remark}

We recall \eqref{eq6.2.9}, that
\begin{displaymath}
\length(V^*)=m_0(1+\alpha^2)^{1/2},
\end{displaymath}
as well as \eqref{eq6.2.14} with $k=\ell$, that
\begin{displaymath}
m_\ell=m_0\prod_{j=1}^\ell\frac{1}{\alpha_j}.
\end{displaymath}
It follows that
\begin{displaymath}
\length(V^*)=(1+\alpha^2)^{1/2}m_\ell\alpha_1\alpha_2\alpha_3\cdots\alpha_{\ell-1}\alpha_\ell,
\end{displaymath}
which can be rewritten in the form
\begin{equation}\label{eq6.2.45}
\frac{2}{\alpha_1\alpha_2\alpha_3\cdots\alpha_{\ell-1}\alpha_\ell}
=\frac{2(1+\alpha^2)^{1/2}m_\ell}{\length(V^*)}.
\end{equation}

Recall that the integer~$U$, given by \eqref{eq6.2.11}, is an upper bound of the continued fraction digits of the badly approximable number~$\alpha$.
It then follows from \eqref{eq6.2.37}, \eqref{eq6.2.45} and Lemma~\ref{lem6.2.1} that every $V^*$-free vertical interval $I_0$ on a vertical edge of the
L-surface satisfies
\begin{equation}\label{eq6.2.46}
\length(I_0)\le\frac{16U^6}{\length(V^*)}.
\end{equation}

In general, $V^*$ can be taken to be an arbitrary segment of the L-line~$V$, so the inequality \eqref{eq6.2.46} proves superdensity under the condition that $V$ has a slope $\alpha$ which is a badly approximable number with even continued fraction digits.
This completes the proof of Theorem~\ref{thm6.1.1}.
\end{proof}

%%%%%%%%%%
%
% SECTION 6.3
%
%%%%%%%%%%

%\subsection{Section deleted}\label{sec6.3}

\setcounter{subsection}{3}

%%%%%%%%%%
%
% SECTION 6.4
%
%%%%%%%%%%

\subsection{Generalization of Theorem~\ref{thm6.1.1}}\label{sec6.4}

We now switch from the L-surface to a large class of finite polysquare translation surfaces, and show how we can find superdense geodesics.

A finite polysquare region $P$ is an arbitrary connected, but not necessarily simply-connected, polygon on the plane tiled with finitely many closed unit squares, called the \textit{atomic squares} or \textit{squares faces} of~$P$, such that the following conditions are satisfied:

(i)
Any two atomic squares in $P$ either are disjoint, or intersect at a single point, or have a common edge.

(ii)
Any two atomic squares in $P$ are joined by a chain of atomic squares where any two neighbors in this chain have a common edge.

Given a finite polysquare region~$P$, we can convert it into a finite polysquare translation surface $\PPP$ by identifying pairs of parallel boundary edges with inward normals in opposite directions.
It is equipped with flat metric, so it is a Riemann surface, with possible conical singularity or singularities, where every square face has zero curvature.
The total curvature $2\pi\chi(S)$ in the Gauss--Bonnet formula, where $\chi(S)$ is the Euler characteristic of the polysquare translation surface~$S$, is concentrated in the finitely many conical singularities.
Geodesic flow on such a surface is then $1$-direction geodesic flow.

The simplest boundary pairing comes from perpendicular translation as in the case of the L-surface, first introduced in Section~\ref{sec6.1}.

A finite polysquare surface which is not a translation surface is the cube surface shown in Figure~6.4.1.
The boundary pairing, involving non-parallel edges, is more complicated.
As on the L-surface, geodesic flow on the cube surface has singularities at the vertices, so it is non-integrable.
It is a $4$-direction flow, as seen in the picture on the right in Figure~6.4.1.

\begin{displaymath}
\begin{array}{c}
\includegraphics[scale=0.85]{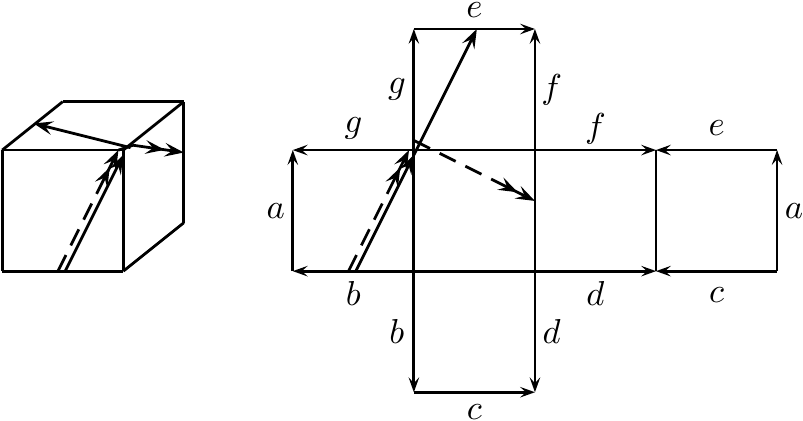}
\vspace{3pt}\\
\mbox{Figure 6.4.1: orbit singularity of geodesic flow on the cube surface}
\\
\mbox{leading to perpendicular directions on the net}
\end{array}
\end{displaymath}

Figure~6.4.2 illustrates a trick to reduce this $4$-direction geodesic flow on the cube surface to a $1$-direction flow on a $4$-times larger surface.

\begin{displaymath}
\begin{array}{c}
\includegraphics[scale=0.85]{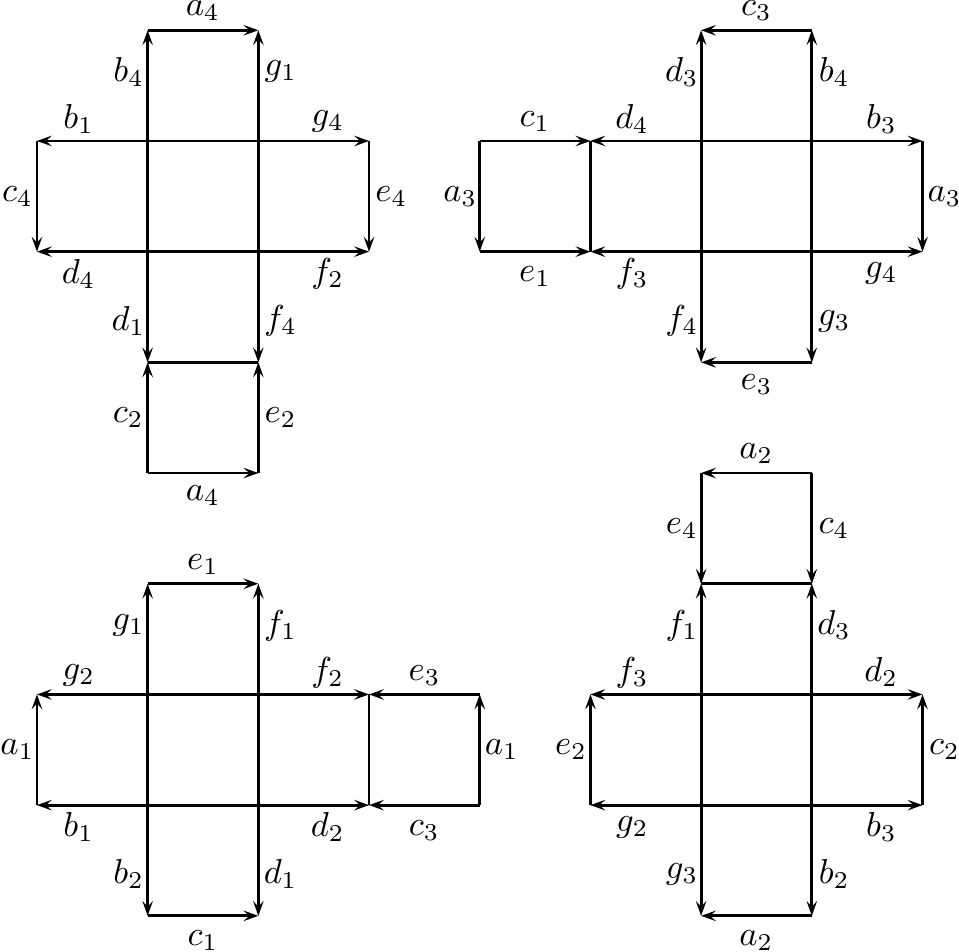}
\vspace{3pt}\\
\mbox{Figure 6.4.2: translation surface which is a $4$-fold covering of the cube surface}
\end{array}
\end{displaymath}

We take four $90$-degree rotated copies of the net in~Figure 6.4.1, and glue them together by making some specific edge-identifications between the different copies.
Figure~6.4.2 shows a translation surface obtained by gluing together the boundary edges with identical labels, which are parallel with inward normals in opposite directions, mapped to each other via translation.
For example, the edge $b_2$ in the lower-left copy is identified with the edge $b_2$ in the lower-right copy of the net of the cube surface, and they are parallel vectors.
And so on. 
This surface is a $4$-fold covering of the cube surface, and its geodesics form a $1$-direction flow.
This $4$-copy construction works for any polysquare surface with $4$-direction geodesic flow.
Hence it suffices to study $1$-direction flow.

Similar $4$-copy construction works for any billiard in a polysquare polygon.
We refer to the polysquare polygon in Figure~6.4.3 as the \textit{snake}.
Figure~6.4.3 shows how the snake billiard on the left is converted into a $1$-direction flow on the right, where the $4$-times larger polysquare surface obtained by gluing together $4$ copies of the snake via iterated reflection across sides, called \textit{unfolding}, and using the given boundary pairing which does not contain perpendicular pairs, and the result is a translation surface.
We refer to this polysquare surface as the \textit{snake-cross surface}.

\begin{displaymath}
\begin{array}{c}
\includegraphics[scale=0.75]{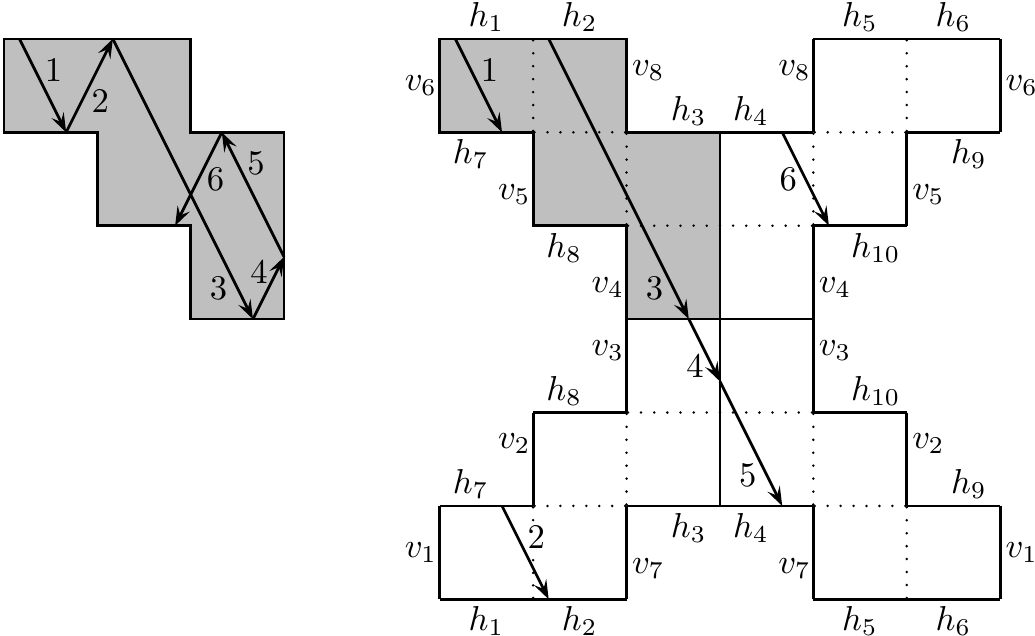}
\vspace{3pt}\\
\mbox{Figure 6.4.3: snake billiard and unfolding to obtain the snake-cross surface}
\end{array}
\end{displaymath}

The polysquare translation surface on the left in Figure~6.4.4 has a missing square face, illustrated by the shaded region.

\begin{displaymath}
\begin{array}{c}
\includegraphics[scale=0.75]{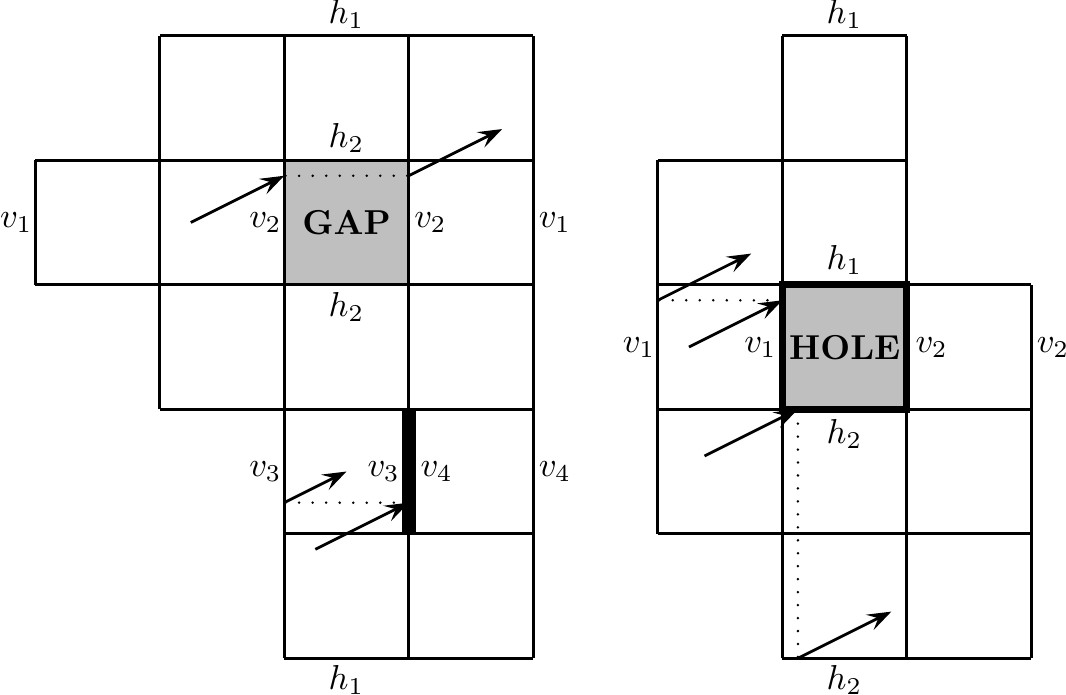}
\vspace{3pt}\\
\mbox{Figure 6.4.4: surfaces with gap, wall or hole}
\end{array}
\end{displaymath}

Note that the edge pairings make this missing part behave like a \textit{gap}.
For instance, when a geodesic hits the edge $v_2$ on the left of this gap, it jumps to the corresponding point on the identified edge $v_2$ on the right of this gap, then continues on its way at the same slope.
It also has a double edge, illustrated by the thick line segment, that behaves like a \textit{wall}.
For instance, when a geodesic hits the edge $v_3$ on the left side of this wall, it jumps back to the corresponding point on the identified edge $v_3$ further back, then continues on its way at the same slope.
Note that the squares on the top row form a horizontal street of length~$3$, as do the squares on the second row from the top and the squares on the third row from the top.
Furthermore, each of the squares on the second row from the bottom forms a horizontal street of length~$1$, while the squares on the bottom row form a horizontal street of length~$2$.
Thus this surface has two horizontal streets of length~$1$, one horizontal street of length~$2$, and three horizontal streets of length~$3$. 
It also has one vertical street of length~$1$, one vertical street of length~$3$, one vertical street of length~$4$, and one vertical street of length~$5$. 
So the street-LCM, the least common multiple of all the street lengths, is $3\times4\times5=60$.

The polysquare translation surface on the right in Figure~6.4.4 also has a missing square face, illustrated by the shaded region.
But it behaves like a thick wall in both the horizontal and vertical direction.
We may call it a \textit{hole}.

Note that some of the boundary pairings are omitted from the pictures, since they come from the simplest form of perpendicular translation.

\begin{remark}
We use the words \textit{gap}, \textit{wall} and \textit{hole} purely for convenience.
They do not have any formal meaning, and the precise details are given by the edge identification process.
Indeed, there can be missing squares that can be hybrid-gap-holes, in the sense that it may act like a gap in the horizontal direction and like a hole in the vertical direction.
Furthermore, there may be edge identifications that make missing squares or walls far more complex than we have described so far.
The important point always to bear in mind is that the edge pairings are what matter.
\end{remark}

As Figure~6.4.2 shows, the street-LCM of the translation surface of the cube surface is~$4$, since every street has length~$4$.

Similarly, the street-LCM of the surface in Figure 6.4.3 is also~$4$, since every street has length $2$ or~$4$.

We have the following generalization of Theorem~\ref{thm6.1.1}.

\begin{thm}\label{thm6.4.1}
Let $\PPP$ be an arbitrary finite polysquare translation surface.
Let $\alpha$ be a badly approximable number with continued fraction expansion
\begin{equation}\label{eq6.4.1}
\alpha=[a_0;a_1,a_2,a_3,\ldots]=a_0+\frac{1}{a_1+\frac{1}{a_2+\frac{1}{a_3+\cdots}}}
\end{equation}
such that for every $i\ge0$, the digit $a_i$ is divisible by the street-LCM of~$\PPP$.
Then any half-infinite $1$-directional geodesic with slope $\alpha$ exhibits superdensity on~$\PPP$.
\end{thm}

Using unfolding such as that illustrated in Figure~6.4.3, we can show that a $4$-directional billiard trajectory in a polysquare region can be reduced to a $1$-directional geodesic on a translation surface which can be viewed as a $4$-copy version of the polysquare region.
In particular, the conclusion of Theorem~\ref{thm6.4.1} applies also to billiards in a finite polysquare region.

Let $\alpha>0$ be any badly approximable number satisfying \eqref{eq6.4.1}, with the extra restriction that every digit $a_i$ is divisible by~$4$. 
Theorem~\ref{thm6.4.1} implies that any geodesic on the cube surface with slope equal to this $\alpha$ exhibits superdensity. 
Similarly, any billiard trajectory in the snake region given in Figure~6.4.3 with initial slope equal to this $\alpha$ exhibits superdensity. 

Let $\alpha>0$ be any badly approximable number satisfying \eqref{eq6.4.1}, with the extra restriction that every digit $a_i$ is divisible by~$60$.
Theorem~\ref{thm6.4.1} implies that any geodesic on the surface in the picture on the left in Figure~6.4.4 with slope equal to this $\alpha$ exhibits superdensity. 

%%%%%%%%%%%%%%%%%%%%

\begin{proof}[Proof of Theorem~\ref{thm6.4.1}]
The proof is a fairly straightforward adaptation of the proof of Theorem~\ref{thm6.1.1}. 
The key concept of \textit{street} in the general case means a \textit{maximal} size strip of $\ell$ consecutive square faces arranged horizontally or vertically.
We call $\ell$ the length of the street.
The concept of \textit{street corner}, such as those in Figures 6.2.2 and 6.2.3 for the L-surface, in the general case means the intersection of a horizontal side of a horizontal street and a vertical side of a vertical street. 

The missing ingredient is a suitable analog of Lemma~\ref{lem6.2.1} in the general case.
However, the discussion between the statement of Lemma~\ref{lem6.2.1} and its proof sets the tone of this discussion.

The reader may recall that our earlier discussion in connection with the proof of Lemma~\ref{lem6.2.1} involves working out ancestors of almost vertical units and almost horizontal units by brute force, where we diligently list ancestors at every stage and make use of the extension rule.
Whereas for any \textit{given} finite polysquare translation surface~$\PPP$, we can repeat this brute force approach with a possibly extremely tedious exercise, in the general case, we have no precise information on the structure of the polysquare translation surface~$\PPP$ to even contemplate such a crude approach.
As it turns out, the solution is relatively simple.

Instead of listing all possible almost vertical units, we classify them into two types.
Consider a given square face of the polysquare translation surface, and consider almost vertical units that start from the bottom edge of this square face.
We say that the almost vertical unit is of type $\uparrow$ in the square face if it starts from the bottom edge of the square face and arrives at the top edge of the same square face without intersecting a vertical side along the way, as shown in the picture on the left in Figure~6.4.5.
We say that the almost vertical unit is of type $\nuparrow$ in the square face if it starts from the bottom edge of the square face but then hits the right edge of the same square face, as shown in the picture on the right in Figure~6.4.5.
The unit then continues into the square face with left edge identified with the right edge of this square face.

\begin{displaymath}
\begin{array}{c}
\includegraphics[scale=0.75]{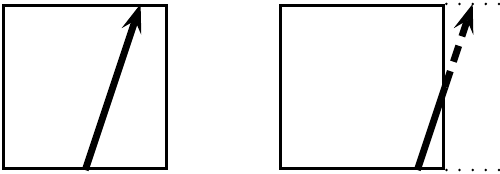}
\vspace{3pt}\\
\mbox{Figure 6.4.5: almost vertical units $\uparrow$ and $\nuparrow$ in the square face}
\end{array}
\end{displaymath}

Likewise, instead of listing all possible almost horizontal units, we classify them into two types.
Consider a given square face of the polysquare translation surface, and consider almost horizontal units that start from the left edge of this square face.
We say that the almost horizontal unit is of type $\rightarrow$ in the square face if it starts from the left edge of the square face and arrives at the right edge of the same square face without intersecting a horizontal side along the way, as shown in the picture on the left in Figure~6.4.6.
We say that the almost vertical unit is of type $\nrightarrow$ in the square face if it starts from the left edge of the square face but then hits the top edge of the same square face, as shown in the picture on the right in Figure~6.4.6.
The unit then continues into the square face with bottom edge identified with the top edge of this square face.

\begin{displaymath}
\begin{array}{c}
\includegraphics[scale=0.75]{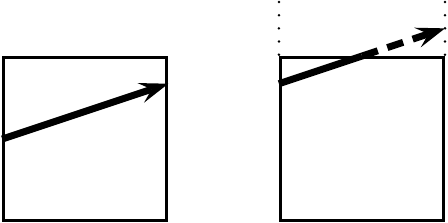}
\vspace{3pt}\\
\mbox{Figure 6.4.6: almost horizontal units $\rightarrow$ and $\nrightarrow$ in the square face}
\end{array}
\end{displaymath}

Recall that Lemma~\ref{lem6.2.1} concerns exhibiting corner cuts.
It is clear that a corner cut is a unit of type $\nuparrow$ or~$\nrightarrow$, but not one of type $\uparrow$ or~$\rightarrow$.
We shall also start our ancestor process by assuming that the first unit is part of a much longer geodesic, so that we may use the extension rule.

First we define the $\PPP$-distance between any two distinct square faces $S_1$ and $S_2$ in the polysquare translation surface~$\PPP$.
We say that their $\PPP$-distance is $1$ if they belong to the same horizontal street or vertical street.
Otherwise, we consider a shortest sequence of alternate horizontal and vertical streets such that the first contains~$S_1$, the last contains~$S_2$, and any two consecutive streets in the sequence intersect.
Then the length of this sequence is the $\PPP$-distance between $S_1$ and~$S_2$.

In Figure~6.4.7, the two square faces $S_1$ and $S_2$ have $\PPP$-distance~$2$.
They are not on the same street.
In the picture on the left, the horizontal street containing $S_1$ intersects the vertical street containing~$S_2$, although the vertical street containing $S_1$ may not necessarily intersect the horizontal street containing~$S_2$.
In the picture on the right, the vertical street containing $S_1$ intersects the horizontal street containing~$S_2$.

\begin{displaymath}
\begin{array}{c}
\includegraphics[scale=0.75]{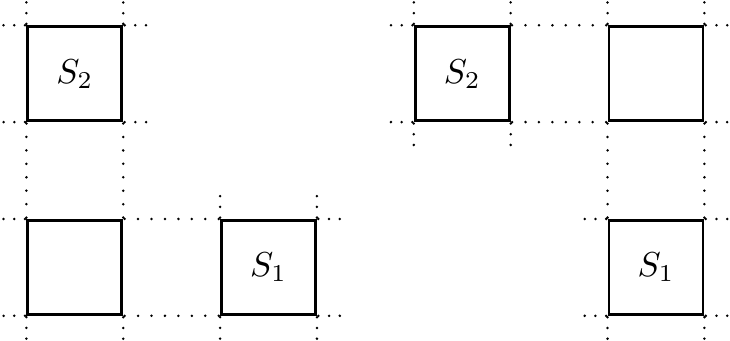}
\vspace{3pt}\\
\mbox{Figure 6.4.7: square faces $S_1$ and $S_2$ with $\PPP$-distance $2$}
\end{array}
\end{displaymath}

The $\PPP$-diameter is then the maximum $\PPP$-distance between any two square faces of~$\PPP$.

The shortline process discussed earlier replaces a detour crossing with a shortcut.
Here we study the reverse process.
Given a finite almost vertical (resp. horizontal) geodesic, we can break it up into a number of \textit{whole} units, and possibly two \textit{fractional} units at the end.
Each whole unit is the shortcut of an almost horizontal (resp. vertical) detour crossing.
Each of the fractional units, if extended to a full unit, is also the shortcut of an almost horizontal (resp. vertical) detour crossing.
For these, we shorten the almost horizontal (resp. vertical) detour crossing by the same fraction and at the appropriate end.
We then take the union of these almost horizontal (resp. vertical) detour crossings and fractional almost horizontal (resp. vertical) detour crossings.
The union is a finite almost horizontal (resp. vertical) geodesic.
We call this the \textit{ancestor geodesic} of the original almost vertical (resp. horizontal) geodesic.

\begin{lem}\label{lem6.4.1}
Suppose that $\PPP$ is a finite polysquare translation surface with $\PPP$-diameter~$K$.
Suppose further that $V$ is a $1$-direction almost vertical geodesic of slope $\alpha$ which is an irrational number given by \eqref{eq6.4.1}, where every continued fraction digit is a positive integer multiple of the street-LCM of~$\PPP$.
Let $\VVV$ (resp. $\HHH$) denote a finite almost vertical (resp. horizontal) geodesic made up of $4$ successive detour crossings of the $i$-generation shortline of $V$ for some even (resp. odd) integer~$i\ge2K$.
Then in every square face of~$\PPP$, the $2K$-generation ancestor geodesic of $\VVV$ (resp. $\HHH$) gives rise to an almost vertical unit of type
$\nuparrow$ (resp. horizontal unit of type $\nrightarrow$) in the square face.
\end{lem}

We shall only prove Lemma~\ref{lem6.4.1} for~$\VVV$, as the argument for $\HHH$ is similar.
We need the following.

\begin{replacement}
Replace a unit by another unit in the same detour crossing.
\end{replacement}

We first prove the following intermediate result.

\begin{lem}\label{lem6.4.2}
Under the hypotheses of Lemma~\ref{lem6.4.1}, suppose that $S'$ and $S''$ are two square faces of $\PPP$ that lie on the same horizontal or vertical street.
Suppose that $A$ is an almost vertical unit of type $\uparrow$ or $\nuparrow$ in~$S'$.
Suppose further that the extension rule and the replacement rule apply.
Then $S''$ contains an almost vertical unit of type $\nuparrow$ in $S''$ that is a $2$-generation ancestor of $A$ or some almost vertical unit of type $\uparrow$ or
$\nuparrow$ replacing~$A$.
\end{lem}

\begin{proof}
(i) Suppose that $S'$ and $S''$ lie on the same horizontal street.
Consider the almost horizontal detour crossing for which $A$ is the shortcut.
This detour crossing must contain a fractional part of an almost horizontal unit of type $\nrightarrow$ that intersects the starting point of~$A$.
This fractional unit intersects the square face~$S'$.
Since the extension rule applies, we may assume that this fractional unit is extended to a full unit, as shown in Figure~6.4.8 if $A$ is of type
$\uparrow$ in $S'$ and in Figure~6.4.9 if $A$ is of type $\nuparrow$ in~$S'$.
On the other hand, this almost horizontal detour crossing also gives rise to a unit of type $\rightarrow$ in every other square face that is in the same horizontal street that contains~$S'$.

Note that each of these almost horizontal units $\nrightarrow$ and $\rightarrow$ is a $1$-generation ancestor of~$A$, with the end point intersecting the right edge of the square face.
This almost horizontal unit is the shortcut of an almost vertical detour crossing that contains a fractional part of an almost vertical unit of type
$\nuparrow$ that intersects the end point of the almost horizontal unit under consideration.
Since the extension rule applies, we may assume that this fractional unit is extended to a full unit, as shown in Figure~6.4.8 by the bold arrows (note that we have not inserted these in Figure~6.4.9), and this unit is a $2$-generation ancestor of~$A$.

\begin{displaymath}
\begin{array}{c}
\includegraphics[scale=0.75]{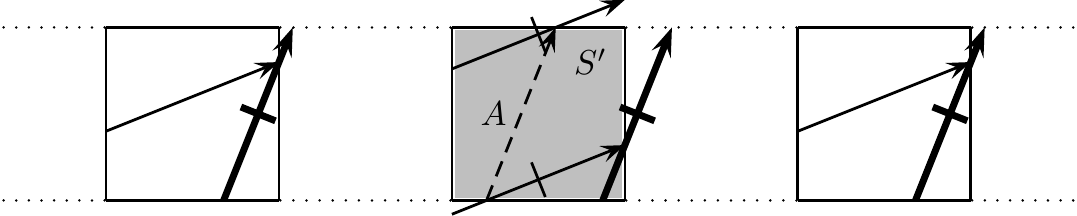}
\vspace{3pt}\\
\mbox{Figure 6.4.8: working along a horizontal street starting with a unit of type $\uparrow$}
\end{array}
\end{displaymath}
\begin{displaymath}
\begin{array}{c}
\includegraphics[scale=0.75]{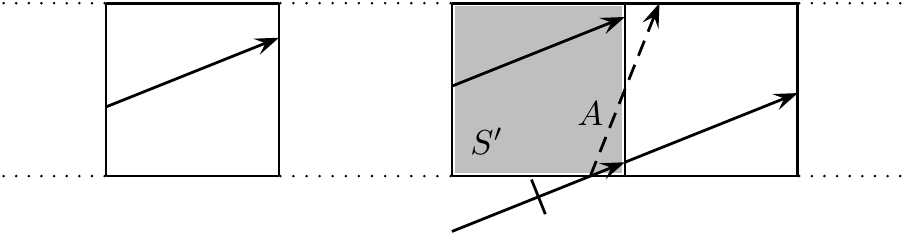}
\vspace{3pt}\\
\mbox{Figure 6.4.9: working along a horizontal street starting with a unit of type $\nuparrow$}
\end{array}
\end{displaymath}

In particular, each of the two square faces $S'$ and $S''$ contains such a $2$-generation ancestor of~$A$.

(ii) Suppose that $S'$ and $S''$ lie on the same vertical street.
Note that the almost vertical unit $A$ in the square face $S'$ is part of an almost vertical detour crossing of the vertical street containing~$S'$.
This detour crossing has an almost vertical unit in every square face in this vertical street.
Thus the square face $S''$ contains an almost vertical unit $B$ that is in the same almost vertical detour crossing as~$A$.
Using the replacement rule, we replace $A$ by~$B$.
Starting with $B$ in $S''$ and considering the horizontal street containing~$S''$, it follows from (i) that $S''$ contains a $2$-generation ancestor of~$B$.
\end{proof}

\begin{proof}[Proof of Lemma~\ref{lem6.4.1}]
Let $A$ be an almost vertical unit in the two middle detour crossings in $\VVV$ that intersects the bottom edge of some square face $S_1$ in~$\PPP$.
Since any other square face $S$ in $\PPP$ has $\PPP$-distance at most $K$ from~$S_1$, there exists a sequence of square faces $S_2,\ldots,S_L$, where $L<K$, such that any consecutive pair of square faces lie on the same horizontal or vertical street, and such that $S_1$ and $S_2$ lie on the same horizontal or vertical street, and $S_L$ and $S$ lie on the same horizontal or vertical street.
Applying Lemma~\ref{lem6.4.2} iteratively at most $K$ times gives us the desired result.
It remains to justify the use of the extension rule and replacement rule in Lemma~\ref{lem6.4.2}.

The extension rule means that ancestor units that are only fractional in the detour crossing for which $A$ is the shortcut are counted in full.
Such an ancestor unit is also part of the ancestry of an almost vertical unit in $\VVV$ adjoining~$A$.
To make sure that the unit $\nuparrow$ at the end of the proof is a \textit{genuine} $4$-generation ancestor of some unit in~$\VVV$, and not there merely as a consequence of the extension rule, we start with a geodesic $\VVV$ with four detour crossings and pick a unit $A$ in the two middle detour crossings.
The two detour crossings at either end of $\VVV$ then give us ample cover.
These extra detour crossings also justify the use of the replacement rule.
\end{proof}

Lemma~\ref{lem6.4.1} is precisely the generalization of Lemma~\ref{lem6.2.1} that we need to complete the proof of Theorem~\ref{thm6.4.1}.
We can now prove Theorem~\ref{thm6.4.1} with a fairly straightforward adaptation of the \textit{magnification process of empty intervals}, the basic idea of the proof of Theorem~\ref{thm6.1.1} for the L-surface, as long as the rule for magnification in a polysquare is followed.
\end{proof}

The restriction that the continued fraction digits of $\alpha$ are divisible by the street-LCM in Theorem~\ref{thm6.4.1} implies that the superdense slopes $\alpha$ satisfying this digit condition form a nowhere dense set in the unit interval.
There is, however, a simple geometric trick to extend this set of superdense slopes to a dense set in the unit interval for every finite polysquare translation surface.

The idea is that every finite polysquare translation surface generates infinitely many new finite polysquare translation surfaces as follows.

Consider, for instance, the L-surface of $3$ square faces.
By drawing the two diagonals on each one of the $3$ square faces, \textit{i.e.}, putting a $\times$ in every square face, we obtain a new polysquare
translation surface with $6$ smaller square faces, each with half the area; see Figure~6.4.10.
The genus remains the same.
By using the boundary pairing in the picture on the right in Figure~6.4.10, we obtain the so-called \textit{diagonal subdivision} surface of the
L-surface, or DS-L-surface.
Here the DS-L-surface is not in the usual horizontal-vertical position, but rotating the picture by $45$ degrees does not alter the nature of the question at hand.

\begin{displaymath}
\begin{array}{c}
\includegraphics[scale=0.75]{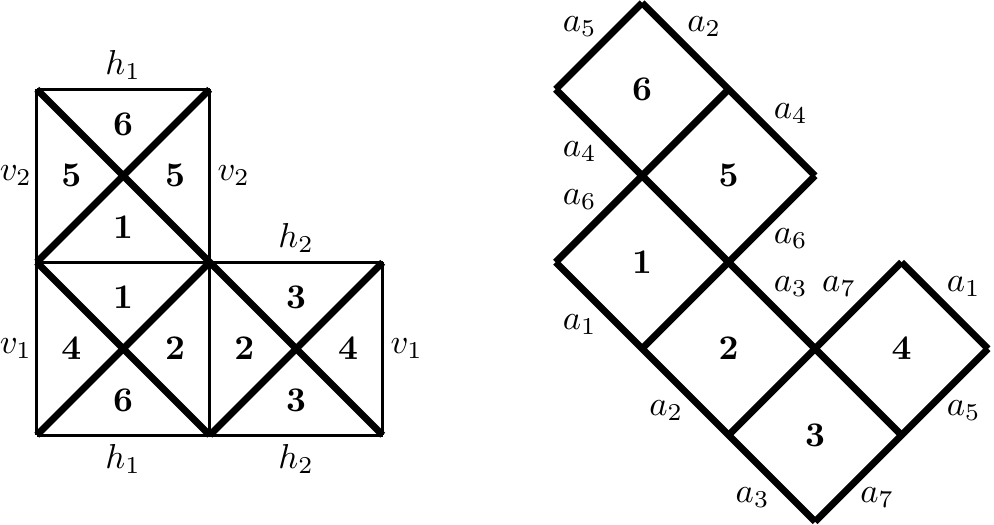}
\vspace{3pt}\\
\mbox{Figure 6.4.10: L-surface and DS-L-surface}
\end{array}
\end{displaymath}

The concept of diagonal decomposition into smaller squares has a far-reaching generalization.
For example, Figure 6.4.11 below shows the $(k,1)$-decomposition of a square face in the special cases $k=2$ and $k=3$.
The full generalization comes from choosing relatively prime integers $k$ and $\ell$ satisfying $1\le\ell<k$.

\begin{displaymath}
\begin{array}{c}
\includegraphics[scale=0.75]{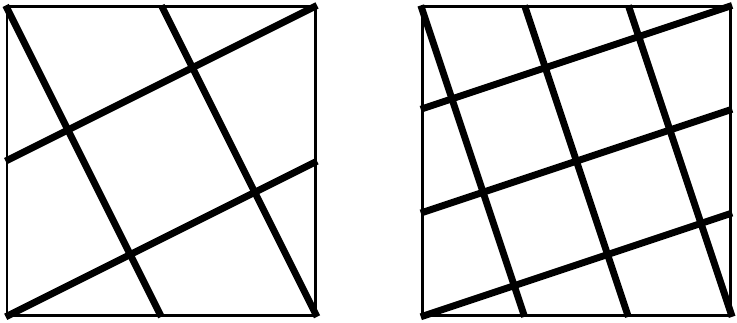}
\vspace{3pt}\\
\mbox{Figure 6.4.11: $(2,1)$-decomposition and $(3,1)$-decomposition}
\end{array}
\end{displaymath}

Applying the corresponding $(k,\ell)$-decomposition on every square face of of an arbitrary polysquare translation surface $\PPP$ with $1$-direction geodesic flow, we obtain a polysquare translation surface $(k,\ell)$-S-$\PPP$, where S stands for \textit{subdivision}.
Here we emphasize the fact that the slope $\ell/k$ can be \textit{any} rational number between $0$ and~$1$, and of course the rationals form a dense set.
The last step is to apply Theorem~\ref{thm6.4.1} for an arbitrary polysquare translation surface $(k,\ell)$-S-$\PPP$.
Thus we obtain the following result.

\begin{cor641}
Let $\PPP$ be any finite polysquare translation surface with $1$-direction geodesic flow.
Then the set of slopes for which every infinite geodesic flow of this slope exhibits superdensity on $\PPP$ is dense in the unit interval.
\end{cor641}

%%%%%%%%%%%%%%%%%%%%

%%%%%%%%%%
%
% SECTION 6.5
%
%%%%%%%%%%

\subsection{Time-quantitative density in a square-maze}\label{sec6.5}

Here we describe a large class of \textit{infinite} polysquare translation surfaces for which geodesic flow exhibits density.
As usual, we are interested in the time-quantitative aspects of density.

We call an infinite polysquare translation surface a \textit{square-maze translation surface} if the lengths of the horizontal and vertical streets are uniformly bounded.
Let $\ell\ge2$ be any integer.
Then we call a square-maze translation surface an \textit{$\ell$-square-maze translation surface} if every street, whether horizontal or vertical, has length at most~$\ell$, and there is a street that has length equal to~$\ell$.
We choose an arbitrary square face of~$\PPP$, choose any of its $4$ corner points, and refer to this particular corner point as the origin~$\bzero$.

\begin{remark}
If geodesic flow on a square-maze surface or region is a $4$-direction flow, like in the case of billiards, then of course we can apply the standard trick of a $4$-copy construction as in Figure~6.4.2 or the trick of unfolding like in Figure~6.4.3.
Such a geometric trick converts the original problem to an equivalent problem of geodesic flow on a ($4$-copy) square-maze translation surface with
$1$-direction flow.
\end{remark}

The class of square-maze translation surfaces forms a very rich family of infinite surfaces.
Note that for a fixed positive integer~$\ell$, there are $\ell$-square-maze translation surfaces with completely different \textit{growth-rate of the neighborhood}, \textit{i.e.}, the rate of growth of the number of square faces that are at a $\PPP$-distance at most $N$ from a given square face as a function of~$N$.
Nevertheless, somewhat surprisingly, this growth-rate of the neighborhood does not show up in our time-quantitative density result Theorem~\ref{thm6.5.1} below.

Along the way, we give examples of square-maze translation surfaces that exhibit linear, or quadratic, or cubic, or exponential growth-rate of the neighborhood.

We call our first example an \textit{infinite shark}; see Figure~6.5.1 below.
The alternating vertical walls, from above and below, resemble shark-teeth, explaining the name.
To obtain an infinite flat surface in Figure~6.5.1, we use the simplest perpendicular boundary pair identification: horizontal translation for vertical edges and vertical translation for horizontal edges.
In the resulting surface every street has length~$2$, so it is a $2$-square-maze translation surface.

\begin{displaymath}
\begin{array}{c}
\includegraphics[scale=0.75]{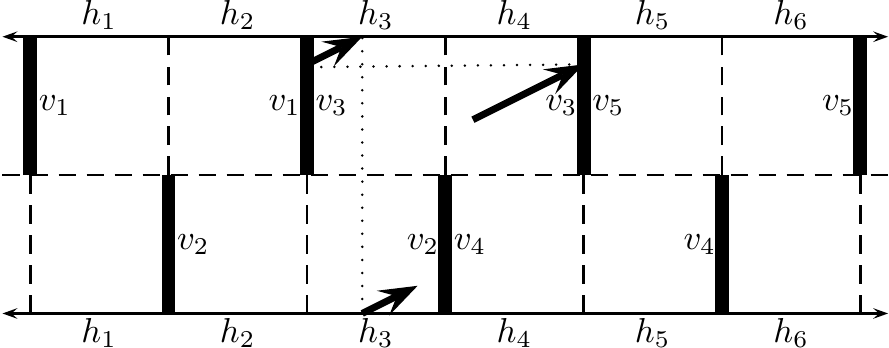}
\vspace{3pt}\\
\mbox{Figure 6.5.1: infinite shark surface as a $2$-square-maze}
\\
\mbox{and $1$-direction geodesic flow}
\end{array}
\end{displaymath}

For $\ell\ge3$, we have the freedom to change the gaps between the teeth of the shark.
This gives an uncountable set of \textit{aperiodic} $\ell$-square-maze translation surfaces with linear growth-rate of the neighborhood similar to that in Figure~6.5.1.

\begin{displaymath}
\begin{array}{c}
\includegraphics[scale=0.75]{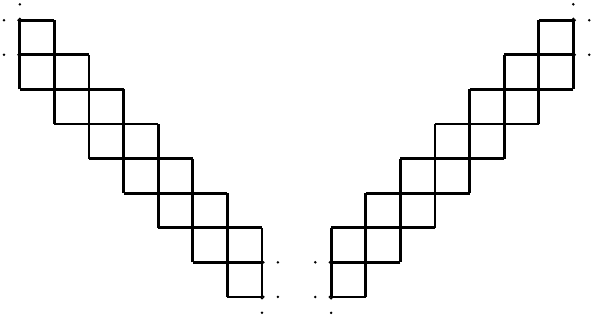}
\vspace{3pt}\\
\mbox{Figure 6.5.2: infinite down-staircase and infinite up-staircase}
\end{array}
\end{displaymath}

By the way, the set of all $2$-square-maze translation surfaces is uncountable.
To justify this claim we can introduce the infinite staircase in Figure~6.5.2 below.
Moving from left to right the staircase goes down, so we refer to it as a \textit{down-staircase}.
Reflecting it across a vertical line we obtain an infinite \textit{up-staircase}.

Figure~6.5.2 shows the \textit{infinite staircase region}, which is a $45$-degree \textit{tilted tower} of infinitely many $2\times1$ rectangles.
The \textit{infinite staircase surface} is obtained from the infinite region by the simplest boundary identification: pairs of vertical boundary edges are identified by horizontal translation, and pairs of horizontal boundary edges are identified by vertical translation.

A $2$-square-maze translation surface is intuitively an \textit{infinite snake}, where we have an infinite degree of freedom of going up-or-down and
left-or-right and by using walls if necessary.
For example, start with any finite down- or up-staircase surface, glue to it any finite shark surface, next glue to it any finite down- or up-staircase surface, next glue to it any finite shark surface, and so on.
This simple construction already provides an uncountable set of $2$-square-maze translation surfaces.
Of course every $2$-square-maze translation surface has a linear growth-rate of the neighborhood.

Figure~6.5.3 shows a double-periodic $3$-square-maze translation surface, which clearly exhibits quadratic growth-rate of the neighborhood.
Here the building blocks are $3\times3$ squares with holes in the middle, and these blocks are glued together by $1\times1$ squares in such a way
that they form a $2$-dimensional lattice (somewhat like~$\Zz^2$), creating further holes.
We use gray color for the holes, which we also indicate by the letter~\textbf{H}, and the boundary pair identification, some of which are indicated, ensures that these holes are \textit{no-go zones} as opposed to the gap in Figure~6.4.4.
Here every right vertical edge of a hole is identified with the left vertical edge (on the same horizontal street) of the next hole to the right, and every top horizontal edge of a hole is identified with the bottom horizontal edge (on the same vertical street) of the next hole above.
Thus we obtain a translation surface where every street has length~$3$.

\begin{displaymath}
\begin{array}{c}
\includegraphics[scale=0.75]{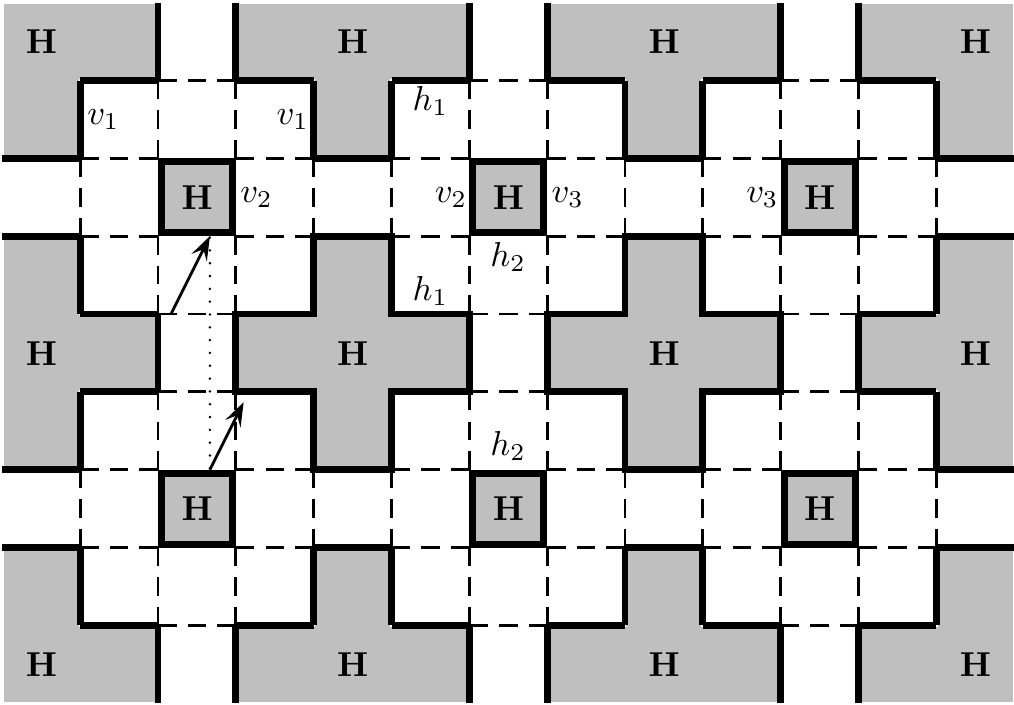}
\vspace{3pt}\\
\mbox{Figure 6.5.3: $3$-square-maze with holes}
\end{array}
\end{displaymath}

\begin{remark}
All missing squares in this section are holes.
\end{remark}

If $\ell$ is substantially larger than~$3$, then we have all the freedom to change the size of the building blocks.
For example, we can replace the $3\times3$ square with an $8\times8$ square, replace the hole size to $4\times4$, and locate the hole inside the $8\times8$ square arbitrarily.
This way we can construct an uncountable set of \textit{aperiodic} $\ell$-square-maze translation surfaces with quadratic growth-rate of the neighborhood similar to that in Figure~6.5.3.  

Another way to construct an uncountable set of aperiodic square-maze translation surfaces is based on~Figure 6.5.4.
Type $+$ on the left shows $4$ unit size gray squares (holes) inside a $4\times4$ big square such that every row and column has one gray square and the distance between any two gray squares is at least one.
Type $-$ on the right shows another configuration of $4$ gray squares with the same property.
Divide the plane into $4\times4$ squares, and in each one place a type $+$ or a type $-$ configuration arbitrarily.
The cardinality of the number of different infinite configurations is the cardinality of $2^{\Zz^2}$.
Thus we obtain an uncountable family of aperiodic $j$-square-maze translation surfaces where $j\le6$.
They all exhibit quadratic growth-rate of the neighborhood.

\begin{displaymath}
\begin{array}{c}
\includegraphics[scale=0.75]{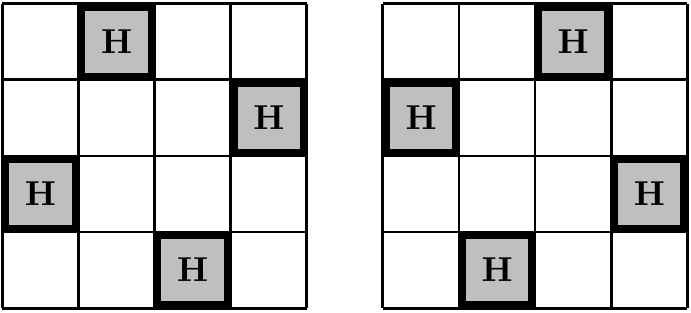}
\vspace{3pt}\\
\mbox{Figure 6.5.4: type $+$ on the left and type $-$ on the right}
\end{array}
\end{displaymath}

This idea can be generalized to larger square blocks.
Let $p>2$ be a prime, and consider a $p\times p$ block of unit size squares.
For ease of description, suppose that the bottom left vertex is $(0,0)$ and the top right vertex is $(p,p)$, and that for every $0\le i,j<p$, $S(i,j)$ denotes the unit size square with bottom left vertex $(i,j)$.
Let $q$ be a prime distinct from~$p$.
For every $0\le i<p$, let $0\le y_i<p$ be the unique solution of the congruence $y_i\equiv qi\bmod{p}$.
Now consider a building block where each square $S(i,y_i)$, $0\le i<p$, is a unit size gray square (hole).
Then every row and column of the $p\times p$ block has precisely one gray unit size square.
We can call this the type $(p,q)$ configuration.
On the other hand, we can repeat the same argument with another prime $q'$ different from both $p$ and~$q$, and obtain a type $(p,q')$ configuration.
Divide the plane into $p\times p$ squares, and in each one place a type $(p,q)$ or a type $(p,q')$ configuration arbitrarily.
Thus we obtain an uncountable family of aperiodic $j$-square-mazes where $j\le2p-2$.

To give an example of a polysquare surface, not necessarily a translation surface, for which the growth-rate of the neighborhood is cubic, we consider first an infinite polycube region, \textit{i.e.}, cube tiled solid region, where the building blocks are illustrated in the picture on the left in Figure~6.5.5.

\begin{displaymath}
\begin{array}{c}
\includegraphics[scale=0.75]{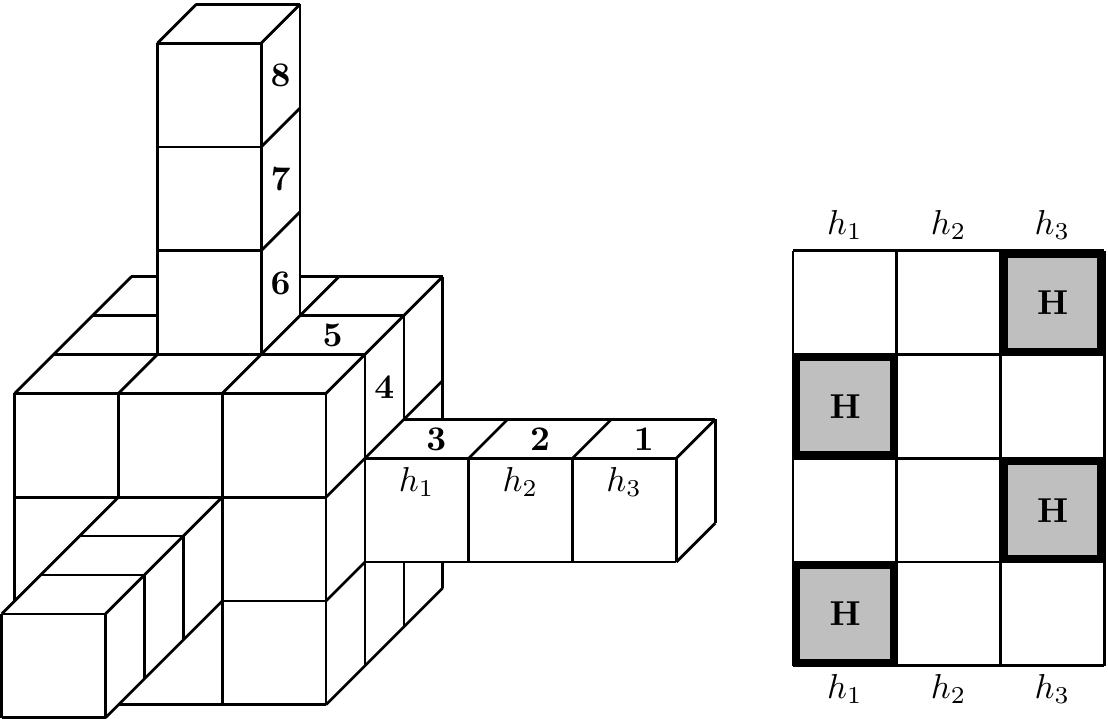}
\vspace{3pt}\\
\mbox{Figure 6.5.5: polycube and holes on the corridor}
\end{array}
\end{displaymath}

More precisely, consider all integer lattice points $\bfn=(n_1,n_2,n_3)\in\Zz^3$ such that every coordinate $n_i$ is divisible by~$6$, and place an aligned $3\times3\times3$ cube centered at each such lattice point~$\bfn$.
We then join these $3\times3\times3$ cubes by corridors.
These are $1\times1\times3$ boxes, with the two ends attached to the middle squares on the relevant faces of the neighboring $3\times3\times3$ cubes.

The polysquare surface of this infinite polycube typically exhibits cubic growth-rate of the neighborhood.
This is an infinite polysquare surface where all the streets are finite.
To see this, observe that the square faces annotated by \textbf{1}--\textbf{8} form part of a street of length~$20$.
Hence a street that contains a square of a corridor has length $4$ or~$20$.
Clearly any street that does not contain a square of a corridor has length~$12$.
The surface is therefore a $20$-square-maze surface.

The $12$ square faces of a typical corridor form a polysquare surface with boundary which is a $4\times3$ rectangle with edge pairings given in the picture on the right in Figure~6.5.5.
We can modify the surface by removing arbitrarily some of these square faces by holes as shown.

For an integer $\ell$ substantially larger than~$12$, we have the freedom to change the lengths of the corridors and the arrangement of the holes.
This way we can easily construct an uncountable family of $\ell$-square-maze surfaces for which the growth-rate of the neighborhood is cubic.

We can go even further.
Given any integer $d\ge4$, it is not difficult to construct an uncountable family of square-mazes for which the growth-rate of the neighborhood is a $d$-th power.

What is more, it is not difficult to achieve \textit{exponential} growth-rate of the neighborhood.
For example, let us go back to the two $4\times4$ squares of types $+$ and $-$ in Figure~6.5.4.
Instead of dividing the plane into $4\times4$ squares and placing a type $+$ or type $-$ configuration arbitrarily in each of them, where the cardinality of distinct configurations is the cardinality of $2^{\Zz^2}$, we shall follow a different pattern.
We can build a $4$-regular infinite tree, \textit{i.e.}, an infinite connected cycle-free graph where every degree is~$4$.
What we get then is an \textit{abstract} or \textit{exotic} polysquare surface, but a perfectly well-defined legitimate surface nonetheless.

More precisely, we start with a $4\times4$ square of type $+$ or type~$-$.
For each of the $4$ sides, we attach a new $4\times4$ square of type $+$ or type~$-$, with arbitrary choice.
Each of these new $4\times4$ squares has $3$ free sides, and for each of these, we attach a new $4\times4$ square of type $+$ or type~$-$, again with arbitrary choice, and call them \textit{second-round new}.
Each of the second-round new $4\times4$ squares has $3$ free sides, and for each of these, we attach a new $4\times4$ square of type $+$ or type~$-$, again with arbitrary choice, and call them \textit{third-round new}.
And so on, we keep going forever.
Then the growth-rate of the $n$-th neighborhood is exponential in the range of~$3^n$.

Of course there are many, many more ways to construct an uncountable set of aperiodic square-mazes.

Let $\PPP$ be an infinite polysquare translation surface which is an $\ell$-square-maze translation surface for some integer $\ell\ge2$, with $1$-direction geodesic flow.

Since every square face in $\PPP$ is the intersection of a horizontal and a vertical street, and $\PPP$ is square-face-connected by definition, we can define the concept of \textit{$\PPP$-distance} between distinct square faces in $\PPP$ as in Section~\ref{sec6.4}.

We say that the $\PPP$-distance between two distinct square faces $S_1$ and $S_2$ is $1$ if they belong to the same horizontal street or vertical street.
Otherwise, we consider a shortest sequence of alternate horizontal and vertical streets such that the first contains~$S_1$, the last contains~$S_2$, and any two consecutive streets intersect.
Then the length of this sequence is the $\PPP$-distance between $S_1$ and~$S_2$.

Note that the $\PPP$-distance is a metric on the collection of all square faces in $\PPP$ if we further define the $\PPP$-distance of any square face and itself to be~$0$.

For simplicity we restrict our attention to slopes of the form
\begin{equation}\label{eq6.5.1}
\alpha=\alpha(a)=[a;a,a,a,\ldots]=a+\frac{1}{a+\frac{1}{a+\frac{1}{a+\cdots}}},
\end{equation}
where the common digit $a\ge\ell!$ is divisible by~$\ell!$.
This condition implies that every street length is a divisor of the common digit $a$ in \eqref{eq6.5.1}, and may be replaced by the weaker condition that $a$ is divisible by the lowest common multiple of $2,\ldots,\ell$.

Our goal is to prove the following time-quantitative density result in a square-maze translation surface.

\begin{thm}\label{thm6.5.1}
Suppose that $\PPP$ is an $\ell$-square-maze translation surface for some integer $\ell\ge2$.
For any fixed constant $\eps>0$, there exist infinitely many numbers $\alpha$ of the form \eqref{eq6.5.1}, where the common continued fraction digit $a\ge\ell!$ is divisible by~$\ell!$, such that the half-infinite geodesic $V(\alpha;t)$, $t\ge0$, starting at the origin~$\bzero$, with slope $\alpha$ and with arc-length parametrization, exhibits time-quantitative density in $\PPP$ in the following precise sense.
For any square face $S$ of~$\PPP$, there is an effectively computable threshold constant $c_0=c_0(S;\eps;\alpha)$ such that for every integer
$n\ge c_0$ and every point $Q\in S$, the initial segment $V(\alpha;t)$, $0<t<n^{3(\log\alpha)/(\log\alpha-\log2)+\eps}$, gets $(1/n)$-close to~$Q$.
\end{thm}

\begin{remark}
We cannot prove ergodicity, but we can prove a time-quantitative form of density of individual orbits.
Note that ergodicity and time-quantitative density are not compatible, and neither one implies the other.
Each says something relevant about the dynamics from two different viewpoints.
\end{remark}

\begin{opthree}
Is the assumption that \textit{the starting point $V(0)$ of $V$ is a square corner} in Theorem~\ref{thm6.5.1} necessary?
Can we have an arbitrary starting point?
\end{opthree}

\begin{proof}[Proof of Theorem~\ref{thm6.5.1}]
We shall adopt the magnification process in the proof of Theorem~\ref{thm6.1.1}, but do not use the trick of exponentially fast zigzagging to a street corner; see Figure~6.2.2.
Thus the argument here is somewhat simpler than that of the proof of Theorem~\ref{thm6.1.1}.

As usual, we assume that every square face of $\PPP$ has side length one.
We pick an arbitrary square face of~$\PPP$, pick one of its $4$ corner points, and call it~$\bzero$.

Let $V(t)=V(\bzero;\alpha;t)$, $t\ge0$, denote the almost vertical geodesic in $\PPP$ that starts from $\bzero$ and has slope~$\alpha$.
Let $H(t)=H(\bzero;\alpha;t)$, $t\ge0$, denote the almost horizontal geodesic in $\PPP$ that starts from $\bzero$ and has slope~$\alpha^{-1}$.

By hypothesis every street length in $\PPP$  is a divisor of the common digit $a$ in \eqref{eq6.5.1}, which implies that the almost horizontal $H(t)$, $t\ge0$, and the almost vertical $V(t)$, $t\ge0$, are \textit{shortlines} of each other.
That is, they are \textit{mutual shortlines}, where the concept of shortline is introduced at the beginning of Section~\ref{sec6.2}.

We follow closely the argument in Section~\ref{sec6.2}.
First we recall the so-called vertical same edge cutting property of the shortline process, which says that the almost vertical $V(t)$, $t\ge0$, and its shortline $H(t)$, $t\ge0$, have precisely the same edge-cutting points on the vertical sides of vertical streets. 
We also have the analogous horizontal same edge cutting property, which says that the almost horizontal $H(t)$, $t\ge0$, and its shortline
$V(t)$, $t\ge0$, have precisely the same edge-cutting points on the horizontal sides of horizontal streets. 

Let $V^*$ be a finite initial segment of $V(t)$, $t\ge0$, and assume that $V^*$ is \textit{long}.
It is clear that $V^*$ consists of a number of \textit{whole} detour crossings and possibly a \textit{fractional} detour crossing at the end.
Clearly the length of $V^*$ is some multiple of $(1+\alpha^2)^{1/2}$, the common length of detour crossings of vertical streets.
In other words,
\begin{equation}\label{eq6.5.2}
\length(V^*)=m_0(1+\alpha^2)^{1/2}
\quad\mbox{for some large positive real number $m_0$},
\end{equation}
where the integer part of $m_0$ is the number of whole detour crossings in~$V^*$.
Each whole detour crossing in $V^*$ has a shortcut, which is part of the almost horizontal shortline $H$ of~$V$.
The last fractional detour crossing in~$V^*$, if extended to a full detour crossing, also has a shortcut, which is also part of the almost horizontal shortline $H$ of~$V$.
For this last fractional detour crossing, we shorten its shortcut by the same fraction and at the appropriate end to obtain a fractional shortcut.
We the take the union of these shortcuts and this fractional shortcut.
This union is a segment of $H$ that we denote by~$H_1^*$.
Clearly the length of $H_1^*$  is some multiple of $(1+\alpha^2)^{1/2}$, the common length of detour crossings of horizontal streets.
In other words,
\begin{displaymath}
\length(H_1^*)=m_1(1+\alpha^2)^{1/2}
\quad\mbox{for some real number $m_1$}.
\end{displaymath}
We keep iterating this.
Let
\begin{align}
\length(V_2^*)
&
=m_2(1+\alpha^2)^{1/2}
\quad\mbox{for some real number $m_2$},
\nonumber
\\
\length(H_3^*)
&
=m_3(1+\alpha^2)^{1/2}
\quad\mbox{for some real number $m_3$},
\nonumber
\end{align}
and so on.
Consider the decreasing sequence
\begin{displaymath}
m_0\ge m_1\ge m_2\ge m_3\ge\cdots.
\end{displaymath}

Repeating the argument of \eqref{eq6.2.12}--\eqref{eq6.2.14}, we obtain the analogous result
\begin{equation}\label{eq6.5.3}
m_k=m_0\alpha^{-k}.
\end{equation}

Let $I_0$ be a $V^*$-free interval on a vertical edge of a square face on a vertical street of~$\PPP$, so that $I_0$ and $V^*$ are disjoint.
It follows from the vertical same edge cutting property of the shortline process that the almost horizontal $\alpha^{-1}$-flow projects (tilted parallel projection) the interval $I_0$ to an $H_1^*$-free interval on a horizontal edge of a horizontal street of~$\PPP$.
The reader may want to go back to Figures 6.2.7 and~6.2.11 for illustration.
Let $I_1$ denote this $H_1^*$-free interval.
Then $I_1$ is a subinterval of a horizontal edge of a square face on a horizontal street \textit{if} the $\alpha^{-1}$-flow does not split the image.
We now iterate this.
It follows from the horizontal same edge cutting property of the shortline process that the almost vertical $\alpha$-flow projects the interval $I_1$ to a $V_2^*$-free interval on a vertical edge of a vertical street of~$\PPP$.
Let $I_2$ denote this $V_2^*$-free interval.
Then $I_2$ is a subinterval of a vertical edge of a square face on a vertical street \textit{if} the $\alpha$-flow does not split the image.
And then the almost horizontal $\alpha^{-1}$-flow projects the interval $I_2$ to an $H_3^*$-free interval on a horizontal edge of a horizontal street
of~$\PPP$.
Let $I_3$ denote this $H_3^*$-free interval.
Then $I_3$ is a subinterval of a horizontal edge of a square face on a horizontal street \textit{if} the $\alpha^{-1}$-flow does not split the image.
And so on, always observing the rule for magnification in a polysquare.

\begin{remark}
Unlike in Section~\ref{sec6.2}, here we cannot use the trick of exponentially fast zigzagging to a street corner; see Figure~6.2.2.
By using that trick in Section~\ref{sec6.2} we can prevent the appearance of \textit{bad} flow, when a singularity splits some image into \textit{two} intervals.
Here we have no choice but to accept the possibility of such splitting, and deal with it as a possible worst case scenario.
It means that if splitting occurs, then of course we take the longer part.
\end{remark}

It is well possible that already~$I_1$, the $\alpha^{-1}$-flow image of the starting $V^*$-free interval~$I_0$, splits.
Let $J_1\subset I_1$ denote the longer part, so that $J_1$ is an $H_1^*$-free interval and a subinterval of a horizontal edge of a square face on a horizontal street
of~$\PPP$.
Clearly $\vert J_1\vert\ge\min\{1,\vert I_0\vert\alpha/2\}$, where $\vert J\vert$ denotes the length of an interval~$J$.
In the next step the almost vertical $\alpha$-flow projects the interval $J_1$ to a $V_2^*$-free interval, which may split.
We take the longer part and denote it by~$J_2$.
Now $J_2$ is a $V_2^*$-free interval and a subinterval of a vertical edge of a square face on a vertical street of~$\PPP$.
Clearly $\vert J_2\vert\ge\min\{ 1,\vert J_1\vert\alpha/2\}$.
And so on.

Thus this magnification process produces a chain of intervals
\begin{equation}\label{eq6.5.4}
I_0=J_0\to J_1\to J_2\to J_3\to\cdots\to J_k\to\cdots
\end{equation}
such that, writing $V_0^*=V^*$, for every integer $i\ge0$,

(1) $J_{2i}$ is a $V_{2i}^*$-free interval and a subinterval of a vertical edge of a square face on a vertical street of~$\PPP$;

(2) $J_{2i+1}$ is an $H_{2i+1}^*$-free interval and a subinterval of a horizontal edge of a square face on a horizontal street of~$\PPP$; and

(3) $\vert J_{i+1}\vert\ge\min\{1,\vert J_i\vert\alpha/2\}$.

It is not an accident that in \eqref{eq6.5.3} and \eqref{eq6.5.4} we use the same unspecified index~$k$.
We complete the proof of Theorem~\ref{thm6.5.1} by making an appropriate choice of this common index.

Let $k=k_0$ be the smallest even integer such that
\begin{equation}\label{eq6.5.5}
\vert I_0\vert(\alpha/2)^k >2.
\end{equation}
Then it follows from (3) that $\vert J_{k_0}\vert=1$.
Combining this with (1), we conclude that $J_{k_0}$ is a \textit{whole} vertical edge of a square face on a vertical street of~$\PPP$, and this vertical edge is $V_{k_0}^*$-free. 

We recall that the square face corner $\bzero$ is the common starting point $V(0)=H(0)$ of the two particular geodesics $V$ and $H$ that are shortlines of each other.

Let $S_0$ denote a square face of $\PPP$ that contains the common starting point~$\bzero$.
Let $S_1$ denote a square face of $\PPP$ that contains the $V^*$-free interval~$I_0$, the first interval in the chain \eqref{eq6.5.4}, on its left boundary, and let $S_2$ denote a square face of $\PPP$ that contains the $V_{k_0}^*$-free edge $J_{k_0}$ on its right boundary.

For $0\le\xi,\eta\le 2$, let $d(\xi,\eta)$ denote the $\PPP$-distance between the square faces $S_\xi$ and~$S_\eta$.

The upper bound $d(1,2)\le k_0$ is a straightforward corollary of the $k_0$-step construction of the magnification process \eqref{eq6.5.4} with $k=k_0$.
Indeed, in each step, the $\PPP$-distance increases by at most~$1$, as a consequence of the triangle inequality.
Combining this with the triangle inequality we deduce that the $\PPP$-distance $d(0,2)$ between the square faces $S_0$ and $S_2$ has the upper bound
\begin{equation}\label{eq6.5.6}
d(0,2)\le d(0,1)+d(1,2)\le d(0,1)+k_0.
\end{equation}

We recall some key facts.
First of all, $J_{k_0}$ is a $V_{k_0}^*$-free  edge on the boundary of the square face $S_2$ in~$\PPP$.
Next, by \eqref{eq6.5.2} and \eqref{eq6.5.3} we have a very good estimation for the length of $V_{k_0}^*$, given by
\begin{equation}\label{eq6.5.7}
\length(V_{k_0}^*)=m_{k_0}(1+\alpha^2)^{1/2}
\quad\mbox{and}\quad
\length(V^*)=m_0(1+\alpha^2)^{1/2},
\end{equation}
where
\begin{equation}\label{eq6.5.8}
m_{k_0}=m_0\alpha^{-k_0}.
\end{equation}

We apply a variant of Lemma~\ref{lem6.4.1} to the $\ell$-square-maze translation surface~$\PPP$, noting that in the proof of Lemma~\ref{lem6.4.1}, we have not used at all the fact that the polysquare translation surface is finite.
We need to make some simple modification to Lemma~\ref{lem6.4.1} as we start the geodesic here at the origin~$\bzero$, so that we can talk about the \textit{first} detour crossing which is a \textit{whole} detour crossing.

\begin{lem}\label{lem6.5.1}
Suppose that $\ell\ge2$ and $\PPP$ is an $\ell$-square-maze translation surface, with $1$-direction geodesic $V$ starting at the origin $\bzero$ and of slope $\alpha$ which is an irrational number given by \eqref{eq6.5.1}, where the continued fraction digit $a$ is a positive integer multiple of~$\ell!$.
Let $S'$ be a given square face of~$\PPP$, and let $\VVV$ denote the finite almost vertical geodesic made up of the first $3$ detour crossings of the $i$-generation shortline of $V$ for some even integer $i\ge4$.
Suppose that a unit contained in the first two detour crossings in~$\VVV$ intersects the bottom edge of~$S'$.
Then in every square face $S''$ of $\PPP$ for which the $\PPP$-distance between $S'$ and $S''$ is at most~$2$, the $4$-generation ancestor geodesic of $\VVV$ gives rise to an almost vertical unit of type $\nuparrow$ in the square face.
\end{lem}

\begin{remark}
Note that the first detour crossing of $V$ and of $V^*$ are the same.
\end{remark}

Recall that $V^*$ is a \textit{long} initial segment of the special almost vertical $V(t)$, $t\ge0$.

We start with $2$ whole detour crossings of $V^*$, of total length $2(1+\alpha^2)^{1/2}$.
Let $h_0\ge4$ be any integer that is divisible by~$4$.
The $h_0$-step shortcut-ancestor process, where at each stage we include fractional units proportionally, now gives rise to an initial segment of $V^*$ with length
\begin{equation}\label{eq6.5.9}
2\alpha^{h_0}(1+\alpha^2)^{1/2}.
\end{equation}
We know from $h_0/4$ iterations of Lemma~\ref{lem6.5.1} that for every square face with $\PPP$-distance at most $h_0/2$ from the square face~$S_0$, this initial segment of $V^*$ of length \eqref{eq6.5.9} gives rise to an almost vertical unit of type $\nuparrow$ in the square face.

We now specify the parameter $h_0$ to satisfy the requirement that this segment is contained in $V_{k_0}^*$, so that
\begin{equation}\label{eq6.5.10}
\length(V_{k_0}^*)\ge2\alpha^{h_0}(1+\alpha^2)^{1/2}.
\end{equation}
Comparing \eqref{eq6.5.8} and \eqref{eq6.5.9}, we see that a condition like
\begin{equation}\label{eq6.5.11}
m_0\alpha^{-k_0}\ge2\alpha^{h_0}
\end{equation}
implies \eqref{eq6.5.10}.

We next claim that
\begin{equation}\label{eq6.5.12}
\frac{h_0}{2}<d(0,1)+k_0,
\end{equation}
where $d(0,1)$ denotes the $\PPP$-distance between the square faces $S_0$ and~$S_1$. 
The proof is by contradiction.
Suppose on the contrary that \eqref{eq6.5.12} does not hold.
Then it follows from \eqref{eq6.5.6} that $h_0/2\ge d(0,2)$, and so for every square face with $\PPP$-distance at most $d(0,2)$ from the square face~$S_0$, the particular initial segment of $V^*$ of length \eqref{eq6.5.9} gives rise to an almost vertical unit of type $\nuparrow$ in the square face.
It then follows from \eqref{eq6.5.10} that for every square face with $\PPP$-distance at most $d(0,2)$ from the square face~$S_0$, $V_{k_0}^*$ gives rise to an almost vertical unit of type $\nuparrow$ in the square face.
It follows that the right edge of the square face $S_2$ is not a $V_{k_0}^*$-free vertical edge.
But this is a contradiction, since we know that $J_{k_0}$ \textit{is} a $V_{k_0}^*$-free vertical edge on the boundary of the square face~$S_2$, and it is the whole edge.
This contradiction proves \eqref{eq6.5.12}.

Recall that $k=k_0$ is the smallest even integer such that \eqref{eq6.5.5} holds.
This implies
\begin{displaymath}
2(\alpha/2)^{-k_0}<\vert I_0\vert\le2(\alpha /2)^{-k_0+2}.
\end{displaymath}

Combining \eqref{eq6.5.7} and \eqref{eq6.5.11}, we have
\begin{equation}\label{eq6.5.13}
\length(V^*)\ge2(1+\alpha^2)^{1/2}\alpha^{h_0+k_0}.
\end{equation}
In view of \eqref{eq6.5.12}, it follows that the choice
\begin{displaymath}
\length(V^*)\ge2(1+\alpha^2)^{1/2}\alpha^{2d(0,1)+3k_0}
\end{displaymath}
clearly implies the inequality \eqref{eq6.5.13}.

We conclude, therefore, that if we specify the length of the initial segment $V^*$ of the special geodesic $V(t)$, $t\ge0$, starting from $\bzero$ with slope~$\alpha$, as
\begin{equation}\label{eq6.5.14}
\length(V^*)=C'(\alpha)\alpha^{2d+3k_0},
\end{equation}
where the constant $C'(\alpha)$ is sufficiently large, then the longest $V^*$-free interval on any edge of a square face with $\PPP$-distance at most $d$ from the square face $S_0$ is \textit{short}.
More precisely, such a $V^*$-free interval has length roughly at most $2(\alpha/2)^{-k_0+2}$.

Let $\eps>0$ be arbitrarily small but fixed.
We choose the integer variable $n$ to satisfy
\begin{equation}\label{eq6.5.15}
n-1<2(\alpha/2)^{k_0-2}\le n.
\end{equation}
Then by \eqref{eq6.5.14} and \eqref{eq6.5.15},
\begin{displaymath}
\length(V^*)=C'(\alpha)\alpha^{2d+3k_0}\le C''(S;\alpha)n^{3(\log\alpha)/(\log\alpha-\log2)+\eps},
\end{displaymath}
provided that $n$ is sufficiently large depending on $\eps>0$ and noting that $d$ depends on~$S$.
This completes the proof of Theorem~\ref{thm6.5.1}.
\end{proof}

\begin{remark}
An important consequence of the shortline method is that dense geodesics exhibit \textit{super-slow escape rate to infinity}.

We can describe this phenomenon in terms of the concept of the $\PPP$-diameter restricted to a finite geodesic.
Suppose that $L$ is a finite geodesic on a polysquare translation surface~$\PPP$.
To calculate the $\PPP$-diameter restricted to the geodesic~$L$, we simply consider the $\PPP$-distance between any two square faces of $\PPP$ visited by~$L$, and find the maximum value among these.

Using the notation of the proof of Theorem~\ref{thm6.5.1}, we consider the decreasing chain
\begin{displaymath}
V^*=V_0^*\to H_1^*\to V_2^*\to H_3^*\to\ldots\to V_{2i}^*\to H_{2i+1}^*\to\ldots
\end{displaymath}
of geodesic segments.
This is \textit{decreasing exponentially fast}, in the sense that the ratio of the lengths of consecutive segments is equal to~$\alpha$.
On the other hand, the shortline method implies that the $\PPP$-diameter restricted to any of these geodesic segments does not exceed the
$\PPP$-diameter restricted to the next geodesic segment in the chain by more than~$1$.
Thus the dense geodesic exhibits \textit{logarithmic escape rate to infinity} in terms of the restricted $\PPP$-distances.
\end{remark}

%%%%%%%%%%
%
% SECTION 6.6
%
%%%%%%%%%%

%\subsection{Section deleted}\label{sec6.6}

\setcounter{subsection}{6}

%%%%%%%%%%
%
% SECTION 6.7
%
%%%%%%%%%%

\subsection{Density on aperiodic surfaces with infinite streets (I)}\label{sec6.7}

What can we say about those infinite polysquare surfaces which are very different from a maze?
For instance, what can we say about those polysquare surfaces which have infinite streets?

In terms of optics, the billiard case of Theorem~\ref{thm6.5.1} is equivalent to the result that, for any $\ell\ge2$, there are slopes $\alpha$ such that every $\ell$-square-maze with reflecting boundary (mirrors) can be illuminated by a single ray of light from a carefully chosen point and with slope~$\alpha$.
This represents an uncountable family of infinite billiards.

The infinite aperiodic billiards with square obstacles belong to the class that physicists call the \textit{Ehrenfest wind-tree models}.
In 1912 Ehrenfest and Ehrenfest wrote a long important encyclopaedia article on the foundations of statistical mechanics; see~\cite{EE}.
The Appendix of Section~5 in the first chapter of the article introduces a \textit{much simplified model}, as the Ehrenfests called it, to illustrate the works of
Maxwell--Boltzmann.
In this model a point particle (wind) moves freely on the plane and collides with the well known reflection law of geometric optics with an infinite number of irregularly placed congruent square scatterers (trees).
In the rest we refer to the particle (wind) as a billiard, and call it the Ehrenfest wind-tree (billiard) model.

The Ehrenfests raised the problem of studying the individual billiard orbits in the wind-tree model.
They wanted to understand the dynamics of this billiard model.

Unfortunately relatively little is known about this original aperiodic version of the problem.
The best known result is due to Sabogal and Troubetzkoy~\cite{ST}.
They can show that in a well-defined class of aperiodic wind-tree models, for a generic, in the sense of Baire, configuration of the obstacles, the wind-tree dynamics is ergodic in almost every direction.
This is very interesting, but unfortunately their method does not give explicit results.
They do not give any explicit configuration of the obstacles that is ergodic in almost every direction.
They cannot even provide an explicit configuration that is ergodic in a single explicit direction.

However, there are many recent results on the dynamics of the \textit{periodic} version of the wind-tree model, introduced by Hardy and Weber~\cite{HW} around 1980.
It concerns the double-periodic arrangement of identical squares, or identical rectangles, where one obstacle is centered at each integer lattice point on the plane.
Here every rectangle (scatterer) has the same size $a\times b$ with $0<a,b<1$, and they are placed in the usual horizontal/vertical position, with horizontal and vertical periods both equal to~$1$.
So these double-periodic models all have infinitely many infinite horizontal and vertical streets.

\begin{displaymath}
\begin{array}{c}
\includegraphics[scale=0.75]{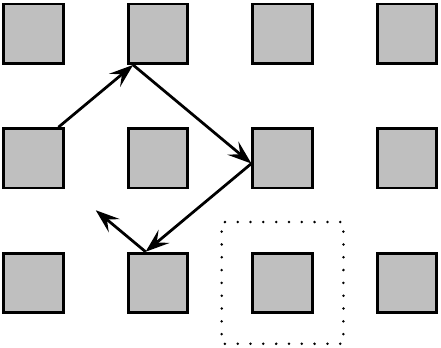}
\vspace{3pt}\\
\mbox{Figure 6.7.1: double-periodic wind-tree billiard model with side length $1/2$}
\end{array}
\end{displaymath}

Figure~6.7.1 illustrates the simplest special case $a=b=1/2$ of such a model, with square scatterers of side length~$1/2$.
The symbolic dotted square in the lower-right corner represents a \textit{table-period}.

There are many recent results in this periodic case, but none of these results guarantee density.

A natural way to make these periodic models \textit{aperiodic} is to drop infinitely many obstacles at some \textit{irregularly chosen} lattice points.

\begin{fconfig}
The idea of $f$-configurations is a natural way to produce an uncountable family of explicit aperiodic configurations of obstacles.
We divide the plane into a union of $3\times3$ squares
\begin{displaymath}
Q(i,j)=\left[3i-\frac{3}{2},3i+\frac{3}{2}\right)\times\left[3j-\frac{3}{2},3j+\frac{3}{2}\right),
\quad
(i,j)\in\Zz^2,
\end{displaymath}
each centered at the lattice point $(3i,3j)\in\Zz^2$.

Let $\FFF$ denote the family of all functions $f:\Zz^2\to\{\pm1\}$.
Clearly $\FFF$ is an uncountable set.

\begin{displaymath}
\begin{array}{c}
\includegraphics[scale=0.75]{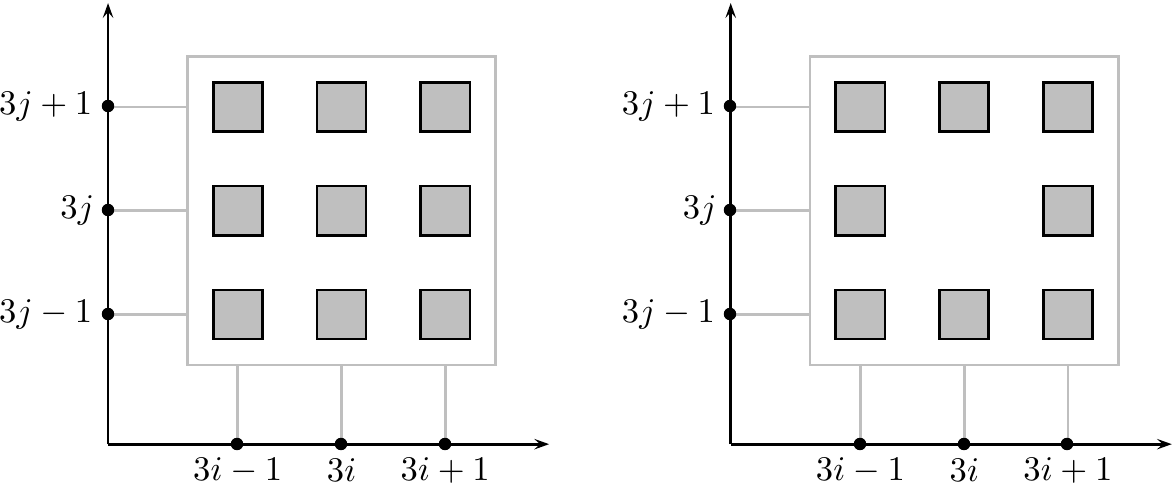}
\vspace{3pt}\\
\mbox{Figure 6.7.2: the $3\times3$ square $Q(i,j)$ with $f(i,j)=+1$ and with $f(i,j)=-1$}
\end{array}
\end{displaymath}

For each $(i,j)\in\Zz^2$, we now place an obstacle in the form of an aligned square of side length $1/2$ centered at each of the $8$ lattice points in $Q(i,j)$ that is distinct from the lattice point $(3i,3j)$.
For each function $f\in\FFF$, we place an extra obstacle in the form of an aligned square of side length $1/2$ centered at the lattice point $(3i,3j)$ if $f(i,j)=+1$, but do not place such an extra obstacle if $f(i,j)=-1$.
We refer to this configuration of obstacles as the $f$-configuration of $Q(i,j)$, as illustrated in Figure~6.7.2.

Notice that in every $f$-configuration an \textit{empty} lattice point without obstacle is surrounded by $24$ \textit{non-empty} lattice points with obstacles, as shown in Figure~6.7.3 below.

\begin{displaymath}
\begin{array}{c}
\includegraphics[scale=0.75]{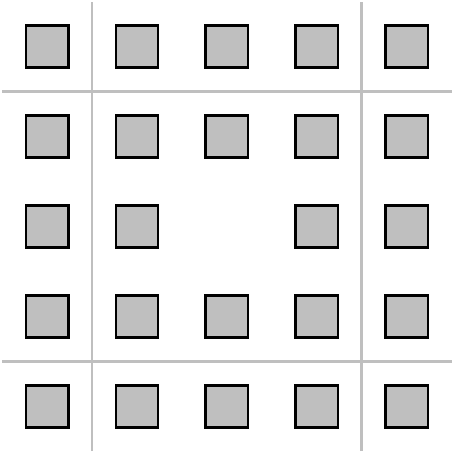}
\vspace{3pt}\\
\mbox{Figure 6.7.3: a lattice point without obstacle in the $f$-configuration}
\end{array}
\end{displaymath}

In fact, all we need in a configuration is that the distance between any two empty lattice points is greater than~$2$.

The set of double-periodic $f$-configurations is countable, negligible compared to the uncountable total, implying that the overwhelming majority of $f$-configurations are aperiodic.
\end{fconfig}

We claim that there are infinitely many explicit starting points and infinitely many explicit quadratic irrational slopes such that the corresponding billiard orbits exhibit time-quantitative density for \textit{all} $f$-configurations with $f\in\FFF$.

To prove this, the obvious difficulty is that we cannot directly apply Theorem~\ref{thm6.5.1} in the original \textit{horizontal/vertical} way, since this wind-tree model has horizontal and vertical streets that are infinite, so it is not a square-maze translation surface.
The trick is to apply it in an \textit{indirect} way by working with \textit{tilted} streets.

In the simpler case when $f(i,j)=+1$, so that there is no missing obstacle, we can consider finite $\pm45$-degree streets as illustrated in Figure~6.7.4.
Note that this tilted street has two orientations: clockwise and counter-clockwise.
For example, in the clockwise direction, we go from $1$ to~$2$, illustrated in lighter shade, then from $2$ to~$3$, illustrated in darker shade, then from $3$ to~$4$, illustrated in lighter shade, and finally from $4$ back to~$1$, illustrated in darker shade.

\begin{displaymath}
\begin{array}{c}
\includegraphics[scale=0.75]{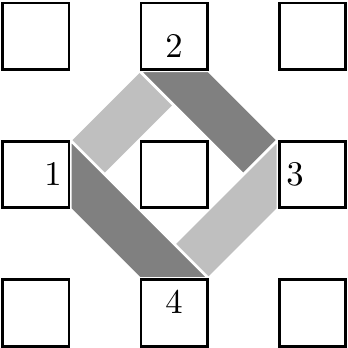}
\vspace{3pt}\\
\mbox{Figure 6.7.4: finite tilted street in the case $f(i,j)=1$}
\end{array}
\end{displaymath}

In the other case when $f(i,j)=-1$, so that there is a missing obstacle, we can consider finite $\pm45$-degree streets as illustrated in Figure~6.7.5.
Note that this tilted street also has two orientations.
For example, we can go from $1$ to~$2$, illustrated in lighter shade, then from $2$ to~$3$, illustrated in darker shade, and so on, and finally from $12$ back to~$1$, illustrated in darker shade.

\begin{displaymath}
\begin{array}{c}
\includegraphics[scale=0.75]{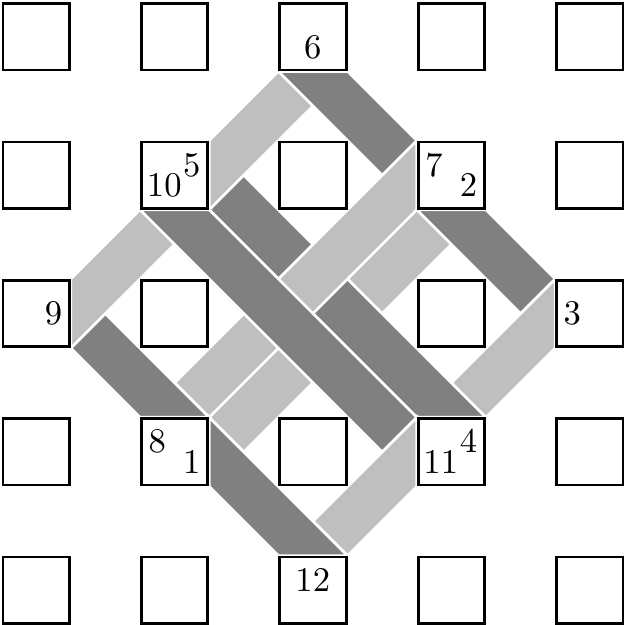}
\vspace{3pt}\\
\mbox{Figure 6.7.5: finite tilted street in the case $f(i,j)=-1$}
\end{array}
\end{displaymath}

While the basic idea of proving density by combining the finite tilted streets with Theorem~\ref{thm6.5.1} is very simple, the technical details are somewhat complicated.
We shall therefore first give a detailed discussion in a simpler case where there is only one infinite street.
For this case, it is much easier to visualize the corresponding infinite polysquare translation surface with $1$-direction geodesic flow.
After that discussion it will be easier to understand the more complicated case of wind-tree models such as those illustrated in Figures 6.7.1--6.7.2 that have infinitely many infinite streets.

One of the simplest examples involving a single infinite street is the $\infty$-L-strip billiard shown in Figure~6.7.6.

\begin{displaymath}
\begin{array}{c}
\includegraphics[scale=0.75]{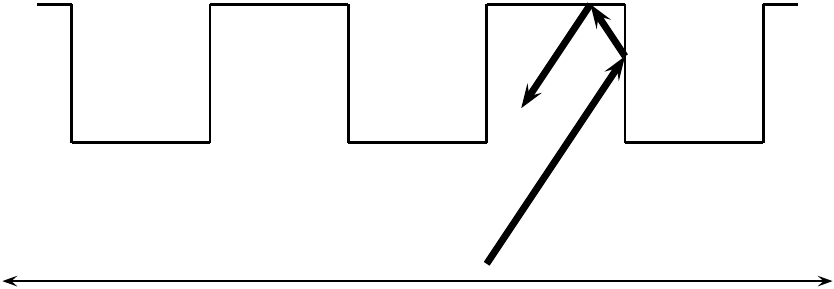}
\vspace{3pt}\\
\mbox{Figure 6.7.6: $\infty$-L-strip region as an infinite billiard table}
\end{array}
\end{displaymath}

For this $\infty$-L-strip, Figure~6.7.7 shows an explicit slope-$2$ street.
The slope-2 billiard flow, illustrated by the dashed arrows, first maps the interval $AB$ to the interval~$CD$ (via reflection on a horizontal edge), illustrated in lighter shade, then onwards to the interval~$EF$ (via reflection on a vertical edge), illustrated in darker shade, then to the intervals $CB$, $AG$, $HI$ and then back to~$AB$.
Reversing the orientation, we have a reverse cycle in the same street, the slope-$(-2)$ version of the tilted street.

\begin{displaymath}
\begin{array}{c}
\includegraphics[scale=0.75]{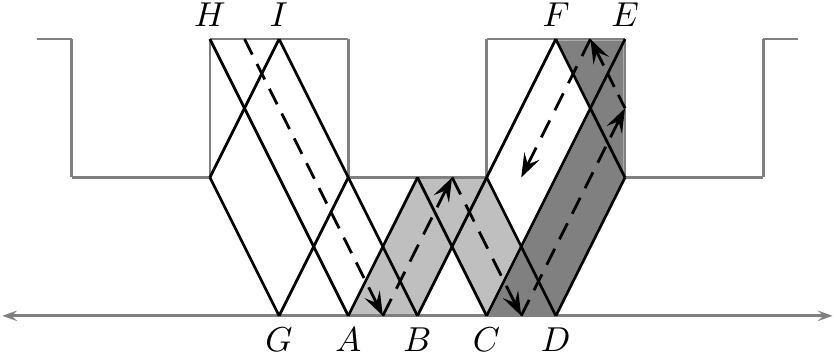}
\vspace{3pt}\\
\mbox{Figure 6.7.7: slope-$2$ street in the $\infty$-L-strip billiard}
\end{array}
\end{displaymath}

Next we use the classical trick of unfolding that reduces the study of the  $\infty$-L-strip billiard, a $4$-direction flow, to a $1$-direction geodesic flow on an appropriate infinite polysquare translation surface.
We call this surface the \textit{translation surface of the $\infty$-L-strip billiard}, and denote it by $\Bil(\infty;1)$.
To find this infinite polysquare translation surface, we glue together $4$ reflected copies of the $\infty$-L-strip region in a suitable way.

Note that the L-shape is the building block of the $\infty$-L-strip region, as illustrated in Figure~6.7.8.
The $\infty$-L-strip region can be split into a doubly-infinite sequence $\ldots,L_{i-1},L_i,L_{i+1},\ldots$ of L-shapes.

\begin{displaymath}
\begin{array}{c}
\includegraphics[scale=0.75]{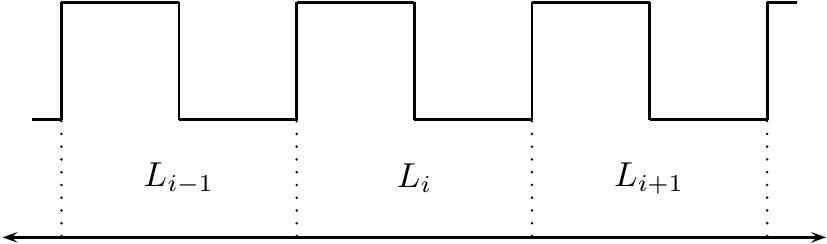}
\vspace{3pt}\\
\mbox{Figure 6.7.8: the L-shape as building blocks of the $\infty$-L-strip region}
\end{array}
\end{displaymath}

To describe $\Bil(\infty;1)$, we take $4$ reflected copies of each of the L-shapes $L_i$; see Figure~6.7.9.
We then obtain $\Bil(\infty;1)$ by gluing together the infinitely many copies of these $4$-copy-$L_i$.

\begin{displaymath}
\begin{array}{c}
\includegraphics[scale=0.75]{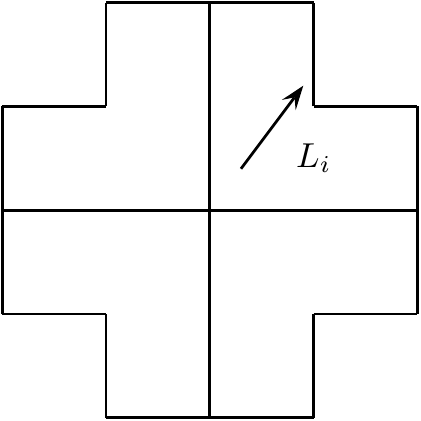}
\vspace{3pt}\\
\mbox{Figure 6.7.9: building the $4$-copy version of the L-shape $L_i$}
\end{array}
\end{displaymath}

Clearly there is billiard flow from each L-shape $L_i$ to its two immediate neighbours $L_{i-1}$ and $L_{i+1}$.
We therefore need to identify corresponding edges of these $4$-copy versions very carefully.
A simple examination will convince the reader that the edge identifications can be as illustrated in Figure~6.7.10.
Note that a $1$-direction geodesic with positive slope on $\Bil(\infty;1)$ that goes from the vertical edge $v_2^{(i)}$ to the vertical edge $v_2^{(i-1)}$ corresponds to a billiard path with negative slope going from $L_i$ to~$L_{i-1}$, whereas a $1$-direction geodesic with positive slope on $\Bil(\infty;1)$ that goes from the vertical edge $v_3^{(i-1)}$ to the vertical edge $v_3^{(i)}$ corresponds to a billiard path with positive slope going from $L_{i-1}$ to~$L_i$.

\begin{displaymath}
\begin{array}{c}
\includegraphics[scale=0.8]{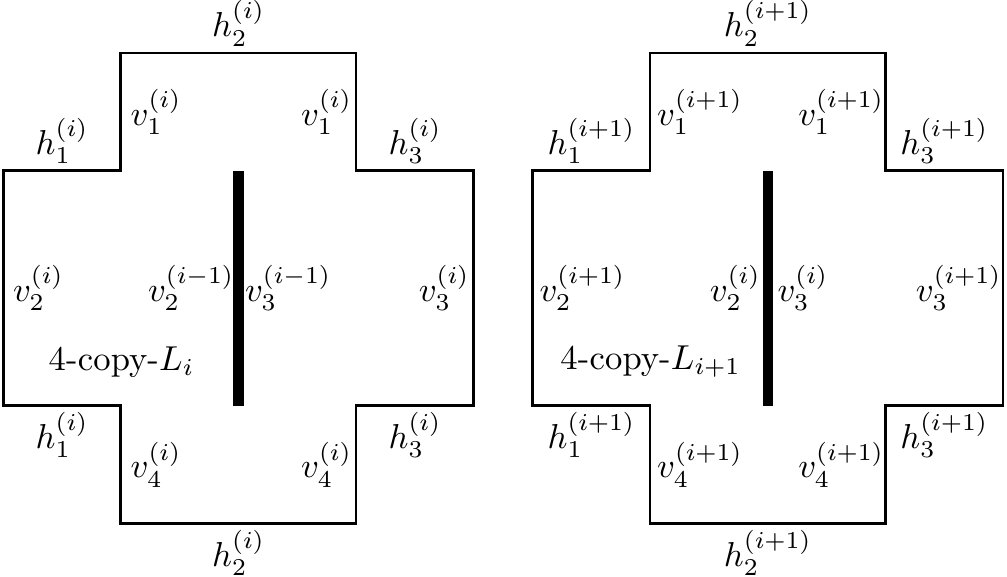}
\vspace{3pt}\\
\mbox{Figure 6.7.10: $4$-copy-$L_i$ and $4$-copy-$L_{i+1}$ together with edge identifications}
\end{array}
\end{displaymath}

It is much easier to visualize the somewhat messy picture of the slope-$2$ street and the corresponding slope-$(-2)$ street in Figure~6.7.7 on the surface $\Bil(\infty;1)$.
Figure~6.7.11 gives a good visualization of the slope-$2$ street.

\begin{displaymath}
\begin{array}{c}
\includegraphics[scale=0.75]{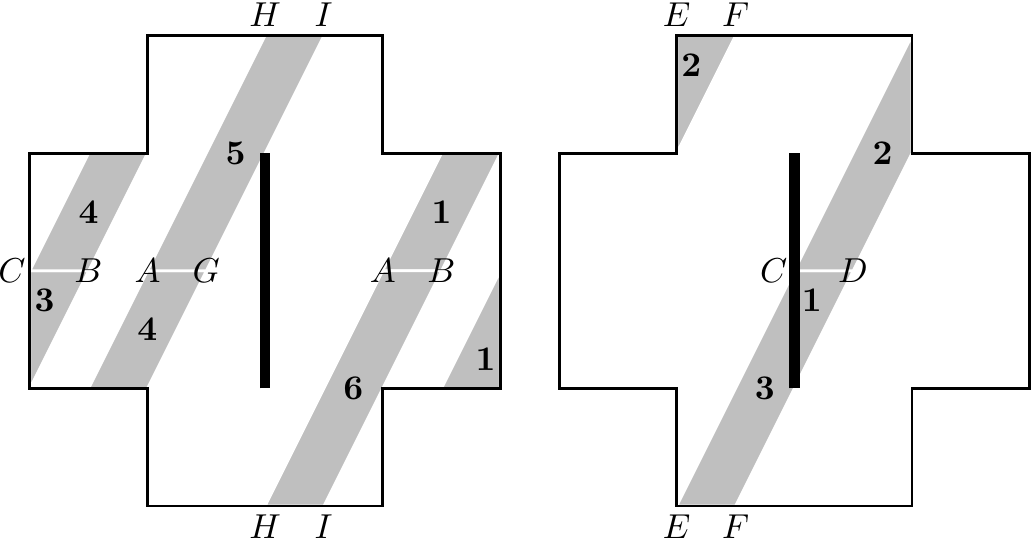}
\vspace{3pt}\\
\mbox{Figure 6.7.11: visualizing the slope-$2$ street in Figure 6.7.7}
\end{array}
\end{displaymath}

Figure~6.7.12 gives a good visualization of the slope-$(-2)$ street.

\begin{displaymath}
\begin{array}{c}
\includegraphics[scale=0.75]{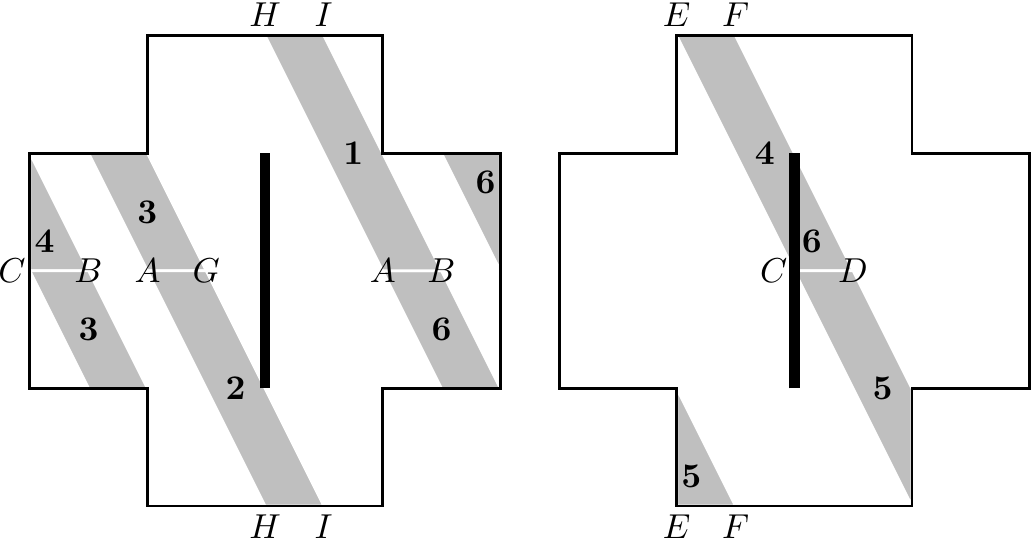}
\vspace{3pt}\\
\mbox{Figure 6.7.12: visualizing the slope-$(-2)$ street in Figure 6.7.7}
\end{array}
\end{displaymath}

If we return to Figure 6.7.7 and work out a slope-$2$ street starting with the interval~$GA$, then the darker shaded region in Figure~6.7.13 gives a good visualization of this new slope-$2$ street.
Note that the union of the shaded areas in Figure~6.7.13 gives precisely one $4$-copy L-shape.
For simplicity, we shall refer to $\NE'_i$ and $\NE''_i$ as the two slope-$2$ streets, or north-east streets, in the $4$-copy-$L_i$.
Similarly, we can refer to $\NW'_i$ and $\NW''_i$ as the slope-$(-2)$ streets, or north-west streets, in the $4$-copy-$L_i$.
Thus the infinite families
\begin{displaymath}
\NE'_i,\NE''_i,
\quad
i\in\Zz,
\end{displaymath}
and
\begin{displaymath}
\NW'_i,\NW''_i,
\quad
i\in\Zz,
\end{displaymath}
of streets, linked together, give a street decomposition of the flat surface $\Bil(\infty;1)$.
These are pairs of congruent finite streets in two directions: slope-$2$ (north-east) and slope-$(-2)$ (north-west).
For convenience, we call them NE-streets and NW-streets respectively.

\begin{displaymath}
\begin{array}{c}
\includegraphics[scale=0.75]{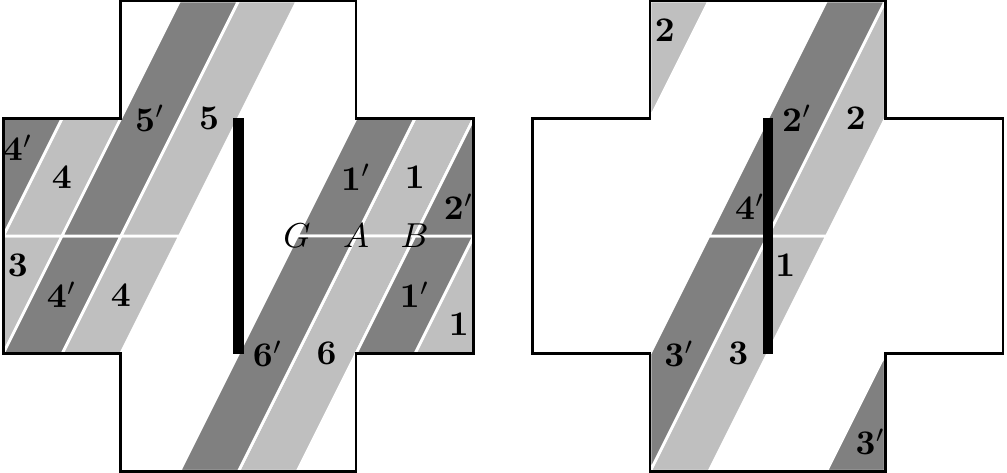}
\vspace{3pt}\\
\mbox{Figure 6.7.13: two slope-$2$ streets in Figure 6.7.7}
\end{array}
\end{displaymath}

This setup is very similar to the situation in Theorem~\ref{thm6.5.1}, where we have finite streets of bounded length in the horizontal and vertical directions.
Indeed, we can use this decomposition to convert the flat surface $\Bil(\infty;1)$ into a \textit{maze translation surface}. 

This is, however, not a square-maze translation surface, but a \textit{rhombus-maze translation surface}.
We can clearly split each of the $4$-copy-L-shapes into a union of rhombi as illustrated in Figure~6.7.14.
Note that edge identification clearly plays a key role in the partitions of the $4$-copy-L-shapes.

\begin{displaymath}
\begin{array}{c}
\includegraphics[scale=0.75]{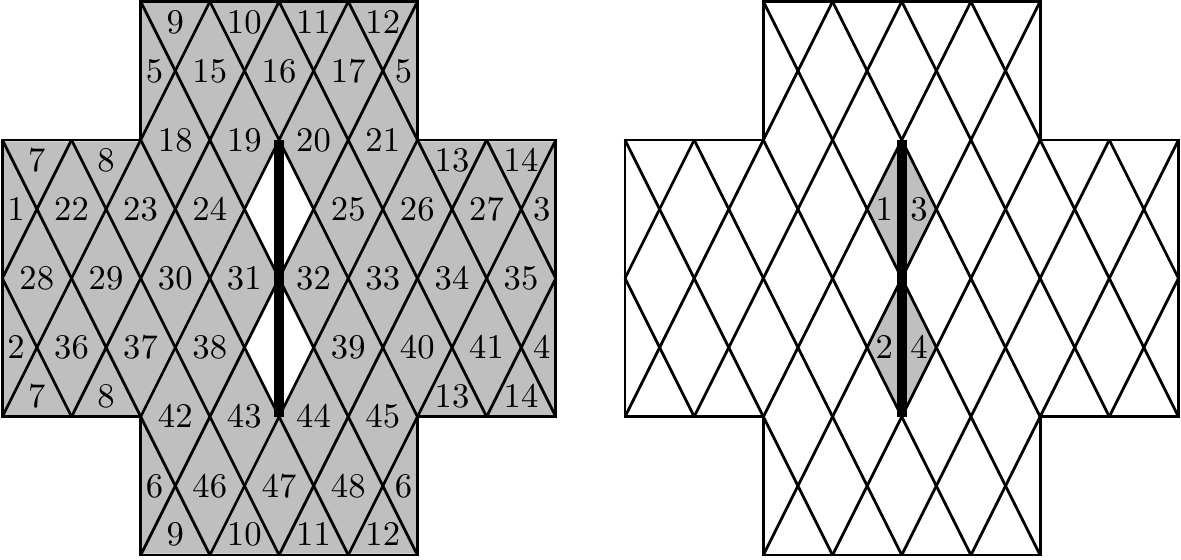}
\vspace{3pt}\\
\mbox{Figure 6.7.14: splitting the $4$-copy-$L_i$ into a union of rhombi}
\end{array}
\end{displaymath}

Let us look at the NE-street that contains the rhombus~$A$, say, in Figure~6.7.15 below.
If we now move north-east, starting at the rhombus~$A$, then it is easy to see that the NE-street comprises the $24$ rhombi $A,B,C,\ldots,X$ as illustrated by the shaded part.
Thus the street length of any NE-street is equal to~$24$.

\begin{displaymath}
\begin{array}{c}
\includegraphics[scale=0.75]{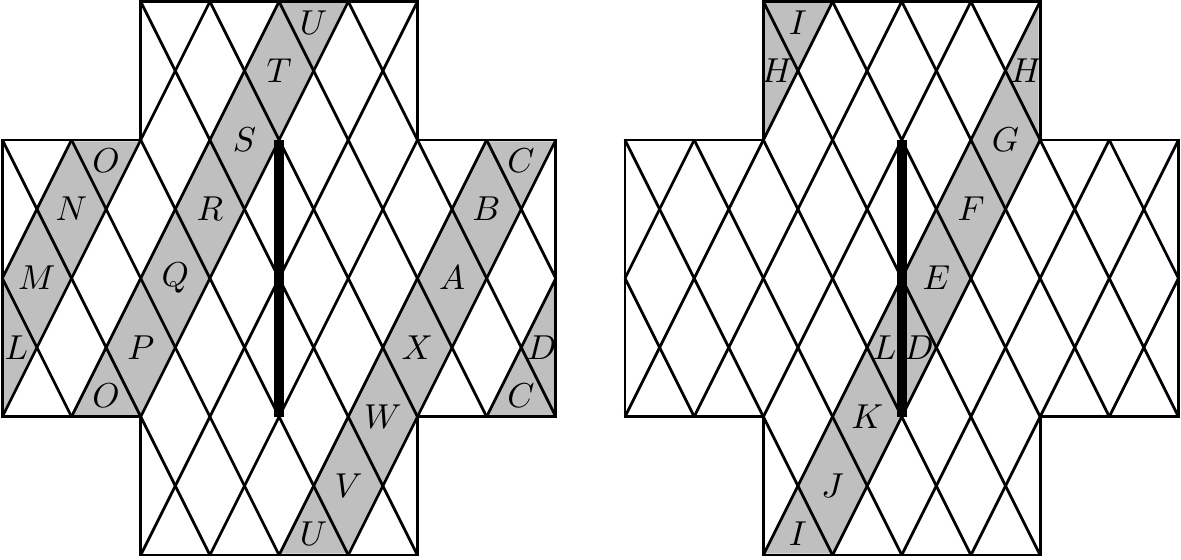}
\vspace{3pt}\\
\mbox{Figure 6.7.15: a street $\NE_i$}
\end{array}
\end{displaymath}

A similar argument will show that the street length of any NW-street is also equal to~$24$.

This gives rise to a rhombus-maze translation surface where every street has length~$24$.

We have the following result concerning time-quantitative density in a rhombus-maze translation surface which is an analog of Theorem~\ref{thm6.5.1}.

\begin{thm}\label{thm6.7.1}
Let $\PPP$ be an $\ell$-rhombus-maze translation surface, where $\ell\ge2$ is a fixed integer.
For any fixed constant $\eps>0$, there exist infinitely many quadratic irrational numbers $\alpha$ such that the geodesic $V(\alpha;t)$, $t\ge0$, on~$\PPP$, starting at the origin~$\bzero$, with slope $\alpha$ and with arc-length parametrization, exhibits time-quantitative density in the following precise sense.
There exists a computable positive constant $\alpha^*=\alpha^*(\alpha)>1$, depending at most on~$\alpha$, such that for any rhombus face $S_0$
of~$\PPP$, there is an effectively computable threshold constant $c_0=c_0(S_0;\eps;\alpha)$ such that for every integer $n\ge c_0$ and every point $Q\in S_0$, the initial segment $V(\alpha;t)$, $0<t<n^{3\alpha^*+\eps}$, gets $(1/n)$-close to~$Q$.
\end{thm}

To deduce this from Theorem~\ref{thm6.5.1}, we simply need a linear transformation to convert the rhombi to squares.
The linear transformation converts geodesics and their shortlines in the corresponding square-maze translation surface into their counterparts in the
rhombus-maze translation surface.

Recall that a $4$-direction billiard orbit in the $\infty$-L-strip region corresponds to a $1$-direction geodesic flow in $\Bil(\infty;1)$ which can be viewed as a rhombus-maze translation surface.
Theorem~\ref{thm6.7.1} thus leads immediately to the following result.

\begin{thm}\label{thm6.7.2}
There are infinitely many explicit quadratic irrational slopes such that every billiard trajectory in the $\infty$-L-strip region starting from a corner and having such a slope exhibits time-quantitative density.
\end{thm}

%%%%%%%%%%
%
% SECTION 6.8
%
%%%%%%%%%%

\subsection{Density on aperiodic surfaces with infinite streets (II)}\label{sec6.8}

The $\infty$-L-strip happens to be periodic, but periodicity is not necessary for the success of our method.
Note that in some periodic cases Hooper~\cite{Ho}, Hooper--Hubert--Weiss~\cite{HHW}, Hubert--Weiss~\cite{HuW}, and
Ralston--Troubetzkoy~\cite{RT} can prove ergodicity for some billiards on infinite surfaces, and ergodicity implies density.
They use a completely different approach that reduces the infinite dynamics to the well understood dynamics of the \textit{period}, which is a compact system.
This reduction method is quite special, and does not work for the infinite periodic case in general.
Of course in the aperiodic case this reduction method breaks down.

Our method, on the other hand, does work in both the infinite periodic and aperiodic cases.
To illustrate this, we describe next an uncountable family of infinite aperiodic surfaces for which we can prove density for some billiards.
We shall generalize the $\infty$-L-strip region which, as seen in Figure~6.7.8, is built from congruent L-shapes.

\begin{lstrips}
Given arbitrary integers $v_i\ge2$ and $h_i\ge2$, we define the $(v_i,h_i)$-L-shape in a most natural way as follows.
There is one horizontal street of $h_i$ unit size square faces and there is one vertical street of $v_i$ unit size square faces, with the left square face of this horizontal street identical to the bottom square face of this vertical street, as shown in Figure~6.8.1.
Every other street has length~$1$.
Thus the special case $v_i=h_i=2$ gives back the usual L-shape.
For every $i\in\Zz$, let $L_i$ be a $(v_i,h_i)$-L-shape, and we take the disjoint union of all $L_i$, $i\in\Zz$, such that the long horizontal streets of $L_i$, $i\in\Zz$, form a single infinite street, and $L_i$ and $L_{i+1}$, $i\in\Zz$, are consecutive.
We refer to this union as the $\{L_i\}_{i\in\Zz}$-strip region.

\begin{displaymath}
\begin{array}{c}
\includegraphics[scale=0.75]{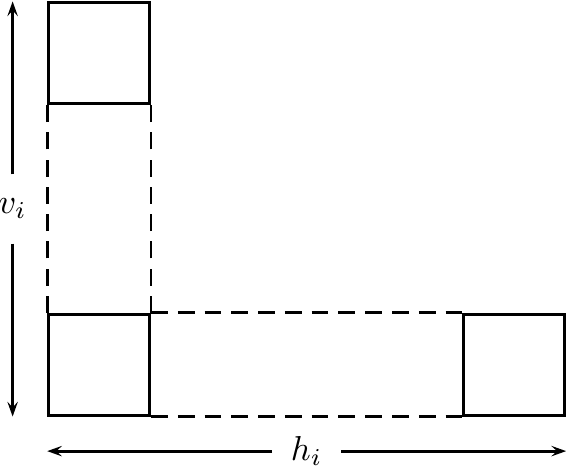}
\vspace{3pt}\\
\mbox{Figure 6.8.1: the L-shape $L_i$}
\end{array}
\end{displaymath}
\end{lstrips}

Consider first the $\{L_i\}_{i\in\Zz}$-strip region in the special case when there is an integer $r$ such that
\begin{equation}\label{eq6.8.1}
h_i\le r,
\quad
v_i\in\{2,4\},
\quad
i\in\Zz.
\end{equation}
Note that under the conditions \eqref{eq6.8.1}, a slope-$2$ street construction similar to that in Figure~6.7.7 can be made, as illustrated in one case by Figure~6.8.2.
In this case, the slope-$2$ billiard flow, illustrated by the dashed arrows, first maps the interval $AB$ to the interval $CD$ (via reflection on a vertical edge), illustrated in lighter shade, then onwards to the interval $EB$ (via multiple reflections on vertical edges), then to the intervals $AF$, $GH$, $IJ$, $KH$, $GF$ and then back to~$AB$.

\begin{displaymath}
\begin{array}{c}
\includegraphics[scale=0.75]{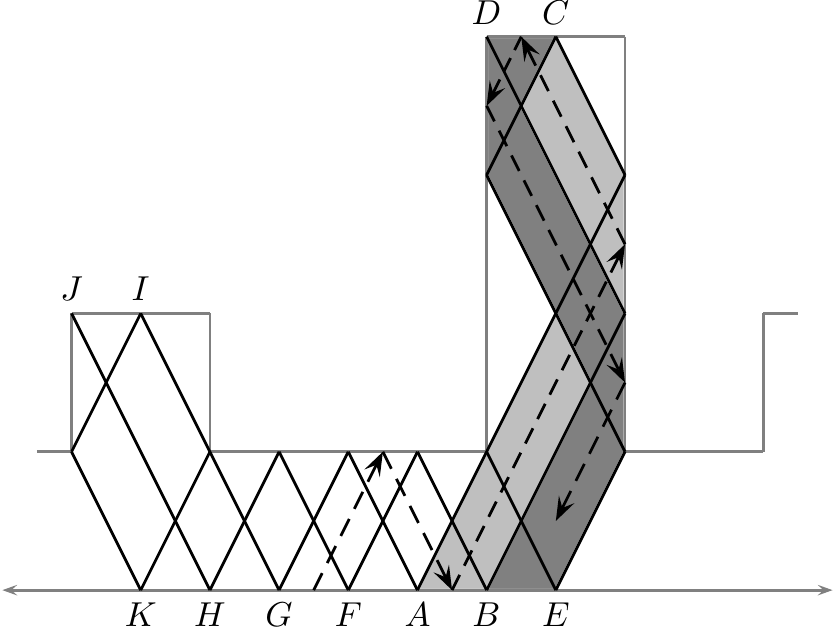}
\vspace{3pt}\\
\mbox{Figure 6.8.2: slope-$2$ street in the case $v_i=2$, $h_i=3$, $v_{i+1}=4$}
\end{array}
\end{displaymath}

Reversing the orientation, we have a reverse cycle in the same street, the slope-$(-2)$ version of the tilted street.

The only irrelevant change is that different horizontal streets may have different lengths, but the lengths are uniformly bounded. 
Hence we can reduce the billiard problem to a problem of geodesic flow on a rhombus-maze translation surface and then apply Theorem~\ref{thm6.7.1}.
This ultimately leads to an analog of Theorem~\ref{thm6.7.2}.

We shall show later that the same method works far beyond the special case \eqref{eq6.8.1}.
Consider, for instance, the $\{L_i\}_{i\in\Zz}$-strip region where there is an integer $r$ such that
\begin{equation}\label{eq6.8.2}
h_i\le r,
\quad
v_i\le r,
\quad
i\in\Zz,
\end{equation}
and
\begin{equation}\label{eq6.8.3}
\mbox{there are at most $r$ consecutive $v_i$, $i\in\Zz$, that are powers of $2$}.
\end{equation}

To reduce a billiard trajectory to geodesic flow on a rhombus-maze translation surface, the basic idea is essentially the same.
We use finite tilted streets, but the slope is not necessarily equal to~$2$. 
We shall show that under the conditions \eqref{eq6.8.2} and \eqref{eq6.8.3}, there always exists an integer $m$ so that we can construct a slope-$m$ street decomposition like in Figures 6.7.7 and~6.8.2.
Such a street decomposition is based on the following elementary number-theoretic lemma.
For any integer $k\ge2$, consider a $k$-tower of $k$ unit size square faces on top of each other in vertical position with two \textit{gates} 
that we call~$g_1$, the left gate, and~$g_2$, the right gate, as illustrated in Figure~6.8.3, which also shows the $4$ possible types of exits for a point billiard.

\begin{displaymath}
\begin{array}{c}
\includegraphics[scale=0.75]{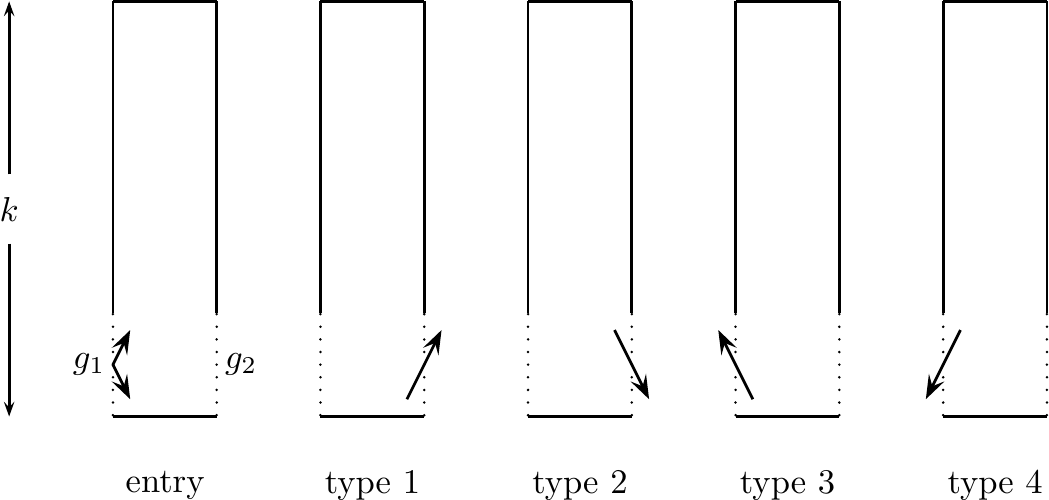}
\vspace{3pt}\\
\mbox{Figure 6.8.3: $k$-tower with two gates $g_1$ and $g_2$}
\end{array}
\end{displaymath}

\begin{lem}\label{lem6.8.1}
Let $k\ge2$ and $m$ with $\vert m\vert\ge2$ be two integers.
Suppose that a billiard enters a $k$-tower through the left gate $g_1$ with slope~$m$.
Let $x_0=x_0(k,m)\ge1$ denote the smallest positive integer such that
\begin{displaymath}
mx_0\equiv0\bmod{2k}
\quad\mbox{or}\quad
mx_0\equiv-1\bmod{2k}.
\end{displaymath}
If $m>0$, then
\begin{itemize}
\item[(i)] if $x_0$ is odd and $mx_0\equiv0\bmod{2k}$, then the billiard has exit type~$1$;
\item[(ii)] if $x_0$ is odd and $mx_0\equiv-1\bmod{2k}$, then the billiard has exit type~$2$;
\item[(iii)] if $x_0$ is even, then the billiard has exit type~$3$.
\end{itemize}
If $m<0$, then
\begin{itemize}
\item[(iv)] if $x_0$ is odd and $mx_0\equiv0\bmod{2k}$, then the billiard has exit type~$2$;
\item[(v)] if $x_0$ is odd and $mx_0\equiv-1\bmod{2k}$, then the billiard has exit type~$1$;
\item[(vi)] if $x_0$ is even, then the billiard has exit type~$4$.
\end{itemize}
\end{lem}

\begin{proof}
For convenience, assume that the billiard has unit vertical speed.

Suppose first of all that $m>0$.
It is clear that at time $mt$, $t=1,2,3,\ldots,$ following entry to the $k$-tower, the billiard hits a vertical side of the $k$-tower, on the right if $t$ is odd, and on the left if $t$ is even.
We want to trap the smallest integer value $t=x_0$ when the billiard hits the side of the bottom square face.

If it hits the side of the bottom square face before a final bounce off the bottom edge, then
\begin{displaymath}
mx_0+1\equiv0\bmod{2k},
\quad\mbox{so that}\quad
mx_0\equiv-1\bmod{2k}.
\end{displaymath}
If $x_0$ is odd, then it hits the right side, and (ii) follows.
Note that $x_0$ cannot be even in this case.

If it hits the side of the bottom square face after a final bounce off the bottom edge, then
\begin{displaymath}
mx_0\equiv0\bmod{2k}.
\end{displaymath}
If $x_0$ is odd, then it hits the right side, and (i) follows.
If $x_0$ is even, then it hits the left side, and (iii) follows.

Suppose next that $m<0$.
Then the billiard bounces off the bottom edge before it goes up the $k$-tower.
It is clear that at time $-mt$, $t=1,2,3,\ldots,$ the billiard hits the side of the $k$-tower, on the right if $t$ is odd, and on the left if $t$ is even.
We want to trap the smallest integer value $t=x_0$ when the billiard hits the side of the bottom square face.

If it hits the side of the bottom square face after a final bounce off the bottom edge, then
\begin{displaymath}
-mx_0-1\equiv0\bmod{2k},
\quad\mbox{so that}\quad
mx_0\equiv-1\bmod{2k}.
\end{displaymath}
If $x_0$ is odd, then it hits the right side, and (v) follows.
Note that $x_0$ cannot be even in this case.

If it hits the side of the bottom square face before a final bounce off the bottom edge, then
\begin{displaymath}
-mx_0\equiv0\bmod{2k},
\quad\mbox{so that}\quad
mx_0\equiv0\bmod{2k}.
\end{displaymath}
If $x_0$ is odd, then it hits the right side, and (iv) follows.
If $x_0$ is even, then it hits the left side, and (vi) follows.
\end{proof}

Using symmetry, we can deduce an analogous result when the billiard enters a $k$-tower through the right gate~$g_2$. 

Note that the case $m\equiv k\bmod{2k}$ is particularly simple.
Since $k\ge2$, we have $x_0=x_0(k,m)=2$, and the billiard bounces back, with exit type $3$ or~$4$.

Exit types $1$ and~$2$ represent \textit{transient states},
and exit types $3$ and~$4$ indicate that the billiard, having come from the left, \textit{bounces back} at this $k$-tower to the left.
To construct a finite tilted street, it is necessary, and also sufficient, that there are \textit{bounce backs on both sides} along the
$\{L_i\}_{i\in\Zz}$-strip, as illustrated in Figure~6.8.4.

\begin{displaymath}
\begin{array}{c}
\includegraphics[scale=0.75]{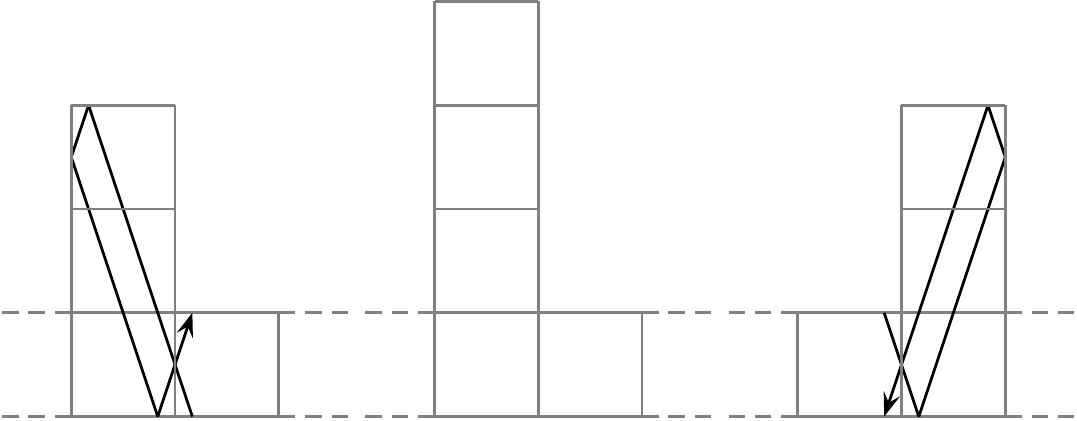}
\vspace{3pt}\\
\mbox{Figure 6.8.4: bounce backs on both sides}
\end{array}
\end{displaymath}

Next we apply Lemma~\ref{lem6.8.1} in some special cases of the $\{L_i\}_{i\in\Zz}$-strip billiard, where, for each $i\in\Zz$, $L_i$ is a
$(v_i,h_i)$-L-shape.
Note that every $1$-tower leads to a transient state, so the value of $h_i$ is irrelevant, and we only need to study $k$-towers for $k\ge2$.

\begin{case1}
Suppose that $h_i\ge2$ and $v_i=2$, $i\in\Zz$.
We have studied this case when $h_i=2$, $i\in\Zz$, in Section~\ref{sec6.7} with slopes $m=\pm2$, and the result extends to other values of~$h_i$.
Now we consider other integer values of~$m$.
To study the effect of the $2$-towers, we can use Lemma~\ref{lem6.8.1} with $k=2$, so that $2k=4$.
Suppose that a billiard enters a $2$-tower through the left gate with slope~$m$.

For $m\equiv2\bmod{4}$, we have $x_0=2$, so conclusion (iii) or (vi) of Lemma~\ref{lem6.8.1} applies.
This means that the billiard exits through the left gate, then later enters the next $2$-tower from the right gate with slope $\pm m\equiv2\bmod{4}$, and so exits through the right gate.
So the billiard path is bounded between these two $2$-towers.
Since the slope is an integer, the billiard can only hit the bottom edge of the $\{L_i\}_{i\in\Zz}$-strip between these two $2$-towers at finitely many points, and so must repeat.

For $m\equiv0,1,3\bmod{4}$, we have $x_0=1,3,1$ respectively, so conclusion (i), (ii), (iv) or (v) of Lemma~\ref{lem6.8.1} applies.
This means that the billiard exits through the right gate, then later enters the next $2$-tower from the left gate with slope $\pm m\equiv0,1,3\bmod{4}$, and so exits through the right gate, and so on, leading to an infinite street.
\end{case1}

\begin{case2}
Suppose that $h_i\ge2$ and $v_i=3$, $i\in\Zz$.
To study the effect of the $3$-towers, we can use Lemma~\ref{lem6.8.1} with $k=3$, so that $2k=6$.
Suppose that a billiard enters a $3$-tower through the left gate with slope~$m$.

For $m\equiv3\bmod{6}$, we have $x_0=2$, so conclusion (iii) or (vi) of Lemma~\ref{lem6.8.1} applies.
This means that the billiard exits through the left gate, then later enters the next $3$-tower from the right gate with slope $\pm m\equiv3\bmod{6}$, and so exits through the right gate.
So the billiard path is bounded between these two $3$-towers.
Since the slope is an integer, the billiard can only hit the bottom edge of the $\{L_i\}_{i\in\Zz}$-strip between these two $3$-towers at finitely many points, and so must repeat.

For $m\equiv0,1,2,4,5\bmod{6}$, we have $x_0=1,5,3,3,1$ respectively, so conclusion (i), (ii), (iv) or (v) of Lemma~\ref{lem6.8.1} applies.
This means that the billiard exits through the right gate, then later enters the next $3$-tower from the left gate with slope $\pm m\equiv0,1,2,4,5\bmod{6}$, and so exits through the right gate, and so on, leading to an infinite street.
\end{case2}

\begin{case3}
Suppose that $h_i\ge2$, $i\in\Zz$, and $v_i=2$ for every integer $i<0$ and $v_i=3$ for every integer $i\ge0$.
To achieve a bounce back at a $2$-tower, we note from Case 1 that the integer slope $m$ of the billiard must satisfy $m\equiv2\bmod{4}$, so that $m$ must be even.
On the other hand, to achieve a bounce back at a $3$-tower, we note from Case 2 that the integer slope $m$ of the billiard must satisfy $m\equiv3\bmod{6}$, so that $m$ must be odd.
So there must be infinite slope-$m$ streets.
\end{case3}

\begin{case4}
There exist integers $r,s,m\ge2$ such that $v_i\le r$, $h_i\le r$, $i\in\Zz$, and for every $i\in\Zz$, there exists $i\le j<i+s$ such that
$m\equiv v_j\bmod{2v_j}$.
In this case, we have $x_0=x_0(v_j,m)=2$.
This means that for any $2s$ successive~$L_i$, there are two bounce backs, guaranteeing finite slope-$m$ streets in between.
\end{case4}

Finally, we establish the following far reaching generalization of Theorem~\ref{thm6.7.2}.

\begin{thm}\label{thm6.8.1}
Consider an arbitrary $\{L_i\}_{i\in\Zz}$-strip region such that there exists an integer $r\ge2$ such that the conditions \eqref{eq6.8.2} and \eqref{eq6.8.3} hold.
Then there are infinitely many explicit quadratic irrational slopes such that every billiard trajectory in the $\{L_i\}_{i\in\Zz}$-strip region starting from a corner and having such a slope exhibits time-quantitative density.
\end{thm}

\begin{proof}
The conditions \eqref{eq6.8.2} and \eqref{eq6.8.3} imply that among $r+1$ consecutive $v_i$, there exists $v_j$ which is not a power of~$2$, so that $v_j$ has an odd prime factor~$p$.
We shall study the $v_j$-tower in~$L_j$, and show that this gives rise to a bounce back.
To ensure a bounce back, we must make sure that exit types $1$ and $2$ do not take place.
It follows from Lemma~\ref{lem6.8.1} that both conditions
\begin{equation}\label{eq6.8.4}
x_0\mbox{ is odd}
\quad\mbox{and}\quad
mx_0\equiv0\bmod{2v_j}
\end{equation}
and
\begin{equation}\label{eq6.8.5}
x_0\mbox{ is odd}
\quad\mbox{and}\quad
mx_0\equiv-1\bmod{2v_j}
\end{equation}
must fail.

To ensure that \eqref{eq6.8.4} fails, it is sufficient that (i) the multiplicity of $2$ in the prime factorization of $m$ is less than the multiplicity of $2$ in the prime factorization of~$2v_j$.
On the other hand, to ensure that \eqref{eq6.8.5} fails, it is sufficient that (ii) the numbers $m$ and $2v_j$ are not relatively prime.

Let $m=p$.
Then both $m$ and $2v_j$ are multiples of~$p$, so they are not relatively prime.
On the other hand, the multiplicity of $2$ in the prime factorization of $m$ is~$0$, while the multiplicity of $2$ in the prime factorization of $2k$ is at least~$1$.
So both (i) and (ii) hold, ensuring that both \eqref{eq6.8.4} and \eqref{eq6.8.5} fail.
\end{proof}

%%%%%%%%%%
%
% SECTION 6.9
%
%%%%%%%%%%

\subsection{Density on aperiodic surfaces with infinite streets (III)}\label{sec6.9}

We now return to the wind-tree billiard models discussed at the beginning of Section~\ref{sec6.7}; in particular, to the billiard models illustrated by Figures 6.7.1--6.7.5, where the side length of the square obstacles is equal to~$1/2$.
We shall use our study of the $\infty$-L-strip billiard to guide us to a better understanding of these more complicated wind-tree models.

Figures 6.7.4 and 6.7.5 illustrate that tilted streets of slopes $\pm1$ lead to finite streets for billiards in $f$-configurations.
Thus the corresponding problem of $1$-direction geodesic flow in the $4$-copy versions of these $f$-configurations can be viewed as a problem of geodesic flow in a rhombus-maze translation surface, where the rhombi are squares tilted at $45$ degrees.

Let us rescale appropriately so that all square obstacles have side length equal to~$1$.
Any $f$-configuration is made up of building blocks as illustrated in Figure~6.9.1.
Each building block contains $8$ L-shapes together with a square $S_{i,j}$ in the middle, or where the middle square $S_{i,j}$ is replaced by yet another $L$-shape
$L_{i,j}$.
The building block is surrounded by $16$ L-shapes from neighboring building blocks.

\begin{displaymath}
\begin{array}{c}
\includegraphics[scale=0.75]{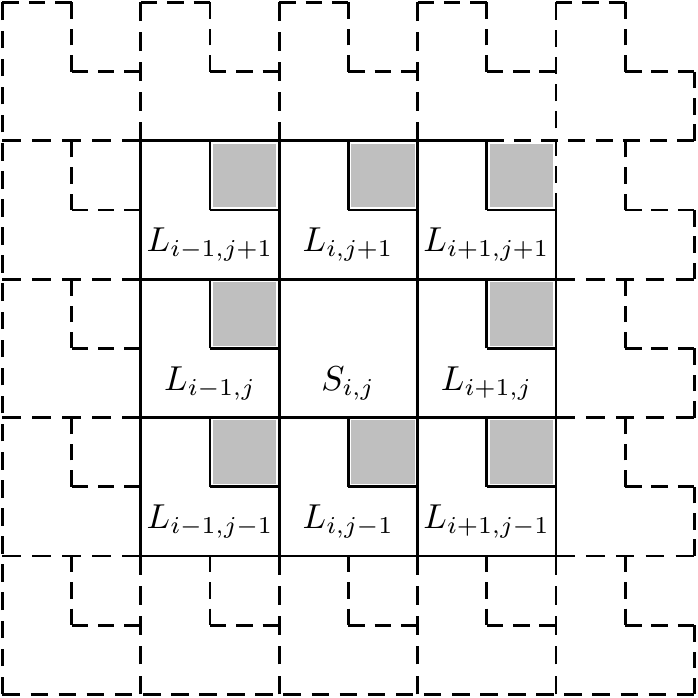}
\vspace{3pt}\\
\mbox{Figure 6.9.1: L-shapes with replacement square}
\\
\mbox{corresponding to missing obstacle}
\end{array}
\end{displaymath}

To construct the $4$-copy versions of these $f$-configurations, we construct the $4$-copy version of each of the L-shapes and squares.
As there is symmetry between the problem of horizontal edge pairing and the problem of vertical edge pairing, we shall only discuss the latter.

For neighboring L-shapes $L_{i,j}$ and $L_{i+1,j}$, the vertical edge identification in their $4$-copy versions, with reference of $j$ omitted, is exactly the same as in Figure~6.7.10.
For neighboring square $S_{i,j}$ and L-shape $L_{i+1,j}$, Figure~6.9.2 illustrates the situation, with reference to $j$ omitted.

\begin{displaymath}
\begin{array}{c}
\includegraphics[scale=0.75]{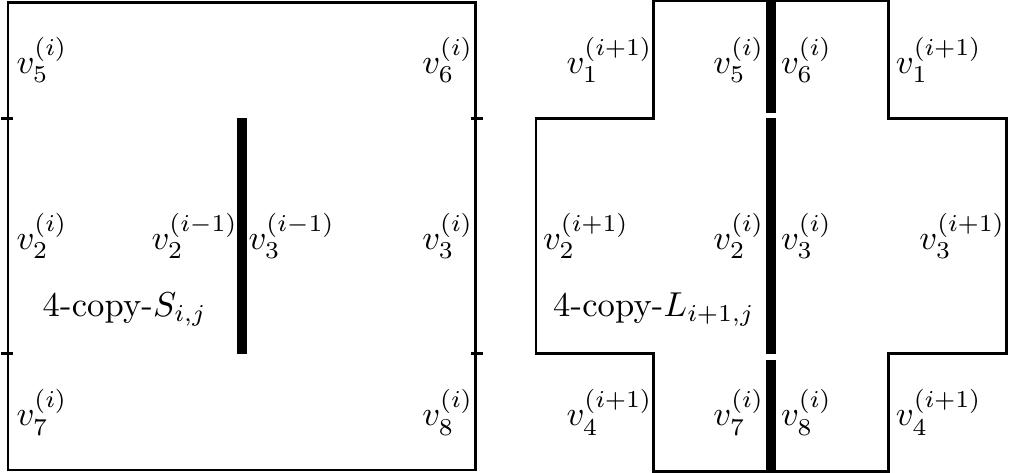}
\vspace{3pt}\\
\mbox{Figure 6.9.2: pairing of vertical edges in $4$-copy versions of $S_{i,j}$ and $L_{i+1,j}$}
\end{array}
\end{displaymath}

Viewed this way, these $f$-configurations share some of the features of the $\infty$-L-strip region.
For instance, we observe that the edge pairings on the double vertical edges in the middle of the $4$-copy versions of the square and the L-shape are precisely the same as in Figure~6.7.10.
The presence of the square $S_{i,j}$ instead of an L-shape $L_{i,j}$ leads to new vertical edge pairings in the $4$-copy versions of $S_{i,j}$ and $L_{i+1,j}$.
However, these new edge pairings do not involve any edges of the the $4$-copy versions of any other squares or L-shapes.

We thus establish the following generalization of Theorem~\ref{thm6.7.2}.

\begin{thm}\label{thm6.9.1}
There are infinitely many explicit quadratic irrational slopes such that every billiard trajectory in any $f$-configuration starting from a corner and having such an initial slope exhibits time-quantitative density on the $f$-configuration.
\end{thm}

Figure~6.9.3 below is an alternative to Figure~6.7.4, and shows a finite street of slope $1/3$ in the same billiard model.
Together, they show that the slopes $1$ and $1/3$ are completely periodic rational directions in the same periodic wind-tree model.
Here \textit{completely periodic} refers to the property that every billiard having this slope is periodic and so finite.

\begin{displaymath}
\begin{array}{c}
\includegraphics[scale=0.75]{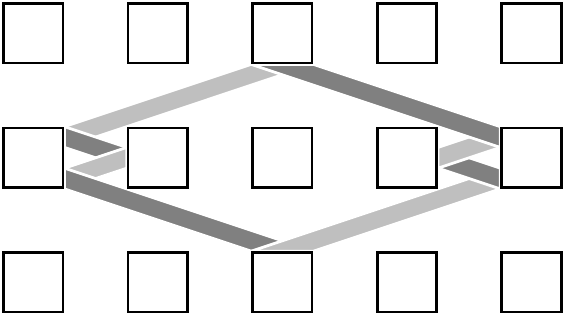}
\vspace{3pt}\\
\mbox{Figure 6.9.3: finite street of slope $1/3$}
\end{array}
\end{displaymath}

For the Ehrenfest periodic wind-tree billiard model with $a=b=3/4$, Figure~6.9.4 shows a finite periodic street of slope~$3$.

\begin{displaymath}
\begin{array}{c}
\includegraphics[scale=0.75]{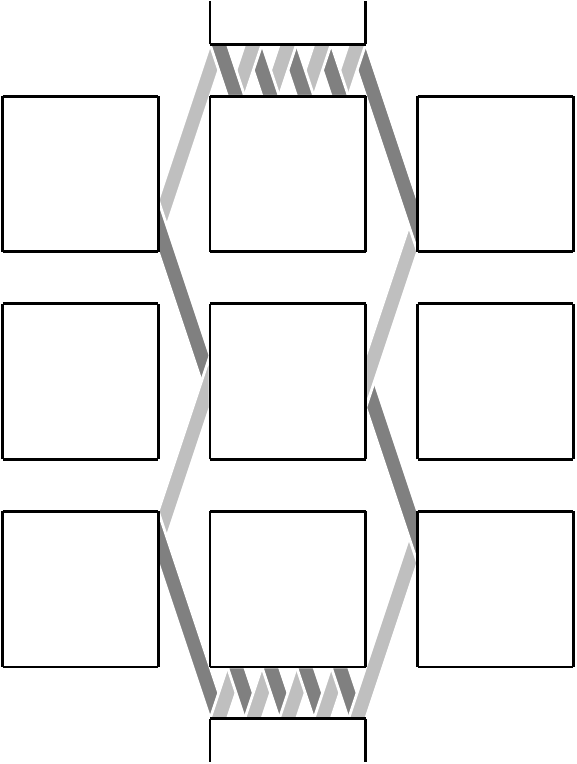}
\vspace{3pt}\\
\mbox{Figure 6.9.4: finite street of slope $3$ in wind-tree billiard model}
\\
\mbox{with $a=b=3/4$}
\end{array}
\end{displaymath}

Thus this billiard model can be reduced to a problem of geodesic flow in a rhombus-maze translation surface.
Figure~6.9.4 is an analog of Figure~6.7.4, and we leave it to the reader to draw an analog of Figure~6.7.5.
With that, it is fairly straightforward to prove a perfect analog of Theorem~\ref{thm6.9.1} in this new case.

We have the following results of Hubert, Leli\`{e}vre and Trubetzkoy~\cite{HLT} concerning complete periodicity.
The first result sheds more light on the collection of initial slopes in Theorem~\ref{thm6.9.1}.

\begin{lem}\label{lem6.9.1}
In the $2$-dimensional Ehrenfest periodic wind-tree billiard model with $a=b=1/2$, a rational slope~$k/\ell$, given in lowest terms, is completely periodic if and only if both $k$ and $\ell$ are odd.
\end{lem}

Since the set of rational slopes~$k/\ell$, given in lowest terms, with both $k$ and $\ell$ odd is dense on the unit circle, one can show that the collection of slopes in Theorem~\ref{thm6.9.1} gives rise to a dense set on the unit circle.

The second result leads to infinitely many instances of the Ehrenfest periodic wind-tree billiard model and their aperiodic $f$-configuration analogs for which one can establish analogs of Theorem~\ref{thm6.9.1}.

\begin{lem}\label{lem6.9.2}
Consider the 2-dimensional periodic wind-tree billiard model with rectangle size $a\times b$, and assume that both $0<a<1$ and $0<b<1$ are rational.
Suppose that $a=p/q$ and $b=r/s$ in lowest terms.

\emph{(i)} If both $p,r$ are odd and both $q,s$ are even, then there exists a completely periodic rational direction.

\emph{(ii)} If both $p,r$ are even and both $q,s$ are odd, then there is no completely periodic rational direction.
\end{lem}

In the case (i), the three authors give an explicit algorithm for finding these perfectly periodic rational directions, based on the work of McMullen~\cite{Mc} concerning the
so-called Weierstrass points on L-surfaces, a subject beyond the scope of our present paper.
Using these perfectly periodic rational directions, one can establish analogs of Theorem~\ref{thm6.9.1} for those Ehrenfest periodic wind-tree billiard models and their aperiodic $f$-configuration-like billiard models.

If $a=p/q$ and $b=r/s$, in lowest terms, are such that both $p,r$ are even and both $q,s$ are odd, then our method breaks down, in view of Lemma~\ref{lem6.9.2}(ii), and we are not able to prove density. 
This leads to the following interesting question.

\begin{opfour}
Is there an analog of Theorem~\ref{thm6.9.1} for the Ehrenfest periodic wind-tree billiard model or its aperiodic $f$-configuration analog for those values of $a=p/q$ and $b=r/s$, in lowest terms, such that $p,r$ are even and both $q,s$ are odd?
\end{opfour}

%%%%%%%%%%%%%%%%%%%%

Our primary interest in this paper is to study infinite aperiodic systems.
Nevertheless, for the sake of completeness, we conclude with a brief summary of what is known about the general double-periodic wind-tree model.
These models are \textit{recurrent} \textit{i.e.}, for almost every direction, the billiard returns arbitrarily close to every point of the infinite trajectory; see~\cite{AH,HLT}.

Note that recurrence does not imply density, but of course density implies recurrence.

On the other hand, Delecroix~\cite{D} has given an explicit set of initial slopes with positive Hausdorff dimension for which the orbit fails recurrence.
In fact, the orbit goes to infinity, \textit{i.e.}, the distance of the billiard from the starting point tends to infinity as the time $t\to\infty$.
What is more, these orbits are self-avoiding!

The double-periodic wind-tree models have an absence of ergodicity in almost every direction; see~\cite{FU}.
Moreover, these models exhibit surprisingly large \textit{super-random} escape rate to infinity for almost every direction. 
Actually, the super-random escape rate to infinity means order of magnitude $T^{2/3}$; see~\cite{DHL}.
More precisely, for a typical direction, a billiard orbit of length $T$ can go as far as $T^{2/3}$ from the starting point.
This is in sharp contrast to our results!
Indeed, whenever our shortline method proves density on an infinite surface, then the orbit exhibits 
much smaller escape rate $\log T$ to infinity.
See the Remark after the proof of Theorem~\ref{thm6.5.1} in Section~\ref{sec6.5}.

The escape rate $T^{2/3}$ is called \textit{super-random}, because the symmetric random walk has escape rate $T^{1/2}$ (square-root size fluctuation), so the escape rate is greater than random.

Perhaps a more standard terminology is \textit{super-diffusive behavior}.

%%%%%%%%%%
%
% REFERENCES
%
%%%%%%%%%%

\end{document}